\documentclass[11pt]{amsart}
\usepackage[margin=0.8in]{geometry}
\usepackage{amsmath,amsfonts,amssymb}
\usepackage[numbers,sort&compress]{natbib}
\usepackage[hidelinks]{hyperref}
\usepackage{graphicx}
\usepackage{mathrsfs,dsfont,tikz,subcaption,enumerate,comment}

\newenvironment{enumeratei}{\begin{enumerate}[\upshape i)]}{\end{enumerate}}
\newenvironment{enumeratea}{\begin{enumerate}[\upshape a)]\setlength{\itemsep}{1ex}}{\end{enumerate}}

\newenvironment{enumerateA}{\begin{enumerate}[\upshape (A)]\setlength{\itemsep}{1ex}}{\end{enumerate}}
\usepackage{paralist}

\newenvironment{inparaenuma}{\begin{inparaenum}[\upshape \bfseries (a) ]}{\end{inparaenum}}

\theoremstyle{plain}

\newtheorem{theorem}{Theorem}[section]
\newtheorem{lemma}[theorem]{Lemma}
\newtheorem{thm}{Theorem}[section]
\newtheorem{lem}[thm]{Lemma}
\newtheorem{cor}[thm]{Corollary}
\newtheorem{prop}[thm]{Proposition}
\newtheorem{defn}[thm]{Definition}

\newtheorem{ass}[thm]{Assumption}

\newtheorem{constr}[thm]{Construction}
\theoremstyle{definition}
\newtheorem{rem}[thm]{Remark}
\theoremstyle{remark}

\renewcommand{\le}{\leqslant} 
\renewcommand{\ge}{\geqslant} 
\renewcommand{\leq}{\leqslant} 
\renewcommand{\geq}{\geqslant}

\newcommand{\ind}{\mathds{1}}

\newcommand{\set}[1]{\left\{#1\right\}}

\newcommand{\eg}{\emph{e.g.,}}

\newcommand{\probc}{\stackrel{\mathrm{P}}{\longrightarrow}}

\def\qed{ \hfill $\blacksquare$}  
 \let\gb=\beta   
     \let\gl=\lambda

\newcommand{\cA}{\mathcal{A}}\newcommand{\cB}{\mathcal{B}}\newcommand{\cC}{\mathcal{C}}
\newcommand{\cE}{\mathcal{E}}
\newcommand{\cI}{\mathcal{I}}
\newcommand{\cL}{\mathcal{L}}
\newcommand{\cM}{\mathcal{M}}
\newcommand{\cP}{\mathcal{P}}
\newcommand{\cS}{\mathcal{S}}\newcommand{\cT}{\mathcal{T}}
\newcommand{\cV}{\mathcal{V}}
  
\newcommand{\vA}{\mathbf{A}}

\newcommand{\vP}{\mathbf{P}}\newcommand{\vQ}{\mathbf{Q}}

\newcommand{\ve}{\mathbf{e}}

\newcommand{\vp}{\mathbf{p}}
\newcommand{\vs}{\mathbf{s}}\newcommand{\vt}{\mathbf{t}}

\newcommand{\mvX}{\boldsymbol{X}}

\newcommand{\mvx}{\boldsymbol{x}}\newcommand{\mvy}{\boldsymbol{y}}

\newcommand{\mvzeta}{\boldsymbol{\zeta}}

\newcommand{\mvvarpi}{\boldsymbol{\varpi}}
   
\newcommand{\fF}{\mathfrak{F}}

\newcommand{\fJ}{\mathfrak{J}}

\newcommand{\fP}{\mathfrak{P}}

\newcommand{\fm}{\mathfrak{m}}

\newcommand{\bb}[1]{\mathbb{#1}}

\newcommand{\bL}{\mathbb{L}}
\newcommand{\bN}{\mathbb{N}}
\newcommand{\bR}{\mathbb{R}}
\newcommand{\bT}{\mathbb{T}}

\newcommand{\bZ}{\mathbb{Z}}        

\newcommand{\sE}{\mathscr{E}}

\DeclareMathOperator{\E}{\mathds{E}}
\DeclareMathOperator{\pr}{\mathds{P}}

\newcommand{\bbT}{\mathbb{T}}
\newcommand{\TT}{\mathcal{T}}
\newcommand{\bs}{\mathbf{s}}

\newcommand{\bt}{\mathbf{t}}
\newcommand{\bfomega}{{\boldsymbol \omega}}
\newcommand{\Zbold}{{\mathbb{Z}}}
\newcommand{\prob}{\mathbb{P}}

\newcommand{\probfr}{\stackrel{\mbox{$\operatorname{P}$-\bf fr}}{\longrightarrow}}
\newcommand{\probcrf}{\stackrel{\mbox{$\operatorname{P}$-\bf efr}}{\longrightarrow}}

\newcommand{\Efr}{\stackrel{\mbox{$\E$-\bf fr}}{\longrightarrow}}

\newcommand{\convas}{\stackrel{\mathrm{a.s.}}{\longrightarrow}}
\newcommand{\convd}{\stackrel{d}{\longrightarrow}}
\newcommand{\convp}{\stackrel{P}{\longrightarrow}}
\usepackage[foot]{amsaddr}

\newcommand{\stod}{\preceq_{\mathrm{st}}}
 \DeclareMathOperator{\BP}{BP}

\newcommand{\age}{{\sf Age} }

\newcommand{\Bern}{{\sf Bernoulli} }

\newcommand{\Poi}{{\sf Poisson} }
\newcommand{\Ma}{{\sf Mac} }
\newcommand{\nnd}{{\sf NoDel} }

\newcommand{\fringe}{{\sf fringe} }

\usepackage[mathscr]{euscript}
\DeclareMathAlphabet{\mathscrbf}{OMS}{mdugm}{b}{n}

\DeclareFontFamily{U}{BOONDOX-calo}{\skewchar\font=45 }
\DeclareFontShape{U}{BOONDOX-calo}{m}{n}{
  <-> s*[1.05] BOONDOX-r-calo}{}
\DeclareFontShape{U}{BOONDOX-calo}{b}{n}{
  <-> s*[1.05] BOONDOX-b-calo}{}
\DeclareMathAlphabet{\mathcalb}{U}{BOONDOX-calo}{m}{n}
\SetMathAlphabet{\mathcalb}{bold}{U}{BOONDOX-calo}{b}{n}
\DeclareMathAlphabet{\mathbcalb}{U}{BOONDOX-calo}{b}{n}

\newcommand{\sss}{\scriptscriptstyle}
\newcommand{\fpm}{\mvvarpi}
\newcommand{\fplus}[1]{\lfloor #1\rfloor_{_+}}
\usetikzlibrary{calc,intersections,through,backgrounds,shapes.geometric}
\usetikzlibrary{graphs}
\newcommand{\Cod}{\mathcalb{CoD}}
\newcommand{\Zma}{\mathcalb{Z}_{\Ma}}
\newcommand{\Mis}{\mathcalb{Mis}}
\usetikzlibrary{calc,intersections,through,backgrounds,shapes.geometric}
\usetikzlibrary{graphs}
\tikzset{every path/.style={line width=.07 cm}}

\begin{document}

\title[Network evolution with macroscopic delay]{Network evolution with Macroscopic Delays: asymptotics and condensation}

\author[Banerjee]{Sayan Banerjee$^1$}
\author[Bhamidi]{Shankar Bhamidi$^1$}
\author[Dey]{Partha Dey$^2$}
\author[Sakanaveeti]{Akshay Sakanaveeti$^1$}

\thanks{$^1$\textsc{Department of Statistics and Operations Research, 304 Hanes Hall, University of North Carolina, Chapel Hill, NC 27599}\\
$^2$\textsc{Department of Mathematics, University of Illinois, Urbana-Champaign, IL}\\
\textit{E-mail addresses}: \texttt{sayan@email.unc.edu}, \texttt{bhamidi@email.unc.edu}, \texttt{psdey@illinois.edu}, \texttt{sakshay@unc.edu}.}


\begin{abstract}
Preferential attachment models typically assume that each arriving vertex observes the current network before choosing its connection. Motivated by distributed systems and social networks, we study \emph{network delay}, where this decision uses only a time-delayed snapshot. We focus on \emph{macroscopic delays}, for which the delay is proportional to the current network size and hence removes a non-vanishing fraction of the available information. We identify the local weak limit as a continuous-time branching process whose reproduction point process has memory of its entire past. Since this non-Markovian description is difficult to analyze directly, we construct a dual branching process in which edges reproduce, recovering enough independence for quantitative analysis. This yields a detailed understanding of how the delay affects features such as the tail behavior of the asymptotic degree distribution, together with necessary and sufficient conditions for condensation—the phenomenon in which a positive fraction of the degree mass escapes to infinity. We conclude by studying the impact of the delay distribution on macroscopic functionals such as the root degree.
\end{abstract}

\maketitle
\tableofcontents

\section{Introduction}
\label{sec:intro}
Models of growing networks seek to explain macroscopic phenomena such as heavy-tailed degrees, information diffusion, sampling bias, and the recoverability of early vertices; see~\cite{albert2002statistical,newman2003structure,newman2010networks,bollobas2001random,durrett-rg-book,van2009random} and the references therein. Time-evolving networks~\cite{holme2012temporal,masuda2016guidance} arise in myriad applications such as social systems, where temporal dynamics affect diffusion and positional advantage~\cite{shah2012rumor,wagner2017sampling,espin2018towards,antunes2023learning}, and biological systems, where protein-interaction networks evolve through copying mechanisms~\cite{navlakha2011network,young2018network}. Their mathematical study has clarified mechanisms for heavy-tailed degree distributions~\cite{albert2002statistical,barabasi1999emergence}, seed reconstruction~\cite{bubeck-mossel,bubeck2017finding,banerjee2020persistence}, and bias in sampling and ranking algorithms~\cite{wagner2017sampling,karimi2022minorities,karimi2018homophily,antunes2023learning}.

A central feature of preferential attachment is its feedback mechanism: current degree determines the chance of receiving future connections. Most network evolution models assume that an arriving vertex observes the entire current network before choosing its connections. In distributed systems this information may be unavailable because of communication latency, privacy constraints, or synchronization bottlenecks. Limited information has motivated two complementary directions:
\begin{enumeratea}
\item connection rules based on local exploration from random locations~\cite{krapvisky2023magic,banerjee2022co};
\item \emph{network delay}, where an arriving vertex sees a delayed snapshot of the network.
\end{enumeratea}
We pursue the second direction. Delayed information is intrinsic to distributed ledgers and related systems~\cite{nakamoto2008bitcoin,popov2018tangle,muller2022tangle,li2019markov}. Relevant probabilistic work includes delay in queueing networks and its relation to unimodularity~\cite{baccelli2019renewal}, fluid limits for distributed ledgers~\cite{king2021fluid}, and the evolution of longest paths and one-endedness~\cite{dey2022asymptotic}. Related directed acyclic graph models appear in~\cite{computers12120257,xiao2024accelerating,ahmed2024optimization,monch2021dag}. This paper and its companion~\cite{BBDS04_meso} answer Question~6 of~\cite{dey2022asymptotic}.

Delay changes both the information set available to an arriving vertex and the reinforcement experienced by existing vertices. The main question is whether this stale information merely perturbs standard preferential attachment or produces a genuinely different large-network limit. The answer depends sharply on the scale of the delay.

We next define the model in the tree setting studied here.

\subsection{Network evolution with delay}
\label{sec:net-mod}
The model has three ingredients:
\begin{enumeratea}
\item a \emph{time-scale parameter} $\gb \in [0,1]$. We study the macroscopic regime $\gb=1$; the companion paper~\cite{BBDS04_meso} treats the mesoscopic regime $\gb<1$;
\item a \emph{delay distribution} $\mu$ on $(\bR_+,\cB(\bR_+))$, which for $\gb=1$ may be assumed to be supported on $[0,1]$;
\item a strictly positive \emph{attachment function} $f:\set{0,1,\ldots}\to\bR_+$. Here $f(k)=k+\alpha$ with $\alpha\geq0$.
\end{enumeratea}

{\begin{defn}[Network evolution with delay]
\label{defn:model}We grow a sequence of random trees $\set{\cT(n):n\geq 0}$ as follows:  

\begin{enumeratei}
	\item Initialize the sequence as $\cT(0)=\cT(1) = \{v_1\}$. At time $2$, the tree $\cT(2)$ consists of two vertices, labeled as $\set{v_1, v_2}$ and attached by a single edge directed from $v_2$ to $v_1$. Call the vertex $v_1$ the ``root'' of the tree. 
    
	\item Suppose the network has been constructed through time $n\geq2$, with vertices $\{v_1,\dots,v_n\}$ and edges directed from children to parents. For $v_i$ and $j\geq i$, let $\deg(v_i,j)$ denote its degree at time $j$ (in-degree plus one except at the root). Set $\deg(v_i,i)=1$ and $\deg(v_i,j)=0$ for $j<i$.

    \item  At time  $n+1$, a new vertex $v_{n+1}$ enters the system. A \emph{normalized} time delay $\xi_{n+1} \sim \mu$ independent of $\cT(n)$ is sampled. 
Conditional on $\cT(\fplus{n - n^{\gb}\xi_{n+1}})$, this new vertex attaches to a vertex $v\in \cT(\fplus{n - n^{\gb}\xi_{n+1}})$ with probability proportional to 
	\[
        f(\deg(v,\fplus{n - n^{\gb}\xi_{n+1}})).
    \]
\end{enumeratei}

Denote by $\cL(\gb, \mu, f)$ the corresponding probability distribution of the sequence of growing random trees.  
\end{defn}}
 
Thus $v_{n+1}$ sees only the vertex set and degrees in $\cT(\fplus{n-n^\gb\xi_{n+1}})$, not the state at time $n$. Delay therefore privileges older vertices, in contrast to \emph{aging} models in which older vertices lose attractiveness~\cite{garavaglia2017dynamics,medo2011temporal,hajra2005aging,csardi2006dynamics,zhu2003effect}.

This distinction is important: a delayed vertex can attach only to vertices already present in its snapshot and evaluates their attractiveness using stale degrees. Thus delay favors early vertices through an information constraint rather than by modifying the attachment function itself.

The chosen vertex is the parent of $v_{n+1}$, producing the rooted tree $\cT(n+1)$ on $n+1$ labeled vertices $\{v_1,\dots,v_{n+1}\}$ rooted at $v_1$. To reduce notational overhead, we henceforth suppress integer parts $\lfloor \cdot \rfloor$. We focus on trees and defer extensions and open problems to Section~\ref{sec:conc}.

Our main regime is $\gb=1$, where the delay $n\xi_{n+1}$ is of the same order as the network size and $1-\xi_{n+1}$ represents the fraction of the network visible to the arriving vertex. In particular, $\xi_{n+1}=0$ gives complete current information, whereas $\xi_{n+1}=1$ exposes only the root. The mesoscopic regime $\gb<1$ with general $f$ is treated in~\cite{BBDS04_meso}.

\subsection{Rationale for the paper and overview of our results}
Formal statements appear in the next three sections. The main contributions are:
\begin{enumeratea}
\item \textbf{Local limits, memory, and duality.}
Unlike the mesoscopic regime~\cite{BBDS04_meso}, where the local limit agrees with the no-delay model under weak assumptions, macroscopic delay changes the limit object. Theorem~\ref{thm:macroscopic-local} identifies it as a continuous-time branching process driven by non-Markovian reproduction point processes whose inter-arrival times depend on their entire history; see Sections~\ref{intuition} and~\ref{sec:lim-obj-macro}. The memory is unavoidable: the chance that a vertex reproduces depends on the degrees visible across all earlier delayed snapshots. The theorem gives expected fringe convergence even in the condensation regime and, when $\mu(\set{1})=0$, convergence in probability to an infinite {\tt sin}-tree. The latter also yields limits for global functionals such as the empirical spectral distribution.

Direct quantitative analysis of this non-Markovian representation is intractable because its inter-arrival times are recursively dependent. We therefore construct a dual process in which \emph{edges reproduce} (Section~\ref{edgedesc}). This representation exposes independent Poisson reproduction mechanisms and makes classical branching-process tools available for studying degree tails and condensation.

\item \textbf{Degree exponents and escape of mass.}
Delay suggests two competing heuristics for the degree distribution. Since early vertices are favored, mass may escape from the bulk and leave a lighter limiting tail. On the other hand, among vertices that remain visible, stale degrees may amplify reinforcement and produce a heavier tail. We show that the atom at maximal delay separates these behaviors. If $\mu(\set{1})>0$, the root attracts a positive density of vertices, condensation occurs, and the limiting degree distribution is strictly lighter-tailed than without delay (Corollary~\ref{cor:special-case-macro} and Remark~\ref{cond}). If instead $\mu(\set{1})=0$ and $\mu(\set{0})<1$, no mass escapes to the root and Theorem~\ref{thm:non-conden} shows that the limiting degree distribution is strictly heavier-tailed than in the no-delay model. Thus the effect of macroscopic delay is not a uniform shift of the classical exponent but a qualitative dichotomy governed by the endpoint behavior of $\mu$.

\item \textbf{Global functionals.}
Theorem~\ref{thm:macro-max-deg} gives the scaling of the root degree and illustrates a route from local to global asymptotics. Such a passage is not automatic because local weak convergence does not control an exceptional early vertex. The proof couples the global dynamics to the limiting edge branching process and controls the approximation error using renewal theory. This provides a template for other global statistics driven by early vertices.
	
\end{enumeratea}

Figures~\ref{fig:sim} and~\ref{fig:deg} illustrate how the delay distribution changes network geometry and compare simulated degree tails with the exponents predicted by Corollary~\ref{cor:degree-tail-macro-non-cond}.

\begin{figure}[htbp]
    \centering
    \begin{subfigure}[b]{0.49\textwidth}
    \includegraphics[width=1.0\linewidth]{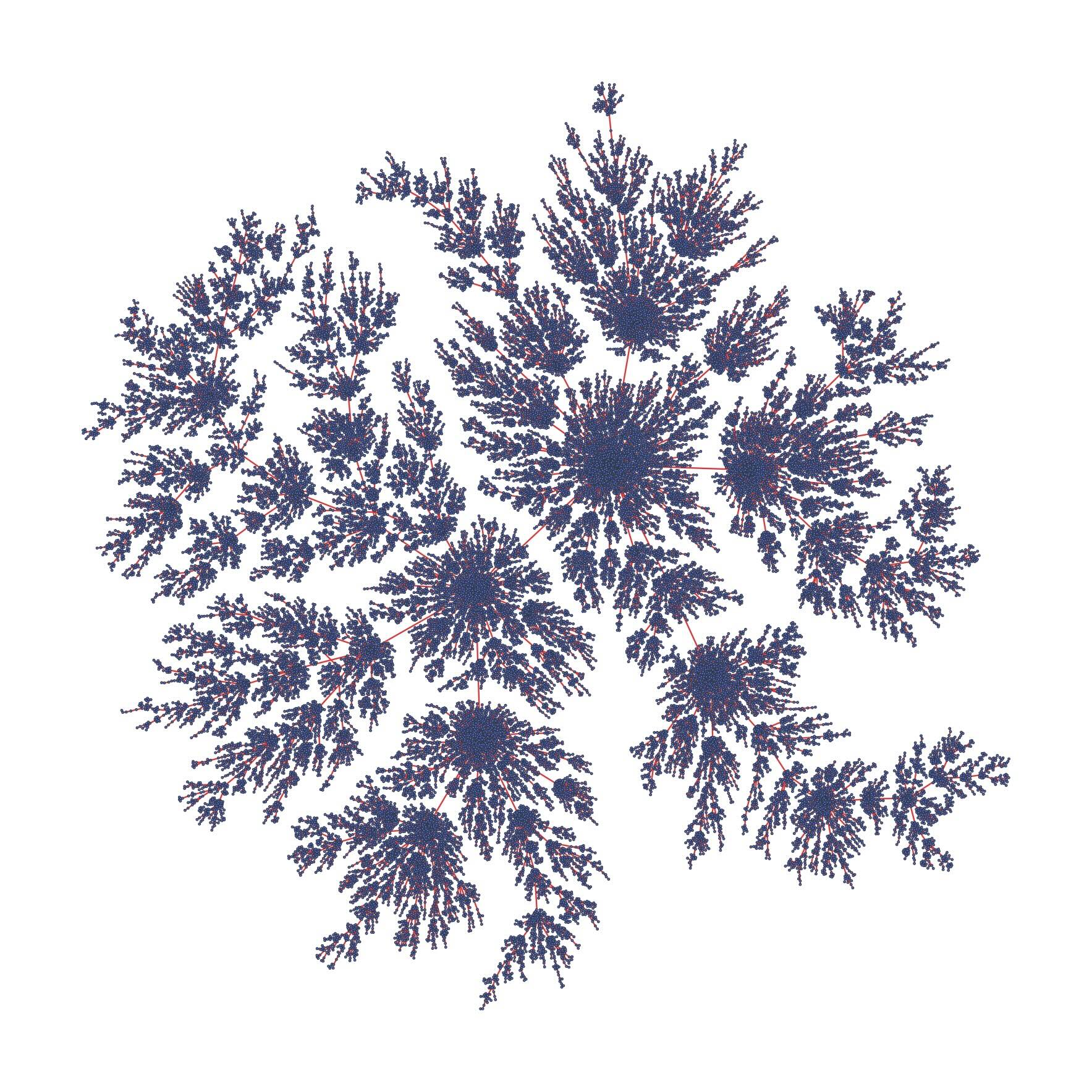}
    \caption{$\xi\equiv0$ (no delay)}
    \end{subfigure}
    \begin{subfigure}[b]{0.49\textwidth}
    \includegraphics[width=1.0\linewidth]{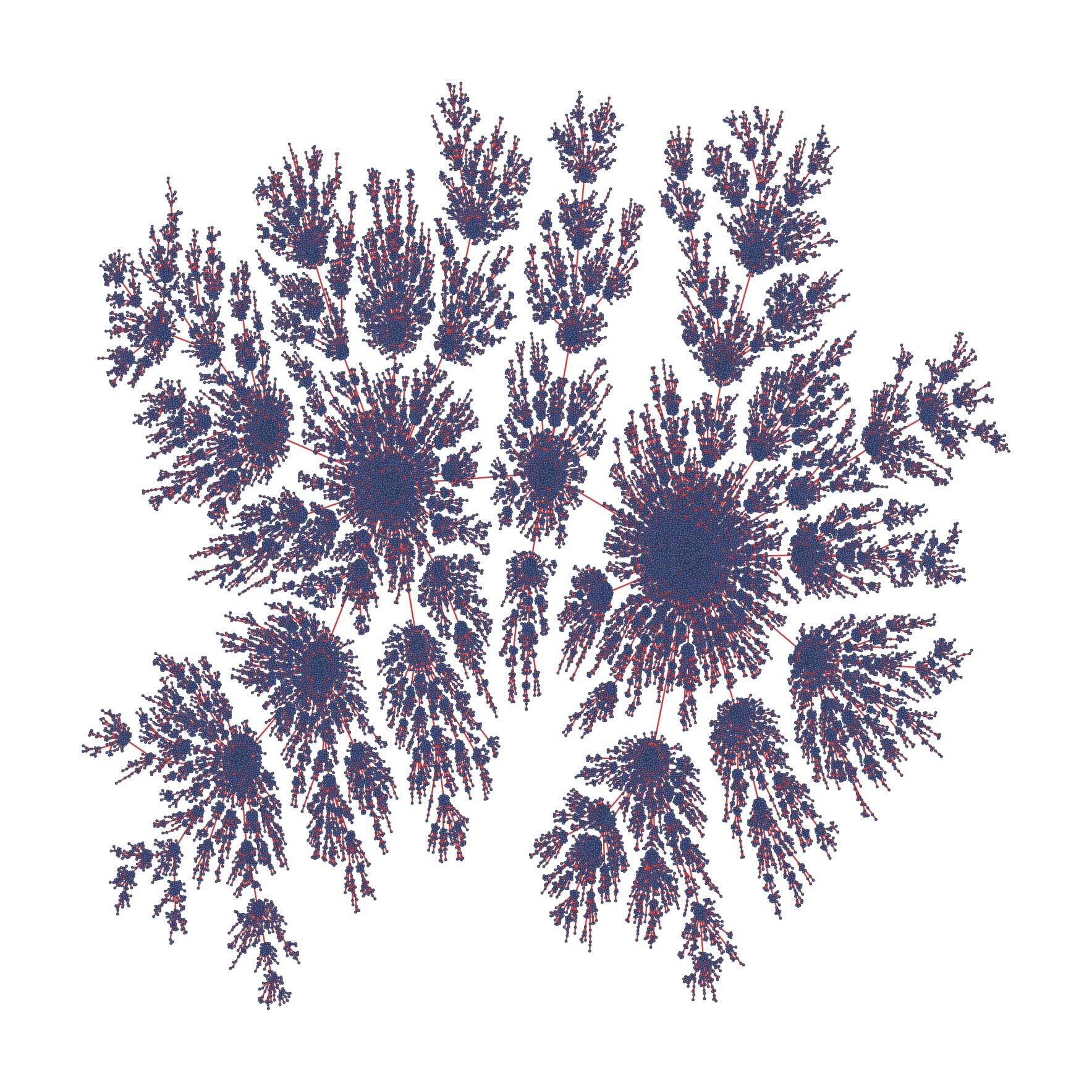}
    \caption{$\xi\sim 1-$Uniform$(0,1)^{1/2}$}
    \end{subfigure}
    \begin{subfigure}[b]{0.49\textwidth}
    \includegraphics[width=1.0\linewidth]{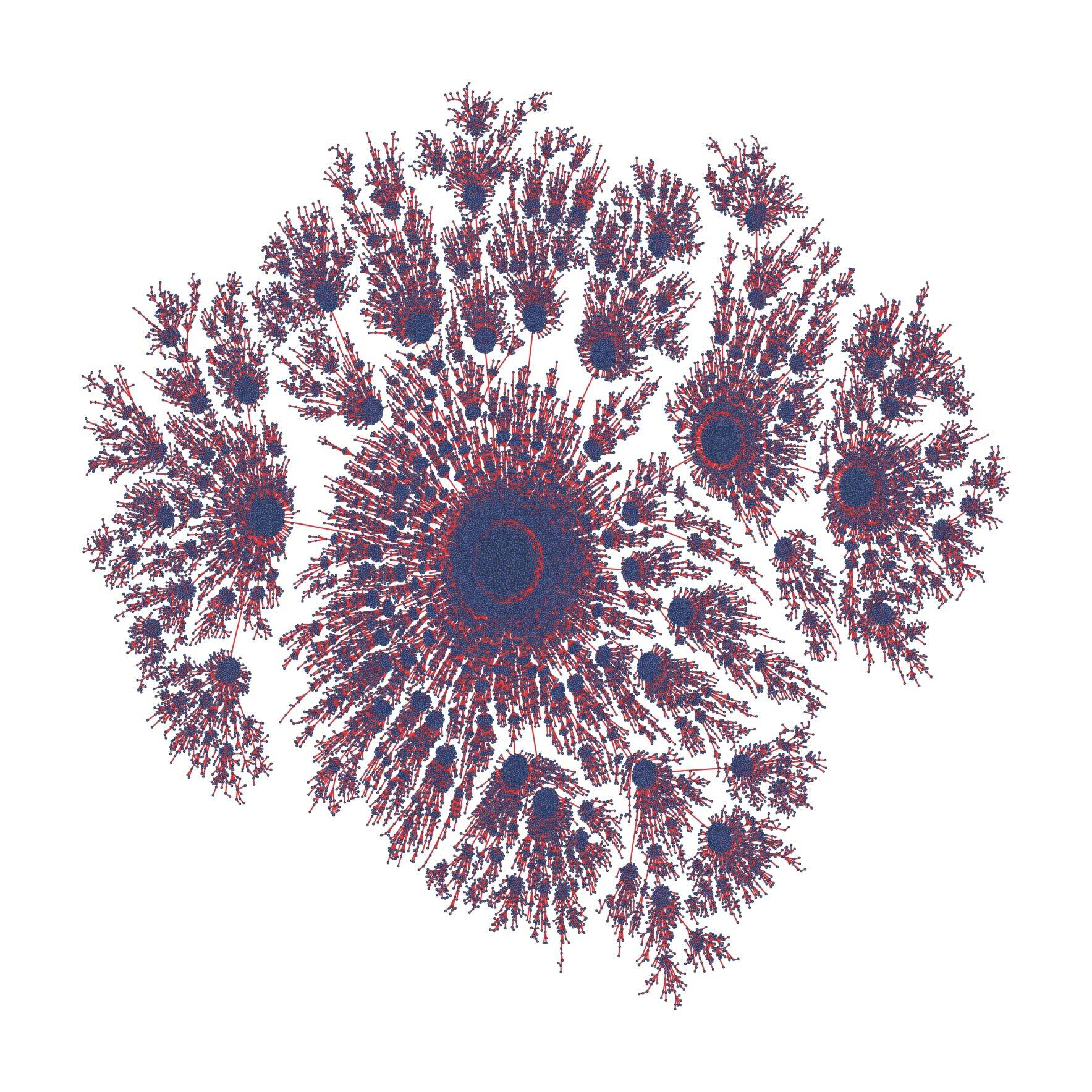}
    \caption{$\xi\sim 1-$Uniform$(0,1)$}
    \end{subfigure}
    \begin{subfigure}[b]{0.49\textwidth}
    \includegraphics[width=1.0\linewidth]{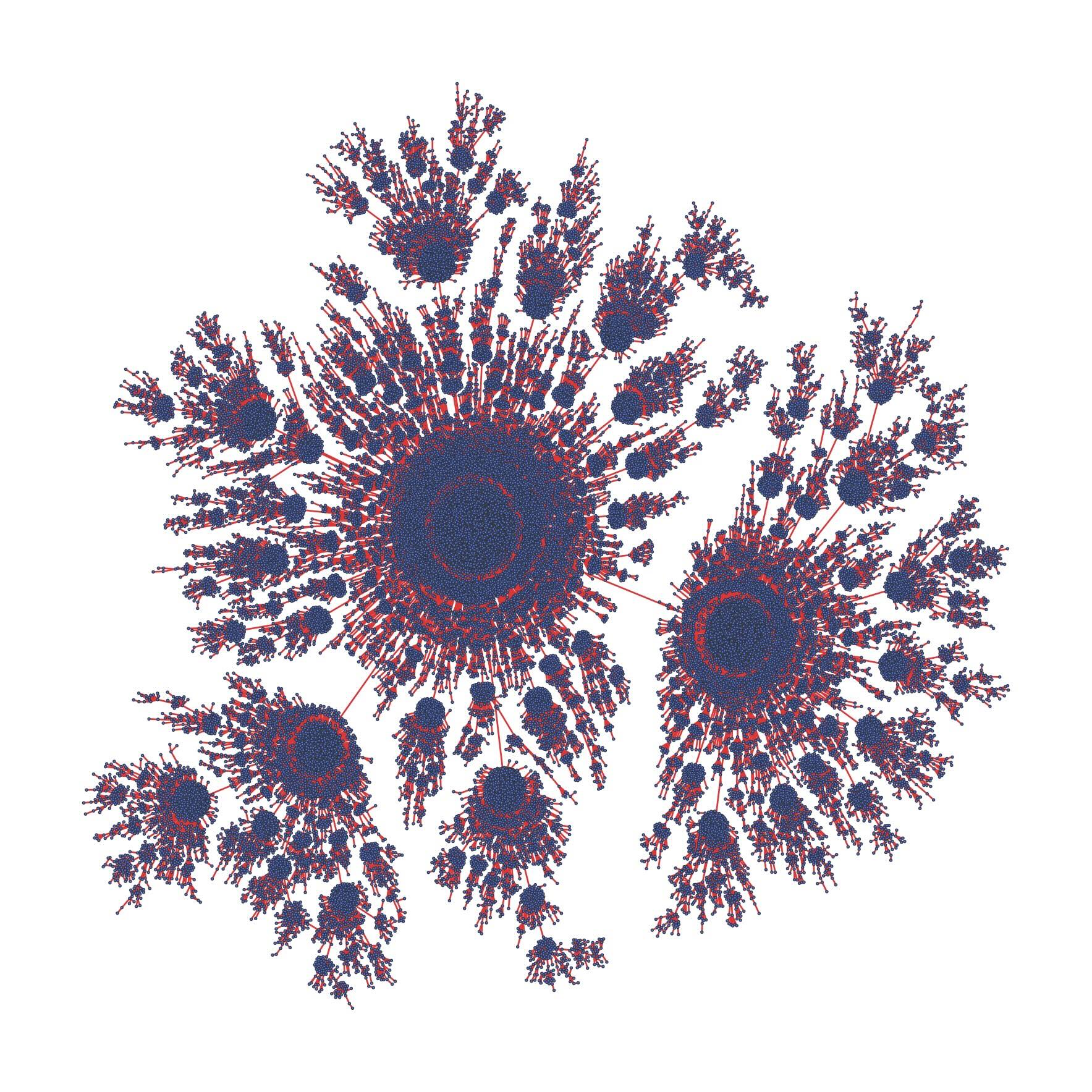}
    \caption{$\xi\sim 1-$Uniform$(0,1)^{2}$}
    \end{subfigure}
    \caption{Simulated Linear Preferential Attachment trees on 50000 nodes with different macroscopic delay distributions.}
    \label{fig:sim}
\end{figure}

\begin{figure}[htbp]
    \centering
    \begin{subfigure}[b]{.49\textwidth}
    \includegraphics[trim={18cm 0 0 1cm},clip,width=\linewidth]{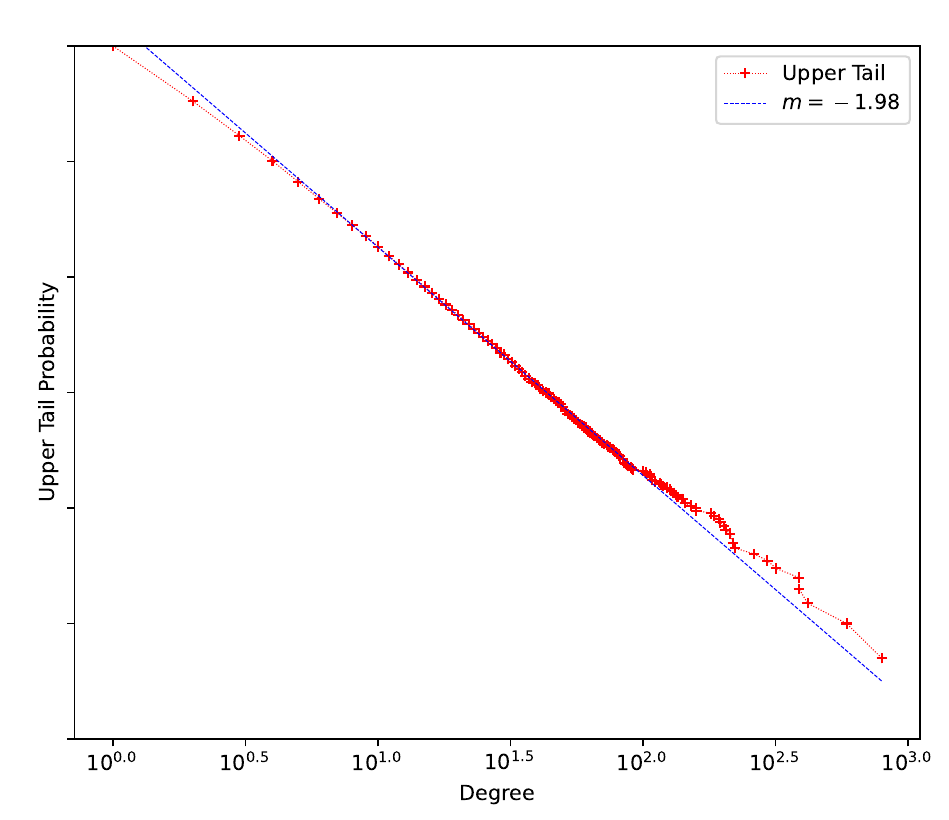}
    \caption{$\xi\equiv0$}
    \end{subfigure}
    \begin{subfigure}[b]{.49\textwidth}
    \includegraphics[trim={18cm 0 0 1cm},clip,width=\linewidth]{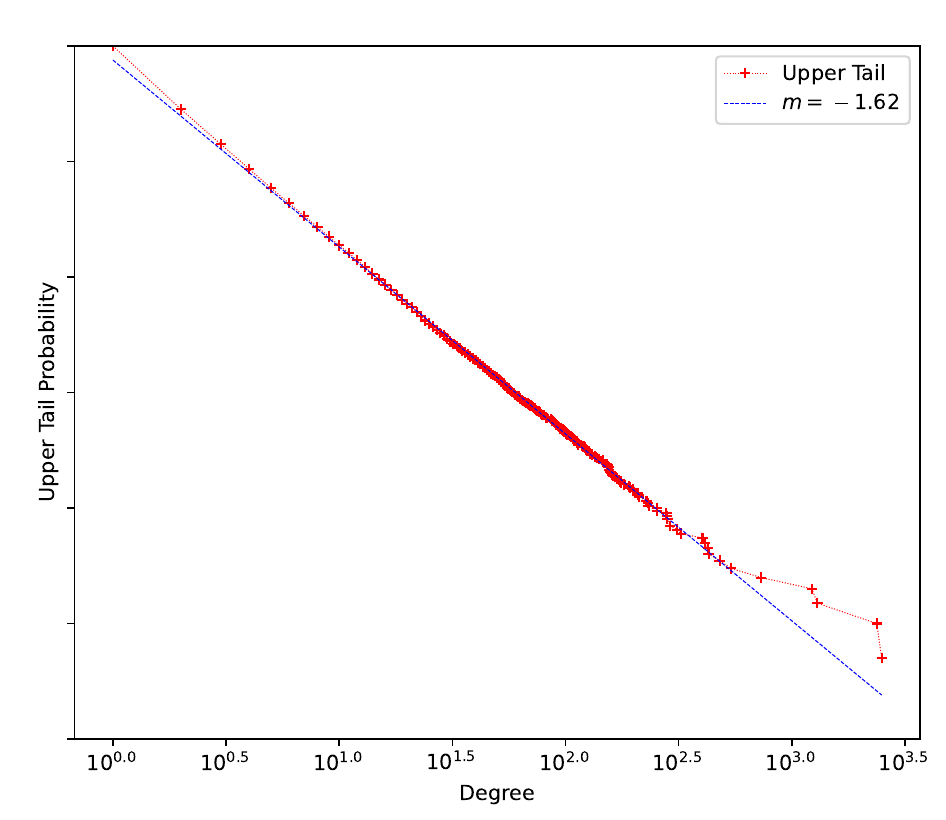}
    \caption{$\xi\sim 1-$Uniform$(0,1)^{1/2}$}
    \end{subfigure}
    \begin{subfigure}[b]{.49\textwidth}
    \includegraphics[trim={18cm 0 0 1cm},clip,width=\linewidth]{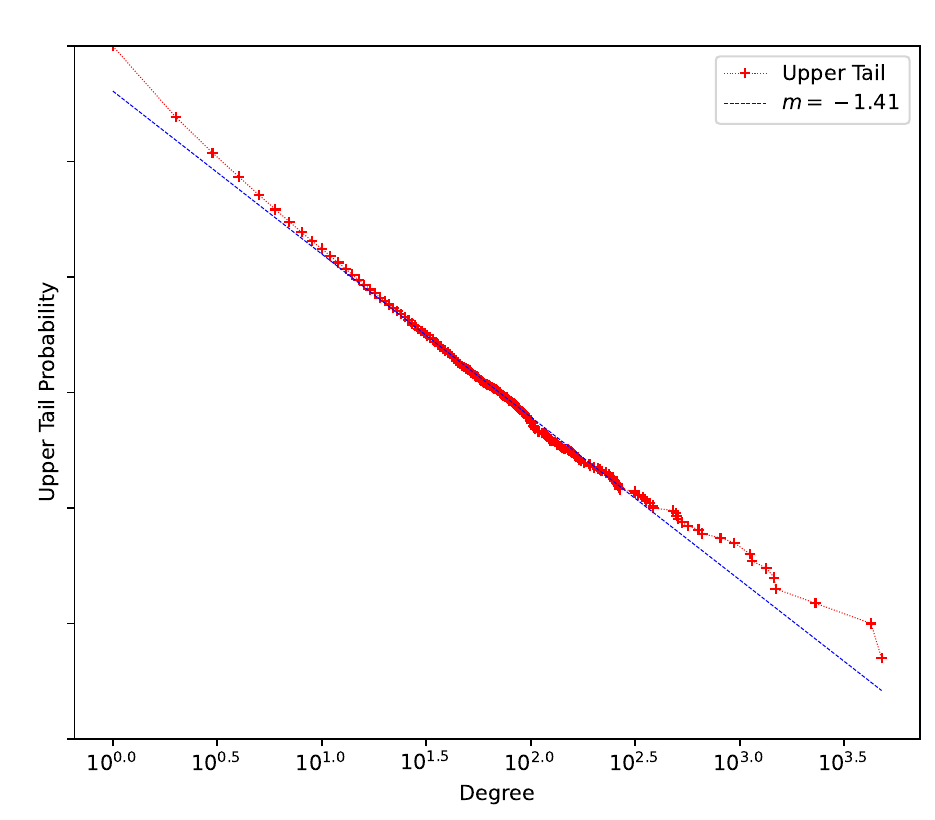}
    \caption{$\xi\sim $Uniform$(0,1)$}
    \end{subfigure}
    \begin{subfigure}[b]{.49\textwidth}
    \includegraphics[trim={18cm 0 0 1cm},clip,width=\linewidth]{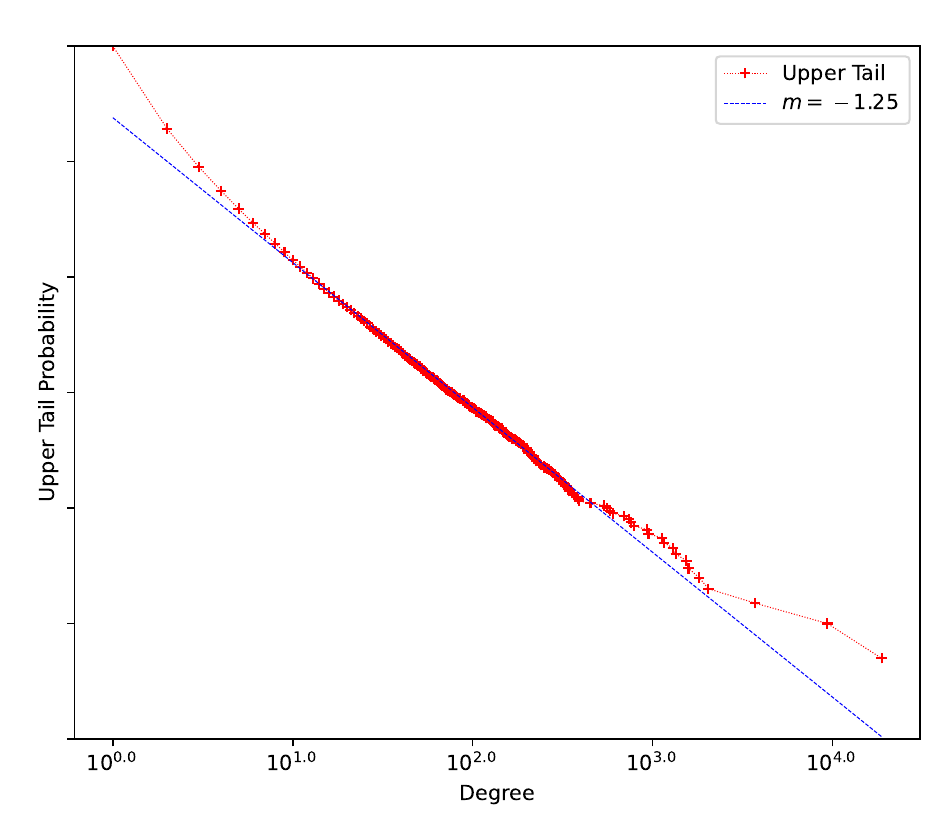} 
    \caption{$\xi\sim 1-$Uniform$(0,1)^{2}$}
    \end{subfigure}
    \caption{Tail exponent of degree distribution for Preferential Attachment trees on 200000 nodes with different delay distributions along with the predicted tail exponent. See Cor.~\ref{cor:degree-tail-macro-non-cond}(b) for a proof of the tail exponent $2/(1-\theta+\sqrt{1+\theta^2})$ when $\xi\sim 1-\text{Uniform}(0,1)^{1/\theta}$.}
    \label{fig:deg}
\end{figure}

\subsection{Organization of the paper}
{ Section~\ref{sec:local-wll} sets up basic notation { for local weak convergence specialized to the setting of trees}. Section~\ref{limobj} describes probabilistic constructions required to describe the local weak limit. Section~\ref{sec:main-res} contains statements of the main results and a brief overview of related work.   Section~\ref{sec:proofs-macro} contains proofs of the main results. We conclude with open problems and potential extensions in Section~\ref{sec:conc}.}

\section{Local weak convergence and {\tt sin}-trees}
\label{sec:local-wll}
\subsection{Overview}
{ Local weak convergence~\cite{aldous-steele-obj,benjamini-schramm,van2023random} has emerged as one standard tool-kit to understand asymptotics of large discrete random structures. As opposed to the general (graph) context, when the underlying sequence of discrete structures are rooted unordered trees, the treatment becomes significantly simpler via orienting the tree using paths originating from the root. We largely follow Aldous~\cite{aldous-fringe}, which developed the foundations for local weak limits for probabilistic models of trees, and defined a {\bf stronger notion} of local weak convergence in terms of convergence {\bf in probability} of fringe distribution of the sequence of random trees to corresponding objects for infinite trees with a single infinite path to infinity (sometimes referred to as {\tt sin}-trees).  The rationale for using the notion of fringe convergence for the rest of the paper are:
\begin{enumeratea}
    \item Fix a vertex at random and consider the sub-tree of descendants of this vertex (called the fringe at this vertex). Then under general conditions on the delay distribution, {\bf including the condensation regime when $\mu(\set{1}) >0$} our results imply convergence to an explicit limit finite random tree distribution described in the next section. 
    \item When the delay distribution has positive probability of attaching to the root ($\mu(\set{1})>0$), one cannot have weak convergence of local neighborhoods since a  positive density of vertices are connected to the root. When $\mu(\set{1}) = 0$, the theory in this section will enable us to show that the fringe distribution convergence in (a) automatically implies that,  in this setting, one {\bf does have} local weak convergence (in fact convergence in probability) to a limit infinite random object.   This automatically allows one to read off laws of large numbers for local functionals as well as global functionals such as the spectral distribution of the adjacency matrix. 
\end{enumeratea}
}


\subsection{Mathematical notation}
We use $\stod$ for stochastic domination between two real-valued probability measures. For $J\geq 1$ let $[J]:= \set{1,2,\ldots, J}$. If $Y$ has an exponential distribution with rate $\gl$, write this as $Y\sim \exp(\gl)$. Write $\bZ$ for the set of integers, $\bR$ for the real line, $\bN$ for the natural numbers and let $\bR_+:=(0,\infty)$. Write $\convas,\convp,\convd$ for convergence almost everywhere, in probability and in distribution, respectively. For a non-negative function $n\mapsto g(n)$,
we write $f(n)=O(g(n))$ when $|f(n)|/g(n)$ is uniformly bounded, and
$f(n)=o(g(n))$ when $\lim_{n\rightarrow \infty} f(n)/g(n)=0$.
Furthermore, write $f(n)=\Theta(g(n))$ if $f(n)=O(g(n))$ and $g(n)=O(f(n))$.
We write that a sequence of events $(A_n)_{n\geq 1}$
occurs \emph{with high probability} (whp) when $\pr(A_n)\rightarrow 1$ as $n \rightarrow \infty$. One of the core objects of this paper is the study of a sequence of growing random trees $\set{\cT_n:n\geq 1}$. Throughout we will write $\deg(v,\cT_n)$ for the degree of the vertex $v$ in tree $\cT_n$ and write for the empirical degree counts,  
\begin{equation}
    \label{eqn:deg-count}
    N_k(n) = \sum_{v\in \cT_n} \ind\set{\deg(v,\cT_n) = k }, \qquad k\geq 1.
\end{equation}

\subsection{Fringe decomposition for trees}
For $n\geq 1$, let $ \bT_{n} $ be the space of all rooted { unordered} trees on  $n$ {vertices}. Let $ \bbT =
\cup_{n=0}^\infty \bT_{n} $ be the space of all finite rooted trees.  Here $\bT_{0} = \emptyset $ will be used to represent the empty tree (tree on zero vertices). Let $\rho_{\bt}$ denote the root of $\bt$.  For any $r\geq 0$ and $\bt\in \bbT$, let $B(\bt, r) \in \bbT$ denote the subgraph of $\bt$ of vertices within graph distance $r$ from $\rho_{\bt}$, viewed as an element of $\bbT$ and rooted again at $\rho_{\bt}$. 

 Given two rooted finite trees $\bs, \bt \in \bbT$,  say that $\bs \simeq \bt $ if  there exists a {\bf root
preserving} isomorphism between the two trees viewed as unlabelled graphs. Given two rooted trees $\bt,\bs \in { \bbT}$ (\cite{benjamini-schramm},~\cite[Equation 2.3.15]{van2023random}), define the distance 
\begin{align}
\label{eqn:distance-trees}
	d_{\bbT}(\bt,\bs):= \frac{1}{1+R^*}, \qquad \text{with } \qquad R^* =\sup\{r: B(\bt, r) \simeq B(\bs, r)
	\}.
\end{align}

 Next, fix a tree $\bt\in \bbT$ with root $\rho = \rho_\bt$ and a vertex $v\in \bt$ at (graph) distance $h$ from the root.  Let $(v_0 =v, v_1, v_2, \ldots, v_h = \rho)$ be the unique path from $v$ to $\rho$. The tree $\bt$  can be decomposed as $h+1$ rooted trees $f_0(v,\bt), \ldots, f_h(v,\bt)$, where $f_0(v,\bt)$ is the tree rooted at $v$, consisting of all vertices for which there exists a path from the root passing through $v$. For $i \ge 1$, $f_i(v,\bt)$ is the subtree rooted at $v_i$, consisting of all vertices for which the path from the root passes through $v_i$ but not through $v_{i-1}$, see Figure~\ref{fig:fringe}. 

 \begin{figure}[htbp]
\centering
\includegraphics[scale=.2]{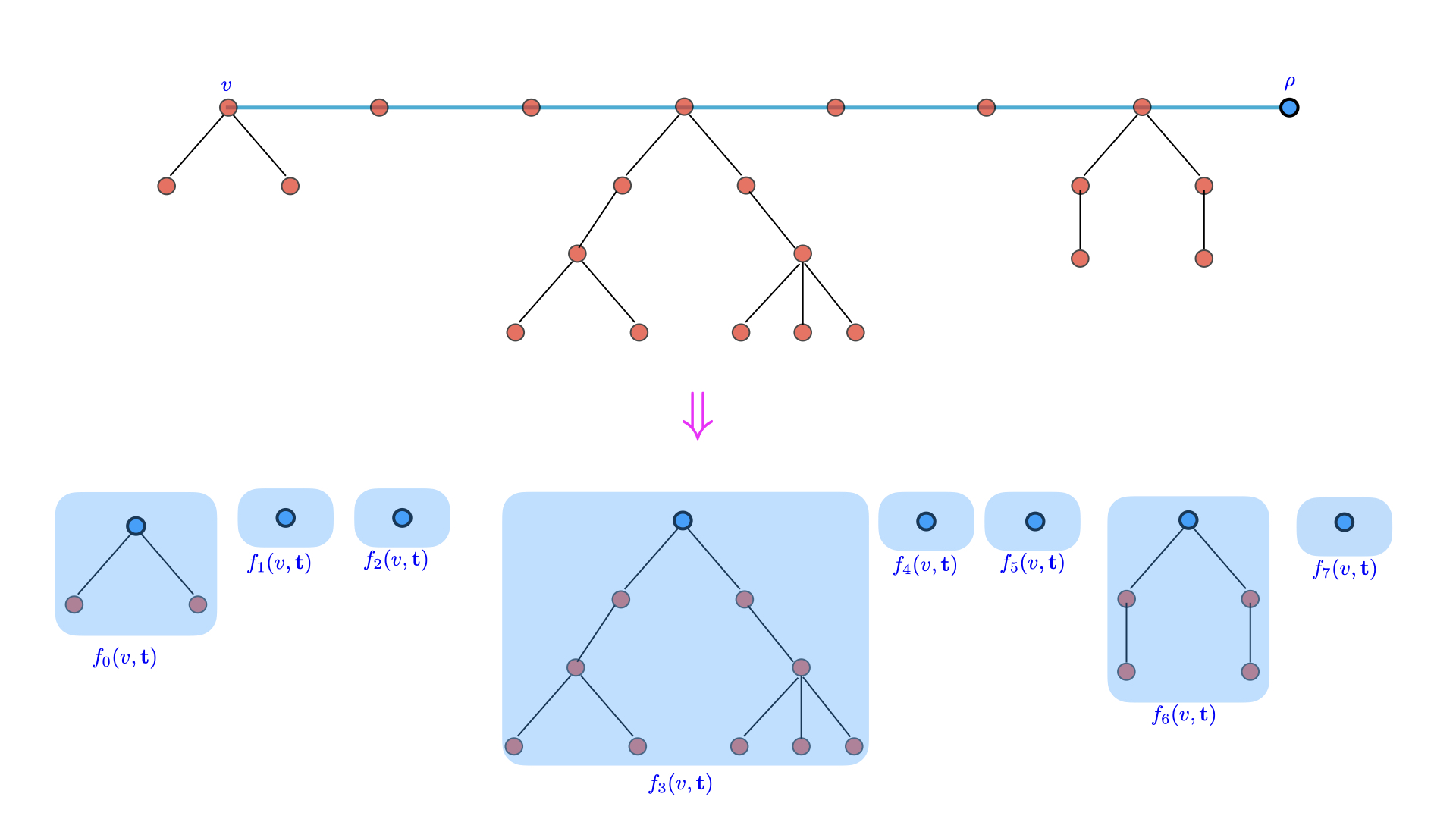}
\caption{Fringe decomposition around vertex $v$ of a finite tree rooted at $\rho$. Here, the blue colors represent the roots of the respective trees. }
\label{fig:fringe}
\end{figure}

  Call the map $(v,\bt) \leadsto \bbT^\infty$ where $v\in \bt$, defined via, 
\[F(v, \bt) = \left(f_0(v,\bt), f_1(v,\bt) , \ldots, f_h(v,\bt), \emptyset, \emptyset, \ldots \right),\]
as the fringe decomposition of $\bt$ about the vertex $v$. Call $f_0(v,\bt)$ the {\bf fringe} of the tree $\bt$ at $v$.
For $k\geq 0$, call $F_k(v,\bt) = (f_0(v,\bt) , \ldots, f_k(v,\bt))$ the {\bf extended fringe} of the tree $\bt$ at $v$ truncated at distance $k$ from $v$ on the path to the root.

Now consider the space $\bbT^\infty$. The metric in~\eqref{eqn:distance-trees} extends  to $\bbT^\infty$, \eg\ using the distance,
\begin{align}
\label{eqn:dist-inf}
	d_{\bbT^\infty}((\bt_0, \bt_1, \ldots),(\bs_0, \bs_1, \ldots)):= \sum_{i=0}^\infty \frac{1}{2^i} d_{\bbT}(\bt_i, \bs_i). 
\end{align}
We can also define analogous extensions to $\bT^k$ for finite $k$.  

Next, an element $\bfomega = (\bt_0, \bt_1, \ldots) \in \bbT^\infty$, with $|\bt_i|\geq 1$ for all $ i\geq 0$,  can be thought of as a locally finite infinite rooted tree with a {\bf s}ingle path to {\bf in}finity (thus called a {\tt sin}-tree~\cite{aldous-fringe}), as follows: Identify the sequence of roots of $\set{\bt_i:i\geq 0}$ with the integer lattice $\Zbold_+ = \set{0,1,2,\ldots}$, equipped with the natural nearest neighbor edge set, rooted at $\rho=0$. Analogous to the definition of extended fringes for finite trees, for any $k\geq 0$, write 
$F_k(0,\bfomega)= (\bt_0, \bt_1, \ldots, \bt_k)$. See Figure~\ref{fig:sin}. 

\begin{figure}[htbp]
\centering
\includegraphics[scale=.2]{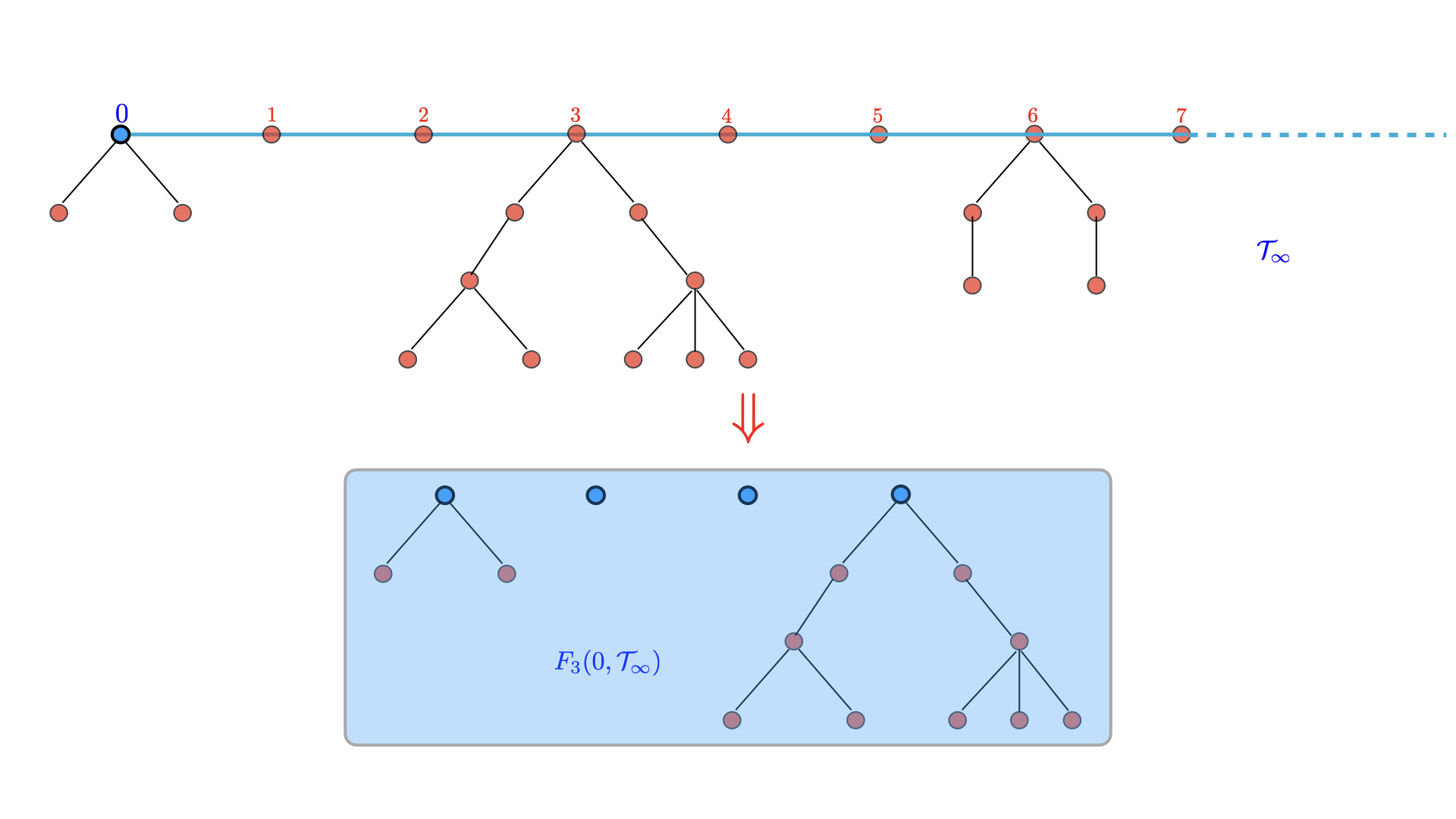}
\caption{A {\tt sin}-tree $\cT_\infty$, namely a tree rooted at $0$ with a single infinite path to infinity, and the corresponding extended fringe $F_3(0,\cT_\infty)$ up to level three about $0$. }
\label{fig:sin}
\end{figure}

Call this the extended fringe of the tree $\bfomega$ at vertex $0$, till distance $k$, on the infinite path from $0$. Call $\bt_0 = F_0(0,\bfomega)$ the {\bf fringe} of the {\tt sin}-tree $\bfomega$. Now suppose $\prob$ is a probability measure on $\bbT^\infty$ such that, for $\TT:= (\bt_0(\TT), \bt_1(\TT),\ldots)\sim \prob$,  $|\bt_i(\TT)|\geq 1$ a.s.  $\forall~i\geq 0$. Then $\TT$ can be thought of as an infinite {\bf random} {\tt sin}-tree. 

Define a matrix $\vQ = (\vQ(\vs,\vt): \vs, \vt \in \bbT)$ as follows: suppose the root $\rho_{\vs}$ in $\vs$ has degree $\deg(\rho_{\vs}) \ge 1$, and let $(v_1,\ldots, v_{\deg(\rho_{\vs})})$ denote its children. { For $1\leq i\leq \deg(\rho_{\vs})$,} let $f(\vs, v_i)$ be the subtree below $v_i$ and rooted at $v_i$, viewed as as an element of $\bbT$. Write,
\begin{align}
\label{eqn:Q-def}
	\vQ(\vs,\vt):= \sum_{i=1}^{\deg(\rho_{\vs})} \ind\set{d_{{ \bbT}}(f(\vs, v_i), \vt) = 0}. 
\end{align} 
Thus, $Q(\vs, \vt)$ counts the number of descendant subtrees of the root of $\vs$ that are isomorphic in the sense of topology to $\vt$. If $\deg(\rho_{\vs})=0$, define $Q(\vs, \vt)=0$. Now consider a sequence $(\bar \vt_0, \bar \vt_1, \dots)$ of trees in $\bbT$ such that $Q(\bar \vt_i, \bar \vt_{i-1}) \ge 1$ { for} all $i \ge 1$. Then there exists a unique infinite {\tt sin}-tree $\TT$ with infinite path indexed by $\Zbold_+$ such that $\bar \vt_i$ is the subtree rooted at $i$ for all $i \in \Zbold_+$. Conversely by taking $\bar \vt_i$ to be the union of (vertices and induced edges) of $\vt_0,\dots, \vt_i$ for each $i \in \Zbold_+$, one may check that every infinite {\tt sin}-tree has such a representation. Following~\cite{aldous-fringe}, we call this the \emph{monotone representation} of the {\tt sin}-tree $\TT$.

\subsubsection{Convergence on the space of trees}
\label{sec:fringe-convg-def}
{ For $1\leq k\leq \infty$, let $\cM_{\pr}(\bbT^k)$ denote the space of probability measures on the associated space. Metrize this space using the topology of weak convergence inherited from the corresponding metric on the space $\bbT^k$, resulting in a Polish space.  See~\cite{billingsley2013convergence} for further details. }
Suppose $\set{\cT_n}_{n\geq 1} \subseteq \bbT$ be a sequence of {\bf finite} rooted random trees on some common probability space (for notational convenience, assume $|\cT_n|= n$, or more generally $|\cT_n|\convas \infty$). For $n\geq 1$ and for each fixed $k\geq 0$, { consider  the empirical distribution of the {\bf extended} fringe up to distance $k$ on the path to the root from each vertex in the tree:} 
\begin{align}
\label{eqn:empirical-fringe-def}
	\fP_{n}^k:= \frac{1}{n} \sum_{v\in \cT_n} \delta\set{F_k(v,\cT_n)}. 
\end{align}
Thus $\set{\fP_{n}^k:n\geq 1}$ can be viewed as a random sequence in  $\cM_{\pr}(\bbT^k)$ and we can talk about weak convergence of this sequence, which is the content of (b) and (c) of the definition below. Further, with $k=0$, define the probability measure  $\E(\fP_n^0)$ on $\bbT$ via the operation $\E(\fP_n^0)[\vt] = \E(\fP_n^0[\vt])~\forall~\vt \in \bbT$.

\begin{defn}[Local weak convergence]
	\label{def:local-weak}
 Fix a probability measure $\varpi$ on $\bT$.
	\begin{enumeratea}
 \item \label{it:fringe-exp} Say that a sequence of trees $\set{\cT_n}_{n\geq 1}$ converges in {\bf expectation}, in the fringe sense, to $\varpi$, if 
 \[\E(\fP_n^{0}) \to \varpi, \quad \text{ on } \cM_{\pr}(\bbT) \quad  \text{ as } n\to\infty. \]
 Denote this convergence by $\TT_n\Efr \varpi$ as $n\to\infty$.
	    \item \label{it:fringe-a}  Say that a sequence of trees $\set{\cT_n}_{n\geq 1}$ converges in probability, in the fringe sense, to $\varpi$, if \[\fP_n^{0} \probc \varpi, \quad \text{ as } n\to\infty. \]
	Denote this convergence by $\TT_n\probfr \varpi$ as $n\to\infty$.
	    \item \label{it:fringe-b} Say that a sequence of trees $\set{\cT_n}_{n\geq 1}$ converges in probability, in the {\bf extended fringe sense}, to a limiting infinite random {\tt sin}-tree $\TT_{\infty}$ if for all $k\geq 0$ one has
	  \[\fP_n^k \probc \prob\left(F_k(0,\TT_{\infty}) \in \cdot \right), \qquad \text{ as } n\to\infty. \]
	Denote this convergence by $\TT_n\probcrf \TT_{\infty}$ as $n\to\infty$.
	\end{enumeratea}
\end{defn} 
{
\begin{rem}\label{lwcrem}
   In part(c), if we replaced convergence in probability with convergence in expectation, namely $\E(\fP_n^k(\cdot)) \to \prob\left(F_k(0,\TT_{\infty}) \in \cdot \right)$ then this is the same as the notion of local weak convergence as in \cite{aldous-steele-obj,benjamini-schramm} specialized to the setting of trees. Such notions allow one to prove convergence of expectations of the degree distribution but not convergence in probability of this same functional.  Part(c) describes a stronger notion, namely, convergence in probability. It is easy to construct settings where one has local weak convergence but {\bf not} (c), for e.g., $\cT_n$ is a rooted random $3$-regular tree upto generation $n$ with probability $1/3$ and random $4$-regular tree  upto generation $n$ with probability $2/3$.    In an identical fashion, one can define notions of convergence in distribution or almost surely in the fringe, respectively, extended fringe sense. 
\end{rem}
}
{
\begin{rem}
\label{rem:star}
   Letting $\varpi_{\infty}(\cdot) = \pr(F_0(0, \cT_{\infty}) = \cdot)$ denote the distribution of the fringe of $\cT_\infty$ on $\bT$, convergence in (c) above clearly implies convergence in notion (b) with $\varpi =\varpi_{\infty}(\cdot) $. However one can have fringe convergence without extended fringe convergence. For example, suppose $\cT_n$ is the $n$-star rooted at the hub. Then it is trivial to check that, with $\varpi_{\infty} = \delta_{\bullet}$, namely the probability measure with all the mass on the rooted tree with one vertex, we have $\TT_n\probfr \varpi_{\infty}$. However, one {\bf does not} have convergence in the extended fringe sense (and thus the sequence of random trees does not converge in the local weak convergence sense). For the random tree models considered in this paper,  we will show that, under weak technical conditions, one always has expected fringe convergence of the sequence of random trees as in (a) (with the limit fringe distribution having an intricate structure). Moreover, when the delay random variable $\xi$ does not have positive mass at $1$, then one automatically has extended fringe convergence (implying local weak convergence) to an appropriate limit object. 
\end{rem}
}

{ It turns out, as discovered in \cite{aldous-fringe},  if the limiting distribution $\varpi$ in (b) has certain `stationarity' and extremality properties defined next, then \emph{convergence in the expected fringe sense implies convergence in the extended fringe sense}. In this setting, {\bf just proving convergence of expectations} of the fringe distribution \emph{automatically guarantees} convergence in probability to a limit random object. }


\begin{defn}[Fringe distribution~\cite{aldous-fringe}] \label{fringedef}
	Say that a probability measure $\varpi$ on $\bbT$ is a fringe distribution if 
	\[\sum_{\vs} \varpi(\vs) \vQ(\vs, \vt) = \varpi(\vt), \qquad \forall~\vt \in \bbT. \]
\end{defn}
It is easy to check that the space of fringe distributions $\cM_{\pr, \fringe}(\bbT) \subseteq \cM_{\pr}(\bT)$ is a convex subspace of the space of probability measure on $\bT$ and thus one can talk about extreme points of this convex subspace. The following fundamental theorem is one of the highlights of~\cite{aldous-fringe}. 
\begin{thm}[\cite{aldous-fringe}]
\label{thm:aldous-efr-pfr}
    Fix a fringe distribution $\varpi \in \cM_{\pr, \fringe}(\bbT)$. Suppose a sequence of trees $\set{\cT_n:n\geq 1}$ converges in the expected fringe sense $\cT_n \Efr \varpi$ as $n\to\infty$. If $\varpi$ is extremal in the space of fringe measures, then the above convergence in expectation automatically implies $\cT_n \probfr \varpi$. 
\end{thm}
The advantage of this result is that for proving convergence in the probability fringe sense, at least under the extremality of the limit object, dealing with expectations is enough. 
The next result shows that convergence in the probability fringe sense often automatically implies convergence to a limit infinite {\tt sin}-tree. We need one additional definition. For any fringe distribution $\varpi$ on $\bbT$, one can uniquely obtain the law $\varpi^{EF}$ of a random {\tt sin}-tree $\TT$ with monotone decomposition $(\bar\bt_0(\TT), \bar\bt_1(\TT),\ldots)$ such that for any $i \in \Zbold_+$, any $\bar \vt_0, \bar \vt_1, \dots$ in $\bbT$,
\begin{align}\label{ftoef}
\varpi^{EF}((\bar\bt_0(\TT), \bar\bt_1(\TT),\ldots, \bar\bt_i(\TT)) = (\bar \vt_0,\vt_1,\dots,\vt_i)) := \varpi(\vt_i) \prod_{j=1}^{i}Q(\vt_i,\vt_{i-1}),
\end{align}
where the product is taken to be one if $i=0$. The following lemma follows by adapting the proof of~\cite[Propositions 10 and 11]{aldous-fringe}, and the proof is omitted.

\begin{lemma}\label{ftoeflemma}
Suppose a sequence of trees $\set{\cT_n}_{n\geq 1}$ converges in probability, in the fringe sense, to $\varpi$. Moreover, suppose that $\varpi$ is a fringe distribution in the sense of Definition~\ref{fringedef}. Then $\set{\cT_n}_{n\geq 1}$ converges in probability, in the extended fringe sense, to a limiting infinite random sin-tree $\TT_{\infty}$ whose law $\varpi^{EF}$ is uniquely obtained from $\varpi$ via~\eqref{ftoef}.
\end{lemma}

Fringe convergence and extended fringe convergence imply convergence of functionals, such as the degree distribution. For example, letting $\cT_{\varpi} \sim \varpi$ with root denoted by $0$ say, convergence in notion~\eqref{it:fringe-a} in particular implies that for any $k\geq 0$, 
\begin{align}
\label{eqn:deg-convg-fr}
	\frac1n\cdot\#\set{v\in \TT_n: \deg(v) = k+1} \convp \prob(\deg(0,\cT_{\varpi})=k).
\end{align}
However, both convergences give much more information about the asymptotic properties of $\set{\cT_n:n\geq 1}$ beyond its degree distribution.

\section{Limit objects for Macroscopic delays}\label{limobj}

{\subsection{Intuition behind limiting objects}\label{intuition}
Before formally exhibiting the local weak limit, we give an intuitive picture of the objects used in its description. In the context of dynamic networks, local weak limits are often convenient to describe in terms of ``continuous time'' branching processes (CTBP) mimicking the original network process. This enables a description of the local weak limit in terms of the CTBP randomly stopped at an independent exponential time \cite{rudas-2,banerjee2022co,banerjee2023local}. Thus, we start by considering a continuous time analog $\Cod$ of the current network process.  Start with one vertex $\Cod(0) = \set{\rho}$ at time $t=0$, and for each $t \ge 0$, a new vertex is born into the system at rate $|\Cod(t)|$. A new vertex $v$ born, say at time $t$, independently samples a delay variable $\eta_v$ and uses it to connect to a vertex in $\Cod(t-\eta_v)$ with probability proportional to its degree in $\Cod(t-\eta_v)$. 

The first question is the relationship between the delay distribution $F_\eta$ of the delay $\eta$ in continuous time and the original delay $\xi \sim \mu$ of the discrete time network process. From the dynamics of $\Cod$, it is clear that the size of the process $\set{|\Cod(t)|:t\geq 0}$ has the same distribution as a rate one Yule process; standard results about the Yule process imply that $e^{-t}|\Cod(t)| \stackrel{a.s., \bL^2}{\to} W$ where $W\sim \exp(1)$. 
Thus, heuristically, one should have, 
\[1- \xi \approx \frac{|\Cod(t-\eta)|}{|\Cod(t)|} \approx e^{-\eta}. \]
This gives $\eta \approx -\log(1-\xi)$ giving an explicit transformation of the continuous time delay $\eta$ in terms of the original delay $\xi$. 

Now, we heuristically construct the reproduction point process that will be used to describe the CTBP and, in turn, the local limit. For some large time $t>0$, consider a vertex $v \in \Cod(t)$, born at time $B(v) < t$, and suppose that $v$ has already attached to $i$ vertices after its birth (think of these as children of $v$) at times $\sigma_1 + B(v)<\dots <\sigma_i + B(v) \le t$. Then the instantaneous rate at which $v$ gives birth to a new child at time $t$ is
\begin{align*}
&\int_0^{t-B(v)}\left(\sum_{j=1}^i \frac{j+\alpha}{2+\alpha}\ind\set{t - a \in [B(v) + \sigma_{j-1}, B(v) + \sigma_j)}\right)\frac{|\Cod(t)|}{|\Cod(t-a)|}dF_\eta(a)\\
& \approx \int_0^{t-B(v)}\left(\sum_{j=1}^i \frac{j+\alpha}{2+\alpha}\ind\set{t - a \in [B(v) + \sigma_{j-1}, B(v) + \sigma_j)}\right)e^a dF_\eta(a).
\end{align*}
In particular, the above implies that for $x>0$, the instantaneous reproduction rate of $v$ at time $\sigma_i + B(v) + x$, given no births in $[\sigma_i + B(v),\sigma_i + B(v) + x)$, is approximately
\begin{align*}
&\int_0^{\sigma_i+x}\left(\sum_{j=1}^i \frac{j+\alpha}{2+\alpha}\ind\set{\sigma_i + B(v) + x - a \in [B(v) + \sigma_{j-1}, B(v) + \sigma_j)}\right)e^a dF_\eta(a)\\
&=\frac{1}{2+\alpha}\left[ \sum_{j=1}^{i} \int_{x+\sigma_i - \sigma_j}^{x+\sigma_i - \sigma_{j-1}} (j+\alpha) e^a dF_{\eta}(a) + \int_0^x (i+1+\alpha) e^{a} dF_{\eta}(a)\right].
\end{align*}
We will later formally define these rates in terms of `hazard rates' for the associated reproduction point processes. 

From the above heuristic, it is natural to expect that the local limit will be a CTBP driven by the above point processes, stopped at an independent exponential time $T_1$. To guess its exact distribution, let $V_t$ denote a randomly selected vertex in $\Cod(t)$, born at time $B(V_t)$, for some large $t>0$. Then its {\bf age} at time $t$, $\age(V_t) = t- B(V_t)$, satisfies 
\begin{equation}
    \label{eqn:958}
    \pr(\age(V_t)> a|\Cod(t)) = \frac{|\Cod(t-a)|}{|\Cod(t)|}\approx e^{-a} \Rightarrow \age(V_t) \approx \exp(1). 
\end{equation}
$T_1$ typically corresponds to the age of a uniformly chosen vertex when the network is large, and thus the above points to $T_1\approx \exp(1)$.

In subsequent sections, we will formalize the above ideas, culminating in the proof of the local limit. }

\subsection{A branching process driven by point processes with memory}
\label{sec:lim-obj-macro}
In the setting of macroscopic delays, we restrict ourselves to the affine Preferential attachment setting with attachment function $f(k) = k+\alpha$ for $k\geq 1$, for a fixed parameter $\alpha\geq 0$. We will write the corresponding model in Definition~\ref{defn:model} as $\cL(\gb \equiv 1, \mu, \alpha)$, where now the delay distribution $\mu$ is supported on $[0,1]$.  In the construction of the continuous-time branching process~\cite{jagers-ctbp-book,jagers-nerman-1,jagers-nerman-2}, one formulation is via describing a point process $\xi$ with inter-arrival times $\set{\sE_{(i-1) \leadsto i}: i\geq 1}$, where for each $i\geq 1$, conceptually $\sE_{(i-1) \leadsto i}$ (in the branching process) has the interpretation as the amount of time for a vertex to go from having $(i-1)$ to $i$ children.  We will now specify a point process via recursive construction of its inter-arrival times.  Recall the delay distribution random variable $\xi\in [0,1]$ from Definition~\ref{defn:model}. Define random variable $\eta$ by,
\begin{align}
\label{eqn:eta-dist-def}
\eta { =} -\log(1-\xi) \text{ with } \pr(\eta \in dx){ =} dF_\eta(x), \text{ for } x\in [0,\infty]. 
\end{align}
Note that $\xi = 1$ corresponds to the event $\eta = \infty$.   Let $\sigma_i = \sum_{j=1}^i \sE_{j-1\leadsto j}$, with $\sigma_0\equiv 0$. The construction proceeds as follows: 
\begin{enumeratea}
	\item {\bf Base case $\sE_{0 \leadsto 1}$:}  Define the hazard function $h_{0\leadsto 1}(\cdot)$ via, 
\begin{align}
\label{eqn:h01-hazard}
h_{0\leadsto 1}(x):= \frac{1+\alpha}{2+\alpha}\int_{0}^x e^u dF_\eta(u), \qquad x\in [0,\infty].
\end{align}
Let $\sE_{0\leadsto 1}$ be the random variable on $[0,\infty]$ with the above hazard rate so that for any $x$, $\pr(\sE_{0\leadsto 1} > s) = \exp(-\int_0^s h_{0\leadsto 1}(x) dx)$.  

	\item {\bf General case (general $i$):} Having constructed $\set{\sE_{j-1 \leadsto j}: 1\leq j \leq i}$, conditional on the above sequence, consider the hazard function for $x> 0$,
\begin{align}
\label{eqn:h-gen-hazard}
h_{i\leadsto i+1}(x) := \frac{1}{2+\alpha}\left[ \sum_{j=1}^{i} \int_{x+\sigma_i - \sigma_j}^{x+\sigma_i - \sigma_{j-1}} (j+\alpha) e^u dF_{\eta}(u) + \int_0^x (i+1+\alpha) e^{u} dF_{\eta}(u)\right].  
\end{align}
	 Let $\sE_{i \leadsto i+1} >0$ a.s. be the random variable with the above hazard rate. 
\end{enumeratea}
 
\begin{figure}[htbp]
	\centering
		\begin{tikzpicture}
		[scale=.8, background rectangle/.style=
		     {draw=blue!30,fill=blue!5,rounded corners=2ex},
		   show background rectangle]
		\draw[very thick,blue, |-] (0,0) -- (11,0);
				\draw[very thick,red, |->] (11,0) -- (16,0);
			\node at ( 14,0) [regular polygon, regular polygon sides= 6, draw=red!50,fill=red!100, label = above: $\sigma_3+x$] {};
		   \node at (-1,0) [draw = red!100, fill = blue!20] {$\mathcal{P}_{\Ma,\alpha}$};
			\node at ( 4,0) [circle,draw=blue!50,fill=blue!100, label = above: $\sigma_{1}$] {};
			\node at ( 6,0) [circle,draw=blue!50,fill=blue!100, label = above: $\sigma_{2}$] {};
			\node at ( 11,0) [circle,draw=blue!50,fill=blue!100, label = above: $\sigma_{3}$] {};
			\draw[very thick,red, |-|] (0,-1) -- (4,-1);
			\node at ( 2, -1) [label = below: $\sE_{0\leadsto 1}$] {};
			\draw[very thick,red, |-|] (4,-2) -- (6,-2);
			\node at ( 5, -2) [label = below: $\sE_{1\leadsto 2}$] {};
			\draw[very thick,red, |-|] (6,-3) -- (11,-3);
			\node at ( 9, -3) [label = below: $\sE_{2\leadsto 3}$] {};
		\end{tikzpicture}	
		
	\caption{Showing the construction of $\sE_{3 \leadsto 4}$ conditional on the preceding inter-arrival times and the corresponding point process $\cP_{\Ma,\alpha}$ in the Macroscopic regime.  }
	\label{fig:point-process-macro}
\end{figure}
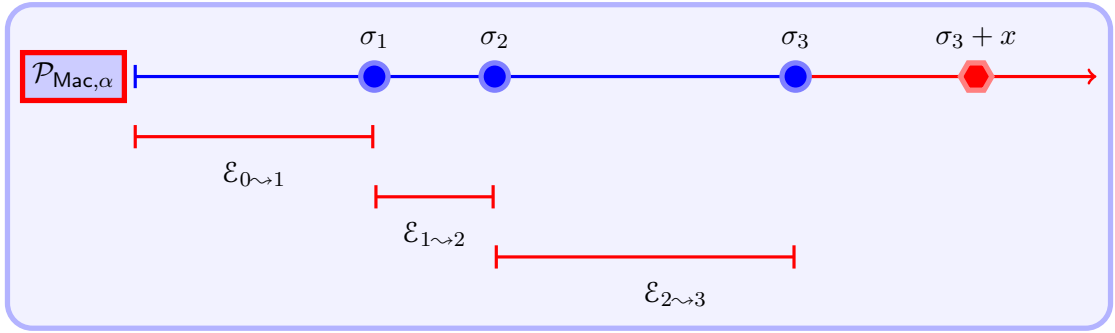

The following gives a simpler representation of the above hazard rates, which can be verified from the definition { by reorganizing the terms in the sum appearing in \eqref{eqn:h-gen-hazard}}. 

\begin{lem}
    \label{lem:hazard-rate-sum}
    Let $h(x) =  \frac{1}{2+\alpha}\int_{0}^x e^u dF_\eta(u)$. Then,
    \[ h_{i\leadsto i+1}(x) = \sum_{j=0}^i h(x+\sigma_i -\sigma_j) + \alpha h(x + \sigma_i), \qquad x\geq 0.  \]
\end{lem}

As shown in Theorem~\ref{thm:macroscopic-local}, the local weak limit of our network model is given by a branching process, with reproductions driven by $\cP_{\Ma, \alpha}$, stopped at an independent exponential time. We now define these objects.
\begin{defn}[BP in the Macroscopic regime]
\label{def:limit-bp-macro}
	Let $\BP_{\Ma,\alpha}(\cdot)$ be a continuous-time branching process started with one individual at time $t=0$, where each individual has offspring distribution $\cP_{\Ma, \alpha}$ defined above. Let $T_{1}$ be an $\exp(1)$ random variable, independent of $\BP_{\Ma,\alpha}$, and write $\varpi_{\Ma,\alpha}$ for the distribution of $\BP_{\Ma, \alpha}(T_1)$, viewed as a random finite rooted tree on $\bbT$, where we retain only genealogical information between individuals in $\BP_{\Ma, \alpha}(T_1)$. 
\end{defn}

\subsection{Alternate description of \texorpdfstring{$\BP_{\Ma}$}{BP-Mac} using edge branching process}\label{edgedesc}

{ Although the above branching process is natural in view of Section \ref{intuition} and is used in the rigorous proof of local weak convergence to the limit, its non-Markovian evolution makes it highly challenging to extract more quantitative information (like the limiting degree distribution exponent). A key insight that we present now is a `dual' description of the branching process $\BP_{\Ma}$ in terms of another branching process $\BP_{\Ma,\ve}$ where the \emph{edges reproduce} (instead of vertices). As we will see later, this construction uncovers a certain amount of independence in $\BP_{\Ma,\ve}$, thereby making it amenable to powerful branching process techniques.} 

Consider the branching process $\BP_{\Ma}$ in Def.~\ref{def:limit-bp-macro} (we will suppress the dependence on $\alpha$ to ease notation).  Let $\set{\Zma(t):t\geq 0}$ denote the process representing the number of children of the root, namely for any $t\geq 0$, $\Zma(t) = \#\set{i: \sum_{j=1}^i \cE_{(j-1)\leadsto j} \leq t}$. By Theorem~\ref{thm:macroscopic-local} below, the limit degree distribution $p_{\Ma}$ is just the p.m.f of the random variable $\Zma(T_1)+1$ where $T_1\sim\exp(1)$ independent of $\Zma$.

Let us first explain the origin of this construction in the no-delay model where the network evolves using ``pure'' preferential attachment, i.e., a new vertex enters the system at each discrete time step and attaches to an existing vertex with probability proportional to the degree of the existing vertex. It is well known (e.g.,~\cite{krapivsky2005network,vazquez2003growing}) that this model is equivalent to the following \emph{edge-copying} model (see Figure~\ref{fig:three graphs}):

\begin{enumeratea}
    \item At each stage, a new vertex $v$ enters the system and picks an edge in the existing graph uniformly at random to copy. Thus, conceptually, we can think of each edge reproducing at rate one, independent across edges. 
    \item Suppose the edge selected to copy is $\set{a,b}$. Then $v$ attaches to $a$ with probability $1/2$ and $b$ with probability $1/2$.
\end{enumeratea}

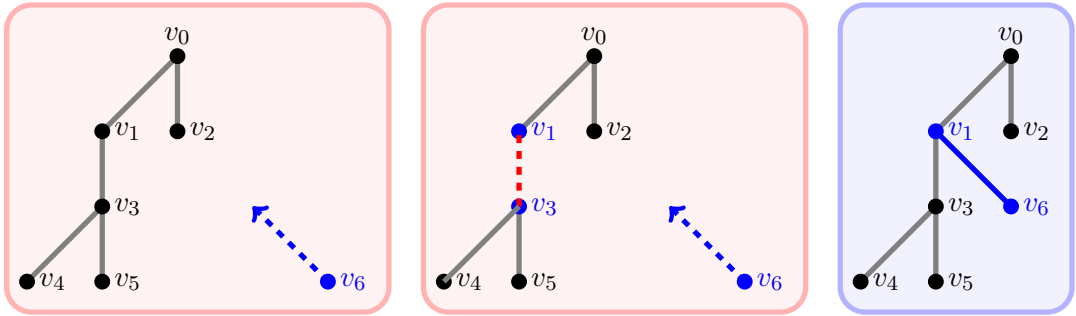
\begin{figure}[h]   
\centering
\begin{subfigure}[b]{0.3\textwidth}
\begin{tikzpicture}[background rectangle/.style=
		     {draw=red!30,fill=red!5,rounded corners=2ex},
		   show background rectangle]
  \draw[gray](0,0)--(0,-1);
  \draw[gray](0,0)--(-1,-1);
  \draw[gray](-1,-1)--(-1,-2);
  \draw[gray](-1,-2)--(-1,-3);
  \draw[gray](-1,-2)--(-2,-3);
    \filldraw(0,0) circle (2pt) node[above] {$v_0$};
   \filldraw(0,-1) circle (2pt) node[right] {$v_2$};
    \filldraw(-1,-1) circle (2pt) node[right] {$v_1$};
    \filldraw(-1,-2) circle (2pt) node[right] {$v_3$};
     \filldraw(-1,-3) circle (2pt) node[right] {$v_5$};
       \filldraw(-2,-3) circle (2pt) node[right] {$v_4$};
       \draw[dashed, blue,->](2,-3)--(1,-2);
       \filldraw[blue](2,-3) circle (2pt) node[right] {$v_6$};
\end{tikzpicture}
\end{subfigure}
\hspace{2ex}
\begin{subfigure}[b]{0.3\textwidth}
\begin{tikzpicture}[background rectangle/.style=
		     {draw=red!30,fill=red!5,rounded corners=2ex},
		   show background rectangle]
  \draw[gray](0,0)--(0,-1);
  \draw[gray](0,0)--(-1,-1);
  \draw[gray](-1,-2)--(-1,-3);
 \filldraw(0,0) circle (2pt) node[above] {$v_0$};
   \filldraw(0,-1) circle (2pt) node[right] {$v_2$};
    \filldraw[blue](-1,-1) circle (2pt) node[right] {$v_1$};
    \filldraw[blue](-1,-2) circle (2pt) node[right] {$v_3$};
     \filldraw(-1,-3) circle (2pt) node[right] {$v_5$};
       \filldraw(-2,-3) circle (2pt) node[right] {$v_4$};
         \draw[-, red, dashed](-1,-2)--(-1,-1);
           \draw[gray](-2,-3)--(-1,-2);   
              \draw[dashed, blue,->](2,-3)--(1,-2);
       \filldraw[blue](2,-3) circle (2pt) node[right] {$v_6$};
\end{tikzpicture}
\end{subfigure}  
\hspace{2ex}
\begin{subfigure}[b]{0.2\textwidth}
\begin{tikzpicture}[background rectangle/.style=
		     {draw=blue!30,fill=blue!5,rounded corners=2ex},
		   show background rectangle]
  \draw[gray](0,0)--(0,-1);
  \draw[gray](0,0)--(-1,-1);
  \draw[gray](-1,-1)--(-1,-2);
  \draw[gray](-1,-2)--(-1,-3);
  \draw[gray](-1,-2)--(-2,-3);
    \filldraw(0,0) circle (2pt) node[above] {$v_0$};
   \filldraw(0,-1) circle (2pt) node[right] {$v_2$};
    \filldraw[blue](-1,-1) circle (2pt) node[right] {$v_1$};
    \filldraw(-1,-2) circle (2pt) node[right] {$v_3$};
     \filldraw(-1,-3) circle (2pt) node[right] {$v_5$};
       \filldraw(-2,-3) circle (2pt) node[right] {$v_4$};
          \draw[-, blue](0,-2)--(-1,-1);
        \filldraw[blue](0,-2) circle (2pt) node[right] {$v_6$};
\end{tikzpicture}
\end{subfigure}
        \caption{$v_6$ is a new incoming vertex, and selects the edge connecting $v_1$ and $v_3$ to ``copy''; it then chooses either $v_1$ or $v_3$ with equal probability (in this case $v_1$).}
        \label{fig:three graphs}
\end{figure}

The dynamics above can be modified to incorporate affine attachment functions as well. It turns out, at least in the limit, one can construct similar dynamics in the delay regime, which allows significantly more tractable expressions for functionals of interest. The rest of this section is organized as follows:
\begin{enumeratei}
    \item We will describe an edge branching process $\BP_{\Ma, \ve}$ for the general macroscopic delay model with affine preferential attachment function with parameter $\alpha$.
    \item The offspring (children/immediate descendants) of $\BP_{\Ma, \ve}$ can be constructed using a simpler primitive branching process $\BP_{\ve}^{\circ}$ and a Poisson modulated immigration process $\chi(\cdot)$. 
\end{enumeratei}

\begin{defn}[Edge branching process $\BP_{\Ma, \ve}$]
    \label{def:bpmave}
    Consider the following branching process  $\BP_{\Ma,\ve}$: 
\begin{enumerateA}
    \item At time $0$, the population consists of two individuals $\tilde{v}_0, \tilde{v}_1$ connected by an edge $\ve_0$ directed from $\tilde{v}_1$ (child) to $\tilde{v}_0$ (parent).
    \item\label{item:rt} At time $t \ge 0$ units after its birth, an existing edge $\ve$ reproduces according to an inhomogeneous Poisson point process with rate $r(t):= \int_0^t e^s dF_{\eta}(s)$, giving birth to a new edge (and thereby, a new individual) in the system.
    \item At each reproduction epoch of edge $\ve$, the resulting new edge attaches to the parent vertex of $\ve$ with probability $1/(2 + \alpha)$ and to the child of $\ve$ with probability $(1 + \alpha)/(2 + \alpha)$.
\end{enumerateA}
Denote by $\set{\BP_{\Ma,\ve}(t) : t \ge 0}$ the \emph{entire set of descendants }  of $\tilde{v}_1$ at time $t$. Further let $\rho_{\ve}(t)$ denote the number of children (immediate descendants) of $\tilde{v}_1$ by time $t$. 
\end{defn}

Besides the description of $\BP_{\Ma,\ve}$ as a more conventional continuous-time branching process (without complex hazard rates), the added advantage of the above construction is that the degree of the root (and thus any other vertex) in this branching process evolves as the size of \emph{another branching process with immigration}.  

Recall the hazard rate function~\eqref{eqn:h01-hazard}.  We will let $\mvzeta_{\ve}$ denote the Poisson point process on $[0,\infty)$ with intensity measure, 
\begin{align}
    \label{eqn:intensity-measure} 
    \mu_{\mvzeta_{\ve}}(dt) = h(t) dt =  \left[\frac{1}{2 + \alpha} \int_0^t e^s dF_\eta(s)  \right] dt. 
\end{align}

\begin{enumeratea}
    \item Let $\BP^\circ_{\ve}$ denote a (continuous time) branching process started with one individual at time $t=0$ and { where individuals reproduce independently according to a Poisson process $\mvzeta_{\ve}$ with intensity measure $\mu_{\mvzeta_{\ve}}$.} Let $\{\BP^\circ_{\ve,i} : i \in \mathbb{N}_0\}$ be iid copies of $\BP^\circ_{\ve}$ with respective sizes given by $\{|\BP^\circ_{\ve,i}(t)| : i \in \mathbb{N}_0\}$.
    \item Let $\{\chi(t): t \ge 0\}$ denote a the Poisson point process with rate $r_{\chi}(t) := \frac{\alpha}{2 + \alpha} r(t), \, t \ge 0,$, with $r(\cdot)$ as in Def.~\ref{def:bpmave}~\eqref{item:rt}. Let $\{\theta_i : i \in \mathbb{N}_0\}$ denote the reproduction times of $\chi$, with $\theta_0=0$ (that is, $\chi(0)=1$).
\end{enumeratea}
    Consider the branching process with immigration $\BP_{\ve}$, starting with one individual, where the progeny of each existing individual grows according to (an independent copy of) $\BP^\circ_{\ve}$, and there is an additional incoming stream of individuals into the population at epochs of $\chi$. The size of this branching process is thus given by
\begin{equation}
    |\BP_{\ve}(t)| := \sum_{i=0}^{\chi(t)-1} |\BP^\circ_{\ve,i}(t - \theta_i)|, \ t \ge 0.
\end{equation}

\begin{prop}
    \label{prop:edge-equiv-bp} 
    We have the distributional equivalence
    \begin{equation}\label{eq:degeq}
        \set{\Zma(t)+1: t\geq 0}  \stackrel{d}{=} \set{\rho_{\ve}(t) : t \ge 0} \stackrel{d}{=} \set{|\BP_{\ve}(t)|: t\geq 0}.
    \end{equation}
    Thus,
    $ \set{\BP_{\Ma,\ve}(t): t \ge 0} \stackrel{d}{=} \set{\BP_{\Ma}(t): t \ge 0}.$
    Consequently, $\varpi_{\Ma}\stackrel{d}{=}\BP_{\Ma,\ve}(T_1)$ and, writing $D_{\Ma}$ for the limit degree random variable in Theorem~\ref{thm:macroscopic-local} with distribution  $p_{\Ma}(\cdot)$,  $ D_{\Ma} \stackrel{d}{=} |\BP_{\ve}(T_1)|$, where $T_1\sim \exp(1)$ independent of $\BP_{\ve}$.   
\end{prop}
\begin{proof}
To prove the proposition, it suffices to show~\eqref{eq:degeq} along with the independence of reproduction point processes across individuals in $\BP_{\Ma,\ve}$. The last distributional equality is immediate from rate considerations of the associated Poisson point processes. We will now show that the first and third objects in~\eqref{eq:degeq} have the same law.

    Write $\set{\tilde{\sigma_k}: k\geq 0}$ for the times of arrival of new individuals into the branching process $\BP_{\ve}(\cdot)$ starting with  $\tilde{\sigma}_0=0$ when the root of $\BP_{\ve}$ is the only vertex in the system. Write $\tilde{\cE}_{k\leadsto (k+1)} = \tilde{\sigma}_{k+1} -\tilde{\sigma}_k$ for the inter-arrival time for the branching process between the $k$-th and $k+1$-th individual. Then, directly from construction, it is easy to check that for any $s> 0$,
    \begin{equation}
        \label{eqn:1112}
        \pr(\tilde{\cE}_{k\leadsto (k+1)} > s \mid \tilde{\sigma}_1, \ldots, \tilde{\sigma}_k ) = \exp\bigg(-\int_0^s  \left(\sum_{i=0}^k h(u+\tilde{\sigma}_k-\tilde{\sigma}_i) + \alpha h(u + \tilde{\sigma}_k)\right) du\bigg).
    \end{equation}
    { Indeed for $\tilde{\cE}_{k\leadsto (k+1)}$ to exceed $s$, the existing individuals should not give birth and the immigration stream $\chi(\cdot)$ should not have an arrival in the interval $[\tilde{\sigma}_k, \tilde{\sigma}_k + s]$ and the above follows using the independence of the associated birth and immigration processes.}
    Comparing this to the hazard rates for $\Zma(\cdot)$ in Section~\ref{sec:lim-obj-macro} and using Lemma~\ref{lem:hazard-rate-sum} completes the proof. The independence of birth processes across individuals also follows via similar rate considerations.
\end{proof}

The above construction will play a crucial role in proving all the main results related to degree distribution tail behavior and analyzing condensation phenomena at the root, which are formally stated in Section~\ref{sec:main-res}.

\section{Main results}
\label{sec:main-res}

{ For the rest of this paper, we set  $\gb = 1$, namely the macroscopic regime,} and the model uses linear attachment with fixed affine parameter $\alpha \geq 0$. Further, we assume that $\pr(\xi =0) < 1$, { since otherwise one is in the no-delay regime,} and $\pr(\xi =1) < 1$ { so that we are not in the trivial regime where all vertices directly attach to the root}. Recall that in the setting {\bf without} delay, the limit degree distribution of the model (with associated degree exponent) is given by~\cite{bollobas2001degree,bollobas2003directed,barabasi1999emergence}:
\begin{equation}
    \label{eqn:no-delay-limit}
    p_{\nnd, \alpha }(k)  = (2 + \alpha) \frac{\Gamma(k + \alpha)\Gamma(3 + 2\alpha)}{\Gamma(k + 3 + 2\alpha)\Gamma(1 + \alpha)},~~\Rightarrow~p_{\nnd, \alpha}(k) \sim C/k^{3+\alpha}, \quad \text{as } k\to\infty . 
\end{equation}
As we describe the main results, the above can be viewed as something to compare and contrast in the setting with macroscopic delays. 

{ \subsection{Formal statement of results}
We now give precise statements of our results. Recall, as before, from~\eqref{eqn:deg-count} that for $k\geq 1$, $N_{k}(n)$ denotes the number of nodes in $\cT(n)$ with degree exactly $k$. 

\begin{thm}[Local weak limit, macroscopic regime]\label{thm:macroscopic-local}
\begin{enumeratea}
\item Consider the sequence of random trees
\[
\set{\cT(n):n\geq 2} \sim \cL(\gb \equiv 1, \mu, f(\cdot) = \cdot+\alpha)
\]
in the macroscopic regime with affine linear attachment function and delay distribution $\mu$ on $[0,1]$.   Then $\set{\cT(n):n\geq 1}$ converge in the expected fringe sense (Def.~\ref{def:local-weak}~\eqref{it:fringe-exp}) to $\varpi_{\Ma,\alpha}$ as in Definition~\ref{def:limit-bp-macro}. In particular, the degree distribution satisfies, 
	\[
 \E\left(\frac{N_k(n)}{n}\right) \to \pr\left(\sum_{j=1}^{k-1} \sE_{\sss j-1\leadsto j} \leq T_1 < \sum_{j=1}^{k} \sE_{\sss j-1\leadsto j}\right):=p_{\Ma, \alpha}(k), \qquad k\geq 1,
 \]
	where $T_1$ is a rate one exponential random variable independent of $\set{\sE_{\sss j-1 \leadsto j}:j\geq 1 }$. 
    \item The limit degree distribution $D_{\Ma, \alpha}$ with pmf $p_{\Ma, \alpha}$ satisfies 
    \[\E(D_{\Ma}) = \frac{(2+ \alpha) + \alpha \pr(\eta < \infty)}{2 + \alpha-\pr(\eta < \infty)}.\]
 \item  If $\pr(\xi = 1) = \pr(\eta =\infty) =0$, then the sequence of random trees $\set{\cT(n):n\geq 2} \sim \cL(\gb \equiv 1, \mu, f(\cdot) = \cdot+\alpha)$ converges in probability in the extended fringe sense (Definition~\ref{def:local-weak}~\ref{it:fringe-b}) to the unique infinite random {\tt sin}-tree with fringe distribution  $\varpi_{\Ma,\alpha}$.
 \end{enumeratea}
\end{thm}
\begin{rem}
    Note that (a) holds, irrespective of whether the delay distribution has mass at $1$ (namely $\mu(\{1\})>0$) or not. Thus, there is always convergence in the expected fringe sense for this model. However, if $\mu(\{1\})>0$, then one can check that one cannot have local weak convergence (see Remarks \ref{lwcrem} and \ref{rem:star}), and (c) says that this is the only obstacle for convergence in probability in the extended fringe sense, which is stronger than local weak convergence. 
\end{rem}

Besides giving asymptotics for empirical distributions of local functionals like the degree distribution, the local limit results also imply convergence of global functions. We give an example. Let $\cT(n)$ be the random tree as above and let $\vA_n$ denote the adjacency matrix of $\cT(n)$ and let $\set{\lambda_i^{\sss(n)}:1\leq i\leq n}$ denote the eigen-values and $\hat{\mu}_n = n^{-1} \sum_{i=1}^n \delta_{\lambda_i^{\sss(n)}}$ denote the empirical spectral distribution of $\vA_n$ where $\delta$ is the Dirac delta function. 

\begin{cor}
\label{cor:rand-adj}
Under the assumptions of Theorem~\ref{thm:macroscopic-local}(c), there exists a deterministic distribution $\mu_\infty$ (whose specific form depends on the parameters $\alpha, \mu$) such that $\hat{\mu}_n \convd \mu_\infty$. The limit distribution has an infinite set of atoms in $\bR$.
\end{cor}

The proof of this result follows directly by combining the extended fringe convergence in Theorem~\ref{thm:macroscopic-local}(c) with Theorem 4.1 in~\cite{bhamidi2012spectra}.

Now, we address the limiting degree distribution asymptotics. To start off, consider an illustrative special case where one vacillates between complete information and no information (thus connecting directly to the root). The following corollary to the local limit result gives an exact formula for the limiting degree distribution in this case.
\begin{cor}
\label{cor:special-case-macro}
    Fix $p\in (0,1)$ and suppose the delay distribution $\mu = \Bern(1-p) = p\delta_{0} + (1-p) \delta_{1}$. Then, for $\alpha = 0$, the limit degree distribution is given by 
    \[p_{\Ma}(k) = \frac{4(k-1)!p^k}{\prod_{j=1}^k (jp+2)} \sim C k^{-(1+2/p)}, \qquad \text{as } k\to \infty,\]
    for a constant $C\in (0,\infty)$. In particular, since the degree exponent is $(1+2/p) >3$, the degree distribution tail is strictly lighter than the regime with no delay. 
\end{cor}
The reader can compare this result with known results in the no-delay setting in~\eqref{eqn:no-delay-limit}. In the presence of delay, with ``older vertices'' getting a higher chance to reinforce their degree, one might imagine that the limiting degree distribution (which gives the degree behavior of ``younger vertices" near the fringe) should have \emph{lighter} tails. In the setting of Corollary~\ref{cor:special-case-macro}, this is seen to be true.

Under a wide array of settings, if $\pr(\xi = 1) =0$ (i.e., no ``trivial'' mechanism for root connectivity), we will see that the degree exponent is always {\bf heavier} than the model without delay. The goal of the next few results is to understand the tail behavior of the degree distribution when $\pr(\xi = 1) = \pr(\eta =\infty) = 0$. As before, we will phrase our results in terms of the transformed delay distribution $\eta$ instead of the original distribution $\xi$.  
The edge branching process construction described in Section~\ref{edgedesc} will play a major role both in the proof of these results. 

Recall the Poisson point process $\mvzeta_{\ve}$ with intensity measure $\mu_{\mvzeta_{\ve}}$ as in~\eqref{eqn:intensity-measure}. The following lemma relates the Laplace transform of the intensity measure to the moments of the delay random variable $\eta$.

\begin{lemma}\label{lapmom}
Assume $\pr(\eta =\infty) = 0$. For any $\theta>0$, we have
$$
\int_0^\infty e^{-\theta t} \mu_{\mvzeta_{\ve}}(dt) = \frac{1}{\theta(2+ \alpha)}\E\left(e^{(1-\theta)\eta}\right).
$$
\end{lemma}

\begin{proof}
    By interchanging the order of integrals, we obtain
\begin{align*}
\int_0^\infty e^{-\theta t} \mu_{\mvzeta_{\ve}}(dt) &= \int_0^{\infty}e^s\int_s^{\infty} \frac{e^{-\theta t}}{2+\alpha}dt \,dF_\eta(s)\\
&= \frac{1}{\theta(2+ \alpha)}\int_0^{\infty}e^{(1-\theta)s}dF_\eta(s)
= \frac{1}{\theta(2+ \alpha)}\E\left(e^{(1-\theta)\eta}\right).
\end{align*}
\end{proof}

\begin{ass}[Malthusian rate and $\hat\mvzeta\log^+\hat\mvzeta$ condition]
\label{ass:xlogx}
     Let $\theta_0 := \inf\{\theta \ge 0: \E\left(e^{(1-\theta)\eta}\right) < \infty\}$. Then, we assume that $\lim_{\theta \downarrow \theta_0} \frac{1}{\theta(2+ \alpha)}\E\left(e^{(1-\theta)\eta}\right) > 1$ (which trivially holds if $\theta_0 = 0$).
     By Lemma \ref{lapmom}, this implies that there exists a unique `Malthusian rate' $\gl > 0$ solving the equation, 
    \begin{align}
        \label{eqn:malthus-edge}
        \int_0^\infty e^{-\gl t} \mu_{\mvzeta_{\ve}}(dt) = 1. 
    \end{align}
     With the above $\gl$, define the random variable $\hat\mvzeta_{\ve}(\gl) = \int_0^\infty e^{-\gl t} \mvzeta_{\ve}(dt)$. Assume $$\E(\hat\mvzeta_{\ve}(\gl)\log^+(\hat\mvzeta_{\ve}(\gl))) < \infty.$$
\end{ass}

The following lemma connects $\gl$ to the distribution of $\eta$ and consequently gives a range for $\gl$.

\begin{lemma}\label{lem:lambdarange}
Assume $\pr(\eta =\infty) = 0$. The parameter $\gl$, defined in Assumption~\ref{ass:xlogx}, satisfies
\begin{equation}\label{eq:lambdaeta}
\E\left(e^{(1-\gl)\eta}\right) = (2 + \alpha)\gl.
\end{equation}
In particular, $\gl \in (1/(2+\alpha), 1)$.
\end{lemma}

\begin{proof}
\eqref{eq:lambdaeta} follows from Lemma \ref{lapmom}. The range for $\gl$ follows from~\eqref{eq:lambdaeta} upon noting that the left-hand side in~\eqref{eq:lambdaeta} is decreasing as a function of $\gl$ and the right-hand side is increasing.
\end{proof}

 We will use $p_{\Ma, \alpha}(\geq k) = \sum_{l=k}^\infty p_{\Ma, \alpha}(l) $ for the tail of the degree distribution. 

\begin{thm}[Degree asymptotics]
    \label{thm:non-conden}
  Assume  $\pr(\xi = 1) = \pr(\eta =\infty) = 0$.
   
    \begin{enumeratea}
        \item Under Assumption~\ref{ass:xlogx}, there exists constant $C_1 > 0$ such that $p_{\Ma, \alpha}(\geq k) \geq C_{1} k^{-1/\gl}$. Since $\gl \in (1/(2+\alpha), 1)$, the degree exponent is {\bf strictly smaller} than the regime without delay (that is, the distribution tails are heavier). 
        \item Further assume there exists $l> 1/\gl$ such that $\E([\hat\mvzeta_{\ve}(\gl(1 \wedge \alpha^{-1}))]^l) < \infty$. Then, there exists constant $C_2 > 0$ such that $p_{\Ma, \alpha}(\geq k) \leq C_{2} k^{-1/\gl}$. 
    \end{enumeratea}
\end{thm}

The above result leads to the following special cases. 

\begin{cor}[Special cases]\label{cor:degree-tail-macro-non-cond}
For the following special cases of the delay distribution $\eta$: 
\begin{enumeratea}
    \item 	Suppose $\eta \in (a,b)$ i.e. in a bounded sub-interval of $\bR_+$. Then there exist (distribution dependent) constants $0 < C_1 < C_2 < \infty $ such that 
$ C_{1} k^{-1/\gl}\leq p_{\Ma, \alpha}(\geq k) \leq  C_{2} k^{-1/\gl} $ where $\gl \in (1/(2+\alpha), 1)$ is the unique solution of the equation in~\eqref{eqn:malthus-edge}.
\item Suppose $\eta \sim \exp(\theta)$ with $\theta > 0$. Then~\eqref{eq:lambdaeta} has a unique positive solution given by
\begin{equation}
\label{eqn:gltheta}
    \gl = \gl(\theta) = \frac{1-\theta + \sqrt{(1-\theta)^2 + 4\theta/(2+\alpha)}}{2}.
\end{equation}
and $p_{\Ma, \alpha}(\geq k) \ge C_{1} k^{-1/\gl}$.
Moreover, if either (i) $\alpha =0, \, \theta>0$ or (ii) $\alpha>0, \,\theta \ge 1$,
then
$p_{\Ma, \alpha}(\geq k) \leq  C_{2} k^{-1/\gl}$.
Here, $0 < C_1 < C_2 < \infty$ are constants depending on $\alpha, \theta$.
\end{enumeratea}
\end{cor}
\begin{rem}
\begin{enumeratea}
    \item In the setting (b) above, $\eta \sim \exp(\theta)$ corresponds to original delay $\xi$ having distribution $F_\xi(x) = 1- (1-x)^{\theta}$ for $x\in [0,1]$. When $\theta \to \infty$, the delay distribution converges weakly to the Dirac mass at zero. This is exhibited in the observation $\gl(\theta) \to 1/(2 + \alpha)$ as $\theta \to \infty$, that is, the degree exponent approaches $2 + \alpha$, namely, the exponent for the affine preferential attachment model without delay. 
    \item For $\alpha>0,\theta \in (0,1)$, to obtain a matching upper bound for the degree distribution, one needs to somehow verify the moment condition of Theorem~\ref{thm:non-conden}(b). We leave this as an open problem. 
\end{enumeratea}
\end{rem}

\begin{rem}[Condensation]\label{cond}
    In view of the above results, it is clear that there is a qualitative difference in the network asymptotics between the cases $\pr(\xi = 1) >0$ and $\pr(\xi = 1) = 0$. If $\pr(\xi = 1) = \pr(\eta =\infty) >0$ as in Corollary~\ref{cor:special-case-macro}, it is easy to see that the root degree satisfies $M(\rho, n) = \Theta_P(n)$. This phenomenon, where the root attracts a non-vanishing fraction of the incoming vertices as the network grows, also implies lack of uniform integrability for the empirical degree distribution as the network grows (``mass escapes to $\infty$''). This lack of uniform integrability, which is also called {\bf condensation}, is seen to be equivalent to $\E(D_{\Ma}) =\sum_{k=1}^\infty k p_{\Ma, \alpha}(k) <2$ (expectation under the limiting degree distribution is strictly less than the limit of the means under the empirical degree distribution). By Theorem~\ref{thm:macroscopic-local}(a), condensation occurs if and only if $\pr(\xi = 1) = \pr(\eta =\infty) >0$.

    In the context of the networks community, this is one example of condensation phenomena where one vertex reaches degree $\Theta_{\pr}(n)$. However, there are more fascinating mechanisms for mass to escape away from the bulk but with the maximal degree satisfying $\Theta_{\pr}(n)$. For various mechanisms for condensation phenomenon to occur see e.g.~\cite{bianconi2001bose,janson2012simply,dereich2017nonextensive,haslegrave2020condensation,betz2018shape,banerjee2022co,bhamidi2012twitter} and the references therein. 
\end{rem}}

The final result describes the scaling of the root degree (in the non-condensation regime), showing that the scaling exponent above for the degree also governs the evolution of more macroscopic functionals. The same scaling should carry over to the maximal degree in this setting, but we will leave this to future work. For notational simplicity, we only consider the case $\alpha=0$. 

\begin{thm}[Root degree asymptotics]\label{thm:macro-max-deg}
	Set $\alpha = 0$. Consider the root degree $M(\rho, n) $. Assume $\eta$ has a density $f_\eta$ such that $\sup_{u \ge 0} e^u f_\eta(u) < \infty$. The following hold under the Assumptions of Theorem~\ref{thm:non-conden} with $\lambda $ as in~\eqref{eqn:malthus-edge}: 
 \begin{enumeratea}
     \item{\bf Upper bound:} Assume $\E(e^{\gamma \eta})< \infty$ for all $\gamma>0$. Then there exists $C_1 \in (0,\infty)$ such that for any $A > 0$,
     \[\limsup_{n\to\infty} \pr(M(\rho, n) > An^{\lambda}) \leq \frac{C_1}{A^{\frac{1}{\lambda +1}}}.\]
     \item {\bf Lower bound:} Assume $\exists\ \gamma >1$ such that $\E(e^{\gamma \eta}) < \infty$. Then there exist finite positive constants  $C_1,C_2$ such that for any $x >0$,
     \[\limsup_{n \to \infty} \prob\left(M(\rho, n) < x n^{\lambda}\right) \le C_1 x^{C_2}. \]
 \end{enumeratea}
\end{thm}

\begin{rem}
    Section~\ref{sec:macro-max-deg-proof} describes an edge copying procedure in a similar vein as (but more involved than) Section~\ref{edgedesc}, which encodes the \textbf{global evolution} of the network process $\{\cT(n): n \ge 1\}$. The proof of Theorem~\ref{thm:macro-max-deg} relies on a detailed quantification of the error incurred in approximating this edge copying procedure by the edge branching process in Section~\ref{edgedesc}. This is achieved using renewal theoretic and coupling techniques. Besides its utility in root degree evolution, we believe that this approach is of independent interest and can be used to furnish asymptotics of a host of other global network functionals, e.g., root PageRank~\cite{banerjee2022pagerank}.
\end{rem}

\subsection{Related work}
We placed this work in the general area of dynamic network models in Section~\ref{sec:intro}. The goal of this section is to discuss research that directly influences this paper. Aldous's paper~\cite{aldous-fringe} heralded the now fundamental notion of local weak convergence to understand asymptotics for large discrete random structures; for more recent work on fringe distributions, see the survey~\cite{holmgren2017fringe}. Our proof techniques use stochastic approximation techniques to understand the historical evolution of subtrees, which was first used in the context of preferential attachment models (without delay) in~\cite{rudas2007random}. While there is limited work to date on models related to delays in the macroscopic regime, understanding probabilistic evolution schemes giving rise to condensation phenomenon in networks has spurred a lot of literature over the last few years starting with~\cite{bianconi2001bose}, see for example~\cite{dereich2017,borgs2007first,bhamidi2012twitter,banerjee2022co}. In the probability community, the study of dynamic network models in the family studied in this paper,  incorporating delay in their evolution,  is still largely in its infancy; the closest papers are~\cite{baccelli2019renewal,king2021fluid,dey2022asymptotic}.

\section{Proofs}
\label{sec:proofs-macro}
{ We start by proving local weak convergence in the macroscopic regime (Theorem~\ref{thm:macroscopic-local}) in Sections~\ref{sec:proof-macro-local} and~\ref{sec:cont-proof-lwc}. Section~\ref{sec:conden-proofs} contains proofs of Theorem~\ref{thm:macroscopic-local}(b) and (c). The results on characterization of the degree exponent, namely Theorem~\ref{thm:non-conden} and Corollary~\ref{cor:degree-tail-macro-non-cond}, are proved in Section \ref{tailexpsec}. The edge branching process described in Section~\ref{edgedesc} will play a central role in the proofs of all the results in Section~\ref{sec:conden-proofs} and Section \ref{tailexpsec}. Section~\ref{sec:macro-max-deg-proof} and Section \ref{lbsec} contain the proof of the root degree asymptotics (Theorem~\ref{thm:macro-max-deg}).}

\subsection{{ Proof of local weak convergence}: Theorem~\ref{thm:macroscopic-local}(a)}
\label{sec:proof-macro-local}
{ For the rest of this section, we emphasize that we work with the discrete time tree process. } For a vertex $n\geq 1$ and $k\geq 0$, define { the random variable} $T_k^n = \inf\set{\ell: \deg(n,\ell) = k+1}$ to be the first time vertex $n$ has its $k$-th child (and $T_0^n = n$) in the discrete time network process $\{\TT(j) : j \ge 1\}$. For $k\geq 1$, we label the $k$th vertex entering the \emph{fringe} of vertex $n$ (the progeny tree associated with vertex $n$) by $\rho_k^n$ i.e \begin{equation*}
    \rho_k^n = \inf\set{m> \rho_{k-1}^n: \text{ vertex }m \text{ attaches to one of } \rho_0^n,\rho_1^n,\dots, \rho_{k-1}^n}\end{equation*}and $\rho_0^n = n$. For $j\geq 0,k\geq 1$, let $T^n_{j,k} = T^{\rho_j^n}_k$ be the birth time of $k$th child of $j$th vertex in the fringe of vertex $n$ (note that $T^n_{0,k} = T^n_k$ and $T^n_{j,0} = \rho^n_j$), and  \begin{align*}
    \tau_{j,k}^n = \log\left(\frac{T^n_{j,k}}{T^n_{j,k-1}}\right)
\end{align*}be the inter-arrival time in $\log$-scale.
One can think of the above quantity as the `continuous time' (if we encode vertex arrivals by epochs of a Yule process) elapsed for the degree of the $j$th descendant of vertex $n$ to increase from $k-1$ to $k$. These form a point process of reproduction times of vertices in the fringe of vertex $n$. { At this stage, we emphasise that the values of $T_{j,k}$'s are whole numbers and this restricts admissible values of $\tau_{j,k}$'s.} The following { lemma} shows that the inter-arrival times for these point processes weakly converge to those of the point process described in Section~\ref{sec:lim-obj-macro} as $n \to \infty$. Further, the point processes associated with distinct vertices in the fringe become asymptotically independent.

\begin{prop}\label{prop:weak-conv-macro} 
Let $\left(\left(\sE_{j,k}\right)_{k\geq 1}: j\geq 0\right)$ be \emph{iid} copies of $\left(\sE_{k-1\leadsto k}\right)_{k\geq 1}$ defined in Section~\ref{sec:lim-obj-macro}. Then, as $n \to \infty$,\begin{align*}
 \left(\left(\tau_{j,k}^n\right)_{k\geq 1}: j\geq 0\right) \convd \left(\left(\sE_{j,k}\right)_{k\geq 1}: j\geq 0\right).
\end{align*}
\end{prop}

We now present an immediate corollary of Proposition~\ref{prop:weak-conv-macro}, which is the key component in the proof of Theorem~\ref{thm:macroscopic-local}(a). Let $U$ denote a $U[0,1]$ random variable independent of the network process $\set{\cT(n)}_{n\geq 1}$. 
\begin{cor}\label{cor:weak-conv-macro-uniform} Let $\set{\left(\sE_{j,k}\right)_{k\geq 1}:j\geq 0}$ be \emph{iid} copies of $\left(\sE_{k-1\leadsto k}\right)_{k\geq 1}$ and $T_1$ be an exponential random variable of rate $1$ independent of $\set{\left(\sE_{j,k}\right)_{k\geq 1}:j\geq 0}$. Then 
    \begin{equation*}
        \left(\log n - \log \lceil n U\rceil,\left(\tau_{j,k}^{\lceil nU \rceil}\right)_{k\geq 1}: j\geq 0 \right) \convd \left(T_1,\left(\sE_{j,k}\right)_{k\geq 1}:j\geq 0 \right).
    \end{equation*}
\end{cor} 
In the remainder of the section, we prove Theorem~\ref{thm:macroscopic-local}(a). The proof of Proposition~\ref{prop:weak-conv-macro} will be provided in the next section.

\begin{proof}[Proof of Theorem~\ref{thm:macroscopic-local}(a)]
{For any $t \in \mathbb{R}_+$ and $\mathbf{w} \in \mathbb{R}_+^{\mathbb{N}_0 \times \mathbb{N}}$, one can construct a unique rooted tree $\mathscr{T}(t,\mathbf{w}) \in \mathbb{T}$ via the following temporal evolution: (a) the root $\emptyset$ is born at time $0$, (b) for any $j \ge 0, k \ge 1$, $w_{j,k}$ denotes the age of the $j$th born vertex in the tree when its $k$th child is born (with root being the $0$th born vertex), (c) retain only those vertices with birth times $\le t$. For $0 \le j < |\mathscr{T}(t,\mathbf{w})|$, denote by $\tau_j(t,\mathbf{w})$ the birth time of the $j$th born vertex in $\mathscr{T}(t,\mathbf{w})$, and let $n(t,\mathbf{w}) := \sup\{j \ge 0 : \tau_j(t,\mathbf{w}) \le t\}$. Define $\mathscr{W} := \{(t,\mathbf{w}) \in  \mathbb{R}_+ \times \mathbb{R}_+^{\mathbb{N}_0 \times \mathbb{N}} : n(t,\mathbf{w}) < \infty, \, \tau_j(t,\mathbf{w}) \neq t \text{ for any } 0 \le j \le n(t,\mathbf{w})\}$. Then, it is clear that $\mathscr{T}: \mathscr{W} \rightarrow \mathbb{T}$ is continuous.}

{Now, consider the evolution of the network process $\set{\cT(n)}_{n\geq 1}$ in the logarithmic time scale, namely, assign the birth time of $0$ to $v_0,v_1$ and $\log j$ to $v_j$ for $j \ge 2$. For any $n \ge 2$ and $2 \le \ell \le n$, the \emph{fringe} of vertex $\ell$ in $\cT(n)$ can be expressed as $\mathscr{T}(\log n - \log \ell, \mathbf{w}^{(\ell)})$, where $w^{(\ell)}_{j,k} =\tau_{j,k}^\ell = \log\left(\frac{T^\ell_{j,k}}{T^\ell_{j,k-1}}\right)$, $j \ge 0, k \ge 1$. In particular, the fringe of a uniformly chosen vertex is given by
$\mathscr{T}(\log n - \log \lceil n U\rceil, \mathbf{w}^{(\lceil n U\rceil)})$. By Corollary~\ref{cor:weak-conv-macro-uniform}, $(\log n - \log \lceil n U\rceil, \mathbf{w}^{(\lceil n U\rceil)}) \convd (T_1, \mathbf{w}^*)$ where $w^*_{j,k} = \sE_{j,k}$, $j \ge 0, k \ge 1$. Moreover, using Lemma~\ref{lem:hazard-rate-sum} (also, by Proposition~\ref{prop:edge-equiv-bp}), $\prob\left((T_1, \mathbf{w}^*) \in \mathscr{W}^c\right)=0$.
Hence, by the continuous mapping theorem, $\mathscr{T}(\log n - \log \lceil n U\rceil, \mathbf{w}^{(\lceil n U\rceil)}) \convd \mathscr{T}(T_1, \mathbf{w}^*)$, which has the law $\fpm_{\Ma,\alpha}$. This implies the local weak convergence in the expected fringe sense.}

The convergence of the (expected) empirical degree distribution is an immediate consequence of this local weak convergence.
\end{proof}

\subsection{Proof of Proposition~\ref{prop:weak-conv-macro}}
\label{sec:cont-proof-lwc}
This section is dedicated to the proof of Proposition~\ref{prop:weak-conv-macro}. { Throughout the section, we fix $K \in \mathbb{N}$. To prove the proposition it is enough to prove \begin{align*}
        \tau^n_{1:K} := \left(\tau^n_{j,k}:0\leq j\leq K, 1 \leq k\leq K\right) \convd \sE_{1:K} := \left(\sE_{j,k}:0\leq j\leq K, 1 \leq k\leq K\right). 
    \end{align*}To this extent, we show 
    \begin{align}\label{eqn:exp-convergence}
        \lim\limits_{n\to \infty}\E\left(f(\tau^n_{1:K})\right) = \E\left(f(\sE_{1:K})\right).
    \end{align} where $f:\bR_+^{K^2+K}\to \bR$ is a continuous function with compact support $\cS \subseteq [0,L]^{K^2+K}$ for some $L \geq 0$.} 
    
    { We divide the proof into four steps \textbf{(i)} we identify the joint density of the random variables $\sE_{1:K}$, \textbf{(ii)} we obtain a tractable expression for the expectation in the LHS of \eqref{eqn:exp-convergence}, \textbf{(iii)} we make several approximations to the obtained expression and \textbf{(iv)} using the approximations in step (iii) we complete the proof. }

\noindent { \textbf{Step$-1:$} We start by computing the joint density of $\sE_{K}= (\sE_{0\leadsto 1},\sE_{1 \leadsto 2},\dots,\sE_{K-1\leadsto K})$.} Recall that $h(x) = \int_0^x\frac{e^u}{2+\alpha}dF_\eta(u)$ and $H(x) = \int_0^x h(y)dy = \int_0^x \frac{e^u}{2+\alpha} (x-u) dF_\eta(u)$. For any $\mvx\in \bR^K_+$ and $1 \le k \le K$, {let \begin{align}\label{hkdef}
    h_{k}(\mvx) &= \sum_{j=0}^{k-1} h(\sigma_{k}(\mvx) -\sigma_j(\mvx)) + \alpha h(\sigma_{k}(\mvx)),\notag\\
    H_{k}(\mvx) &= \sum_{j=0}^{k-1} H(\sigma_{k}(\mvx) -\sigma_j(\mvx)) + \alpha H(\sigma_{k}(\mvx)),
\end{align} where $\sigma_i(\mvx) = \sum\limits_{l=1}^i x_l$ and $\sigma_0(\mvx) = 0$.} 
Using these functions and the hazard rate description for the inter-arrival times in Section~\ref{sec:lim-obj-macro}, we conclude that the joint density of $\sE_{K}= (\sE_{0\leadsto 1},\sE_{1 \leadsto 2},\dots,\sE_{K-1\leadsto K})$ is given by
\begin{align}\label{eqn:joint-density}
    \fF_K(\mvx) =  \left(\prod\limits_{i=1}^K h_i(\mvx)\right)\exp\left(-H_K(\mvx)\right), \ \mvx \in \bb{R}_+^K.
\end{align} 
Note that $h(x)$ and $H(x)$ are { are non-decreasing functions and therefore discontinuous at most at countably many points.} Hence, the density $\fF_K$ is continuous almost surely on $\bR_+^K$ w.r.t.~Lebesgue measure. 

\noindent { \textbf{Step-$2$:}
We next express $\E\left(f(\tau^n_{1:K})\right)$ as an integral over $\bR_+^{K^2+K}$. We begin by developing some notation to express the distribution of $\tau^n_{1:K}$.}

    Recall $(\rho_j^n:j\geq 0)$ and $(T_{j,k}^n: j,k\geq 0)$ defined at the beginning of Section~\ref{sec:proof-macro-local}. Observe that the values of $(\rho_j^n:0\leq j\leq K)$ are determined by the values of $(T_{j,k}-T_{j,k-1}: 0\leq j\leq K, 1 \leq k \leq K)$. Thus, for $\mvx \in \bR^{K^2+K}_+$, one can define $(\rho_j^n(\mvx):0\leq j\leq K)$, the values of $(\rho^n_j:0\leq j\leq K)$ on the event $E_{\mvx} = \set{T^n_{j,k}-T^n_{j,k-1} = \lfloor x_{j,k}\rfloor :0\leq j\leq K, 1\leq k\leq K}$, as follows:
    \begin{itemize}
    \item {$\rho_0^n(\mvx) = T^n_{0,0}(\mvx) = n$} and $T^n_{0,k}(\mvx) = \rho_0^n(\mvx) + \sum\limits_{l=1}^k \lfloor x_{0,l}\rfloor $ for $1 \le k\leq K$.
    \item Given $\set{T_{j,k}^n(\mvx):j\leq i,1\leq k\leq K}$ for $0 \le i < K$, define \begin{align*}
        \rho^n_{i+1}(\mvx) &= \text{the }(i+1)^{th} \text{ smallest value in }\set{T_{j,k}^n(\mvx):0\leq j\leq i,1\leq k\leq K},\\
        T_{i+1,k}^n(\mvx) &=  \rho^n_{i+1}(\mvx) + \sum\limits_{l=1}^k \lfloor x_{i+1,l}\rfloor, \ {T^n_{i+1,0} = \rho^n_{i+1}(\mvx)}.
    \end{align*}
\end{itemize} 
Also, define $\tau^n_{j,k}(\mvx) = \log\left(\frac{T_{j,k}^n(\mvx)}{T_{j,k-1}^n(\mvx)}\right)$ and $\tau^n_{1:K}(\mvx) = \left(\tau^n_{j,k}(\mvx):0\leq j\leq K, k\leq K\right)$ - the values of inter-arrival times in $\log$-scale on the event $E_{\mvx}$. { We also note that with the above definitions, one can also view $\rho_{i}^n(\mvx)$ and $T^n_{i,k}(\mvx)$ as maps from $\mvx \in \bR_+^{K^2+K} \to \bN$.}

 On the event $E_{\mvx}$, an incoming vertex { attaches to a vertex in the fringe of vertex $n$, observed by the incoming vertex at that stage, only at times} $
    \cA_{\mvx} = \set{T^n_{j,k}(\mvx):0\leq j\leq K, 1\leq k\leq K }$. Also, let $
    N_{\mvx} =\sup \cA_{\mvx}$ and $
    \cI_{\mvx} = \set{n+1,n+2,\dots,N_{\mvx}}\setminus\cA_{\mvx}$ be the times when the fringe is inactive, meaning there are no attachments.
    
In the discrete-time network process $\set{\cT(n):n\geq 1}$, an incoming vertex can only attach to a single existing vertex in the network. So, we say $\mvx \in \bR^{K^2+K}$ is \emph{valid} if all $T_{j,k}^n(\mvx)$ are distinct for $0\leq j\leq K$ and $1\leq k\leq K$. 

For $0\leq j\leq K$ and $i \in \mathbb{N}$, { let $p^n_j(i,\mvx)$ denote the probability that an incoming vertex at time $i$ attaches to the $j^{th}$ vertex in the fringe of vertex $n$ given the history up to time $i-1$ on the event $\set{T^n_{j,k}-T^n_{j,k-1} = \lfloor x_{j,k}\rfloor \  \forall \ k \text{ such that } T_{j,k}^n(\mvx) < i}$. { The condition $T_{j,k}^n(\mvx) < i$ is needed to ensure that vertex labeled $T_{j,k}^n$ which will be part of the fringe eventually has arrived before time $i$, as the vertex $i$ can only attach to a vertex that has arrived before it.} Explicitly, we have  }
\begin{align}\label{eqn:prob-step}
\begin{split}
    p^n_j(i,\mvx) &= \frac{\ind\set{i\leq T_{j,K}^n(\mvx)}}{2+\alpha}
        \left[\sum\limits_{k =0}^K \E\left(\frac{\ind\set{i(1-\xi) \geq T_{j,k}^n(\mvx)}}{\lfloor i(1-\xi)\rfloor}\right) \right.\\
        &\qquad\qquad\qquad\qquad\left.+ \alpha \E\left(\frac{\ind\set{i(1-\xi)\geq \rho^n_{j}(\mvx)}}{\lfloor i(1-\xi)\rfloor}\right)\right].
        \end{split}
\end{align} { In the right hand side of the above equation, the summation corresponds to the degree of vertex $j$ the incoming vertex $i$ observes. Furthermore,
 note that the sum above runs over $k$ such that $T_{j,k}^n(\mvx) < i$, and hence the right-hand side depends only on the network run up to time $i-1$.}
We then have 
\begin{align}
\label{eqn:prob-Ex}
    \vP_{0,n}(\mvx) = \pr(E_{\mvx}) &{= \ind\set{\mvx \text{ is \emph{valid}}} \prod_{i\in \cI_{\mvx}} \left(1-\sum\limits_{j :\rho^n_j(\mvx) <i} p_j^n(i,\mvx)\right) . \prod\limits_{j,k\leq K}p_{j}^n\left(T_{j,k}(\mvx),\mvx\right)}\notag\\
    &= \ind\set{\mvx \text{ is \emph{valid}}} \prod_{i\in \cI_{\mvx}} \left(1-\sum\limits_{j=0}^K p_j^n(i,\mvx)\right) . \prod\limits_{j,k\leq K}p_{j}^n\left(T_{j,k}(\mvx),\mvx\right).
\end{align} { The last equality follows from the fact that, on the event $E_{\mvx}$, at most $K$ vertices are in the fringe of vertex $n$ and $p_j^n(i,\mvx) = 0$ for $j \geq i$.

At this stage, also note that for a discrete random variable $\mvX$ taking integer values, for any continuous function $g$ with compact support, we have $$\E(g(\mvX)) = \sum_{m\in \bZ}g(m) \pr(\mvX = m)= \int_{\bR}g(\lfloor x \rfloor)\pr(\mvX = \lfloor x \rfloor)dx.$$ One can extend the above to discrete random vectors taking values in $\bN^{K^2 + K}$, and thus by ~\eqref{eqn:prob-Ex}. we have $$
    \E(f(\tau^n_{1:K})) = \int_{\bR_+^{K^2 + K}} f(\tau^n_{1:K}(\mvx)) P_{0,n}(\mvx)d\mvx. $$}

\noindent { \textbf{Step$-3$:} We now approximate $\vP_{0,n}(\mvx)$ using the following sequence of approximations.} \begin{enumerate}
    \item We replace the terms in the first product in~\eqref{eqn:prob-Ex} by exponentials to get $\vP_{1,n}(\mvx)$ where \begin{align}\label{eqn:prob-ex-1}
    \vP_{1,n}(\mvx) = \ind\set{\mvx \text{ is \emph{valid}}} \prod\limits_{i\in \cI_{\mvx}} \exp\left(-\sum\limits_{j=0}^K p_j^n(i,\mvx)\right)  . \prod\limits_{j,k\leq K}p_{j}^n\left(T_{j,k}^n(\mvx),\mvx\right).
\end{align}
\item We next replace the $\lfloor i(1-\xi)\rfloor$ with $i(1-\xi)$ in~\eqref{eqn:prob-step} to get $\vP_{2,n}(\mvx)$ with \begin{align}\label{eqn:prob-ex-2}
\vP_{2,n}(\mvx) = \ind\set{\mvx \text{ is \emph{valid}}} \prod\limits_{i\in \cI_{\mvx}} \exp\left(-\sum\limits_{j=0}^K q_j^n(i,\mvx)\right)  \cdot \prod\limits_{j,k\leq K}q_{j}^n\left(T_{j,k}^n(\mvx),\mvx\right),
\end{align} 
where 
\begin{align}\label{eqn:q_j-defn}
\begin{split}
    q^n_j(i,\mvx) &= \frac{\ind\set{i\leq T_{j,K}^n(\mvx)}}{2+\alpha}
        \left[\sum\limits_{k=0}^K \E\left(\frac{\ind\set{i(1-\xi) \geq T_{j,k}^n(\mvx)}}{i(1-\xi)}\right) \right.\\
        &\qquad\qquad\qquad\qquad\qquad\qquad\left.+ \alpha \E\left(\frac{\ind\set{i(1-\xi)\geq \rho^n_{j}(\mvx)}}{ i(1-\xi)}\right)\right].
        \end{split}
\end{align}

\item  Finally, we approximate the sums of type $\sum\limits_{i=m_1}^{m_2}\frac{1}{i}$ with $\log\frac{m_2}{m_1}$ appearing in the exponent of the first term in~\eqref{eqn:prob-ex-2} to get $\vP_{3,n}(\mvx)$ where \begin{align}\label{eqn:prob-ex-3}
        \vP_{3,n}(\mvx) = \ind\set{\mvx \text{ is \emph{valid}}} \exp\left(-\sum\limits_{j=0}^K r_j^n(\mvx)\right)  . \prod_{j=0}^{K}\prod_{k=1}^K q_{j}^n\left(T_{j,k}^n(\mvx),\mvx\right).
    \end{align} where \begin{align}\label{eqn:r_j-defn}
        r^n_j(\mvx) = \sum_{k=0}^K H\left(\log\left(\frac{T_{j,K}^n(\mvx)}{T_{j,k}^n(\mvx)}\right)\right)+ \alpha H\left(\log\left(\frac{T_{j,K}^n(\mvx)}{\rho_j^n(\mvx)}\right)\right).
    \end{align} 
\end{enumerate} An upper bound on errors made by each approximation is given by the following lemma, whose proof is given in Appendix~\ref{app}. 
\begin{lem}\label{lem:bigggg-lemma}Let $\cC$ be a bounded subset in $R_{+}^{K^2+K}$. Then we have 
\begin{align*}
    \sup_{i=0,1,2} \, \sup_{\mvx \in \cC}\left|\vP_{i,n}(\mvx) - \vP_{i+1,n}(\mvx) \right|\leq  \frac{CK^4}{n^2} \left[n+\sup\limits_{\mvx\in\cC}\left(\sum\limits_{j=0}^K \sum\limits_{k=1}^K x_{j,k}\right)\right]
\end{align*}for some constant $C \geq 0$ independent of the set $\cC$, $K$ and $n$.
\end{lem}

\noindent{ {\bf Step-4:} Using all the above observations, we now complete the proof. Observe that by Lemma~\ref{lem:bigggg-lemma}, we have \begin{align*}
    \left|\E(f(\tau^n_{1:K})) - \int_{\bR_+^{K^2+K}}f(\tau^n_{1:K}(\mvx))\vP_{3,n}(\mvx)\right| \leq \frac{C}{n}
\end{align*} for some constant $C \geq 0$ {depending only on $K$ and $L$ where the support of $f$ is $\mathcal{S} \subseteq[0,L]^{K^2 + K}$}.} Thus, to prove~\eqref{eqn:exp-convergence}, it is sufficient to show \begin{align}\label{eqn:sufficient-condition}
    \lim\limits_{n\to \infty}\int_{\bR_+^{K^2+K}} f(\tau^n_{1:K}(\mvx)) \vP_{3,n}(\mvx)d\mvx  = \E\left(f(\sE_{1:K})\right).
\end{align}
To compute the integral on LHS of~\eqref{eqn:sufficient-condition}, consider the transformation where the inter-arrival times are in exponential scale where we make the change of variables $\mvx \to \mvy$ with { $y_{j,k} = \tau^n_{j,k}(\mvx) = \log\left(\frac{T_{j,k}^n(\mvx)}{T_{j,k-1}^n(\mvx)}\right), \ 0 \le j \le K, 1 \le k \le K$}. More precisely, for $\mvy \in \bR_+^{K^2 + K}$, we define the following:\begin{itemize}
    \item Let $\hat{\rho}^n_0(\mvy) = n$.
    \item For $i < K$, given $\set{\hat{\rho}^n_j(\mvy):0\leq j\leq i}$, let $\hat{\rho}^n_{i+1}(\mvy)$ to be the value of $\rho^n_{i+1}(\mvy)$ when
    \[
    \left\{
    \begin{aligned}
    &T_{j,k}^n-T_{j,k-1}^n
    = \left\lfloor \hat{\rho}^n_j(\mvy)
    \exp\left(\sum\limits_{l=1}^{k-1}y_{j,l}\right)
    \cdot\left(\exp(y_{j,k})-1\right) \right\rfloor :\\
    &\hspace{7cm}0\leq j\leq i,\ 1\leq k\leq K
    \end{aligned}
    \right\}.
    \]
\end{itemize}  
We define $G_n: \mathbb{R}^{K^2 + K} \rightarrow \mathbb{R}^{K^2 + K}$ by $$G_n(\mvy) = \left(\hat{\rho}^n_j(\mvy)\exp\left(\sum\limits_{l=1}^{k-1}y_{j,l}\right)\cdot \left(\exp(y_{j,k})-1\right): 0\leq j\leq K, 1\leq k\leq K \right).$$ { As noted above, the function naturally appears when one transforms the arrival times to exponential scale. Observe that the coordinates of $G_n(y)$ are the inter-arrival times without the floor function.} We compute the integral on LHS of~\eqref{eqn:sufficient-condition} after applying the transformation $\mvx \mapsto G_n(\mvy)$ {(which, as we will see, is the `approximate' inverse of the transformation $\mvy \mapsto \tau^n_{1:K}(\mvx)$)}. Note that the function $G_n(\mvy)$ is piece-wise differentiable and bijective. {Moreover, the $(K^2 + K) \times (K^2 + K)$ matrix of partial derivatives of $\mvy \mapsto G_n(\mvy)$} is lower-triangular (when it is defined). Therefore, the Jacobian of the transformation is \begin{align}\label{Jac}
    \fJ_n(\mvy) = \prod_{j=0}^{K}\prod_{k=1}^K \left[\hat{\rho}^n_j(\mvy)\exp\left({\sum_{i=1}^{k}y_{j,i}}\right)\right].
    \end{align} 
    Let $\hat{T}^n_{j,k}(\mvy) = T^n_{j,k}(G_n(\mvy))$ and $\hat{\tau}^n_{j,k}(\mvy)=\tau^n_{j,k}(G_n(\mvy))$. Then after the transformation, the integral on the LHS of~\eqref{eqn:sufficient-condition} is\begin{align*}
        \int_{\bR_+^{K^2+K}} f(\tau^n_{1:K}(\mvx)) \vP_{3,n}(\mvx)d\mvx = \int_{\bR_+^{K^2+K}} f(\hat{\tau}^n_{1:K}(\mvy)) \vP_{3,n}(G_n(\mvy)) \fJ_n(\mvy)d\mvy. 
    \end{align*} 

    { Recall the definition of $\pr_{3,n}$ from \eqref{eqn:prob-ex-3}. Also, from the definition of $r^n_j$ in~\eqref{eqn:r_j-defn} and $H_K$ in~\eqref{hkdef}, we have $
    r^n_j(G_n(\mvy)) = H_K(\hat{\tau}_j^n(\mvy))$. Since $\hat{\rho}^n_j(\mvy) \geq n$, we have used the bounds on $\hat{T}^n_{j,k}(\mvy)$ obtained above that $\lim\limits_{n\to \infty} \hat{\tau}^n_{j,k}(\mvy) = y_{j,k}$. 
    
    Also, recalling $q^n_j(i,\mvx)$ defined in~\eqref{eqn:q_j-defn} and $\eta = - \log (1- \xi)$, we have for any $0 \le j \le K, 1 \le k \le K$,
\begin{align*}
    &q^n_j(\hat{T}^n_{j,k}(\mvy),G_n(\mvy)) \\
    &= \frac{1}{\hat{T}^n_{j,k}(\mvy)}\left[\sum_{l=0}^{k-1}\int_0^{\log \left(\frac{\hat{T}^n_{j,k}(\mvy)}{\hat{T}^n_{j,l}(\mvy)}\right)}\frac{e^u}{2 + \alpha}dF_\eta(u) + \alpha \int_0^{\log \left(\frac{\hat{T}^n_{j,k}(\mvy)}{\hat{T}^n_{j,0}(\mvy)}\right)}\frac{e^u}{2 + \alpha}dF_\eta(u) \right]\\
    &= \frac{h_k\left(\hat{\tau}_j^n(\mvy)\right)}{\hat{T}^n_{j,k}(\mvy)}
\end{align*}
where $\hat{\tau}^n_{j}(\mvy) := (\hat{\tau}^n_{j,k}(\mvy): 1\leq k\leq K)$, and we recall $h_k$ from~\eqref{hkdef}. Combining this with~\eqref{Jac}, we obtain
\begin{align}\label{eqn:bound-for-dct}
\begin{split}
&\left[\prod\limits_{j=0}^K\prod\limits_{k=1}^K q^n_j(\hat{T}^n_{j,k}(\mvy),G_n(\mvy))\right] \fJ_n(\mvy) \\
&\qquad= \left(\prod_{j=0}^{K}\prod_{k=1}^K h_k(\hat{\tau}^n_{j}(\mvy))\right) \left(\prod_{j=0}^{K}\prod_{k=1}^K \frac{\hat{\rho}^n_j(\mvy)\exp\left({\sum_{i=1}^{k}y_{j,i}}\right)}{\hat{T}_{j,k}^n(\mvy)}\right).
\end{split}
\end{align} Using this along with~\eqref{eqn:bound-for-dct}, we thus obtain
    \begin{align}
        \lim\limits_{n\to\infty} \prod\limits_{j=0}^K \prod\limits_{k=1}^K \left[q^n_j(\hat{T}^n_{j,k}(\mvy), G_n(\mvy))\right] \fJ_n(\mvy) &= \prod_{j=0}^{K}\prod_{k=1}^K h_k(\mvy_j)\label{eqn:pointwise-limit-1},\\
        \lim\limits_{n\to \infty}r_j(G_n(\mvy)) &= H_K(\mvy_j)\label{eqn:pointwise-limit-2},
    \end{align} where $\mvy_j := (y_{j,k} : 1 \le k \le K)$.  Furthermore, by~\eqref{eqn:pointwise-limit-1},~\eqref{eqn:pointwise-limit-2} and definition of $P_{3,n}(\mvx)$ in~\eqref{eqn:prob-ex-3}, we have on the set $\set{\mvy : G_n(\mvy) \text{ is } \emph{valid}} \cap [0, \log\left(e^L + K\right)]^{K^2+K}$, 
\begin{align*}
 \lim\limits_{n\to \infty}\vP_{3,n}(G_n(\mvy)) \fJ_n(\mvy) =  \prod_{j=0}^K \left[ \exp\left(-H_K(\mvy_j)\right)\prod_{k=1}^K h_k(\mvy_j)\right].
\end{align*} The above observations are useful in handling the integral over the region where $\mvx$ is valid. We next argue that the contribution to the integral from invalid points goes to zero as $n\to \infty$ in the following lemma by showing (i) integrand is uniformly bounded, and (ii) the Lebesgue measure of invalid points goes to zero as $n\to \infty.$ We state this in the following lemma and proof is provided in the Appendix \ref{app}.

\begin{lem}\label{lem: invalid-set}
The following hold \begin{enumerate}
\item $\vP_{3,n}(G_n(\mvy)) \fJ_n(\mvy)$ is uniformly bounded by on when $\hat{\tau}^n_{1:K}(\mvy) \in \cS \subseteq [0,L]^{K^2+ K}$ (the support of $f$)
    \item Let $\lambda$ be the Lebesgue measure on $R^{K^2+K}$. Then, as $n\to \infty$, \begin{align*}
        \lambda(\set{\mvy: G_n(\mvy) \text{ is not \emph{valid}}} \cap \cL) \to 0 
    \end{align*} where $\cL = [0,\log(e^L+K)]^{K^2+K}$.
\end{enumerate}
\end{lem}}
 
Finally, using the above lemma and by DCT, we have \begin{align*}
    \lim\limits_{n\to \infty}\int_{\bR_+^{K^2+K}} f(\tau^n_{1:K}(\mvx)) \vP_{3,n}(\mvx)d\mvx &=    \lim\limits_{n\to \infty}\int_{\bR_+^{K^2+K}} f(\hat{\tau}^n_{1:K}(\mvy)) \vP_{3,n}(G_n(\mvy)) \fJ_n(\mvy) d\mvy\\
    &=\int_{\bR_+^{K^2+K}} f(\mvy) \prod_{j=0}^K \left[ \exp\left(-H_K(\mvy_j)\right)\prod_{k=1}^K h_k(\mvy_j)\right]d\mvy\\
    &=\E(f(\sE_{1:K})).
\end{align*} The last line follows from the expression for joint density of $(\sE_{0\leadsto 1},\sE_{1\leadsto 2},\dots,\sE_{k-1\leadsto k})$ in~\eqref{eqn:joint-density}.  This finishes the proof.\qed\\

\noindent {\bf Proof of Corollary~\ref{cor:special-case-macro}:} We now briefly describe the proof of the special case considered in this corollary where the delay has a Bernoulli distribution. The corresponding $\eta$ transform is given by $\eta = p\delta_{0} + (1-p) \delta_{\infty}$. Thus, using Lemma~\ref{lem:hazard-rate-sum}, the hazard for $\cE_{k-1 \leadsto k}$ is given as
\[h_k(x) = \frac{k}{2} \E(e^{\eta} \ind\set{\eta \leq x}) = \frac{kp}{2}. \]
Thus $\cE_{k-1 \leadsto k} \sim \exp(kp/2)$, independent across different $k$. Thus the tail probability, {
\begin{align*}
p_{\Ma \geq }(k) = \pr(\sum_{j=1}^k \cE_{j-1 \leadsto k} \leq T_1 ) &= \E(\exp(-\sum_{j=1}^k \cE_{j-1 \leadsto k})) = \prod_{j=1}^k \frac{kp}{kp+2} \end{align*}
 where the last line follows from properties of exponential distribution and independence of random variables $\cE_{k-1 \leadsto k}$. The result finally follows from the following approximation.$$\prod_{j=1}^k \frac{jp}{jp+2} = \exp\left(\sum_{j=1}^k \log\left(1-\frac{2}{jp}\right)\right)\approx \exp\left(-\frac{2}{p}\sum_{j=1}^k \frac{1}{j}\right)\approx \exp\left(-\frac{2}{p}\log k\right) = k^{-\frac{2}{p}}. $$ }
\qed\\

\subsection{Proofs of Theorem~\ref{thm:macroscopic-local}(b) and (c)}
\label{sec:conden-proofs}

We will start by using Proposition~\ref{prop:edge-equiv-bp} to prove the following. 
\begin{lem}
    \label{lem:mean-dma}
    In the setting of Theorem~\ref{thm:macroscopic-local}, the limit degree distribution $D_{\Ma}$ satisfies 
    \[\E(D_{\Ma}) = \frac{(2+ \alpha) + \alpha \pr(\eta < \infty)}{2 + \alpha-\pr(\eta < \infty)}.\]
\end{lem}

\begin{proof}
    By Prop.~\ref{prop:edge-equiv-bp}, 
    \begin{align}\label{dexp}
    \begin{split}
        \E(D_{\Ma}) &= \E(|\BP_{\ve}(T_1)|) = \E\left(\sum_{i=0}^{\chi(T_1)-1} |\BP^\circ_{\ve,i}(T_1 - \theta_i)| \right)\\
        & = \E\left(\sum_{i=0}^{\infty}\ind\set{T_1 \ge \theta_i}|\BP^\circ_{\ve,i}(T_1 - \theta_i)| \right)\\
        &= \left(\sum_{i=0}^{\infty}\pr(T_1 \ge \theta_i)\right)\E\left(|\BP^\circ_{\ve}(T_1)|\right) = \E\left(\chi(T_1)\right)\E\left(|\BP^\circ_{\ve}(T_1)|\right),
            \end{split}
    \end{align}
    where the fourth equality is a consequence of the memoryless property of exponential random variables, { more precisely, for any $i \ge 0$, the conditional distribution of $T_1 - \theta_i$ given $T_1> \theta_i$ is $\exp(1)$.}
    
    Define $m(t) = \E(|\BP^{\circ}_{\ve}(t)|)$. Using the forward decomposition of the branching process based on the reproduction times of the offspring of the root and their progeny sizes, one gets the standard renewal equation $m(t) = 1+\int_0^t m(t-s) \mu_{\mvzeta_{\ve}}(ds)$. Thus writing  $\fm(t) := e^{-t} m(t)$, the above identity gives 
    $\fm(t) = e^{-t} + \int_0^t \fm(t-s) e^{-s} \mu_{\mvzeta_{\ve}}(ds)$
 and hence, 
    \begin{align*}
        \int_0^\infty \fm(t) dt &= 1 + \int_0^{\infty}\int_0^t \fm(t-s) e^{-s} \mu_{\mvzeta_{\ve}}(ds)dt = 1 + \int_0^{\infty}\left(\int_s^\infty \fm(t-s)dt\right) e^{-s} \mu_{\mvzeta_{\ve}}(ds)\nonumber\\
        &= 1+ \int_0^\infty \fm(t) dt  \cdot \int_0^\infty e^{-s} \mu_{\mvzeta_{\ve}}(ds) \Rightarrow  \int_0^\infty \fm(t) dt = \frac{1}{1-\left(\int_0^\infty e^{-s} d\mu_{\mvzeta_{\ve}}(ds)\right)}.
    \end{align*}
    Now the term in the denominator, 
    \begin{align}\label{intrep}
        \int_0^\infty e^{-s}d\mu_{\mvzeta_{\ve}}(ds) 
        &= \frac{1}{2 + \alpha}\int_0^\infty e^{-s}\int_0^s e^udF_\eta(u) ds\\
        & = \frac{1}{2 + \alpha} \int_0^\infty \bigg(\int_u^\infty e^{-s} ds \bigg) e^udF_\eta(u) = \frac{1}{2+\alpha}\int_0^\infty dF_\eta(u) = \frac{\pr(\eta < \infty)}{2 + \alpha},
    \end{align}
    and hence,
    \begin{equation}\label{bexp}
        \E\left(|\BP^\circ_{\ve}(T_1)|\right) = \int_0^\infty \fm(t) dt = \frac{2 + \alpha}{2 + \alpha - \pr(\eta < \infty)}.
    \end{equation}
    Moreover, as $\chi(\cdot)$ has rate $r_{\chi}(t) := \frac{\alpha}{2 + \alpha} r(t), t \ge 0,$
    $$
    \E\left(\chi(T_1)\right) = 1 + \int_0^\infty \frac{\alpha}{2 + \alpha}e^{-t}\int_0^t e^u dF_\eta(u)dt = 1 + \frac{\alpha}{2 + \alpha} \pr(\eta < \infty), 
    $$
    using~\eqref{intrep}.
    Using this and~\eqref{bexp} in~\eqref{dexp} completes the proof. 
\end{proof}
{\noindent{\bf Completing the proof of Theorem~\ref{thm:macroscopic-local}(b) and (c):} Theorem~\ref{thm:macroscopic-local}(b) is proved in Lemma~\ref{lem:mean-dma}. By Theorem~\ref{thm:macroscopic-local}(b), if $\pr(\eta =\infty) = 0$, then the mean number of children of the root in the original branching process $\BP_{\Ma}(T_1)$ is 
\[
\E(\Zma(T_1)) = \E(D_{\Ma})-1  = 1.
\] 
By~\cite[Proposition 3]{aldous-fringe}, $\varpi_{\Ma}$ is a fringe distribution in the sense of Definition~\ref{fringedef}. Moreover, by~\cite[Theorem 13]{aldous-fringe} (see also the paragraph following the theorem), $\varpi_{\Ma}$ is an extremal fringe distribution. Theorem~\ref{thm:macroscopic-local}(c) now follows from Theorem~\ref{thm:aldous-efr-pfr} and Lemma~\ref{ftoeflemma}.}

\subsection{Tail exponents: Proof of Theorem~\ref{thm:non-conden} and Corollary~\ref{cor:degree-tail-macro-non-cond}}\label{tailexpsec}  
Recall again, from Prop.~\ref{prop:edge-equiv-bp}, the characterization of the limit degree $D_\Ma$ in terms of the total size of the edge branching process at a random time $T_1$ i.e.  $|\BP_{\ve}(T_1)|$. We will use this to quantify the behavior of the tail exponents.  We first start with some implications of the assumptions of (a) and (b) of Theorem~\ref{thm:non-conden}.

\begin{lem}
    \label{lem:bpe-moment}
    Consider the edge branching process $\BP_{\ve}$ with offspring distribution $\mvzeta_{\ve}$.
    \begin{enumeratea}
        \item Under Assumption~\ref{ass:xlogx}, with $\lambda$ as in~\eqref{eqn:malthus-edge}, there exists a strictly positive finite random variable $W_{\ve}$ such that $e^{-\lambda t}|\BP_{\ve}(t)| \stackrel{a.s., \bL^1}{\longrightarrow} W_{\ve}$ as $t\to\infty$. 
        \item  Further, suppose for some $l >1$,  $\E([\hat\mvzeta_{\ve}(\gl(1 \wedge \alpha^{-1}))]^l) < \infty$. Then there exists a finite constant $H_l \in (0,\infty)$ such that 
        \[\sup_{t\geq 0}\E\left(\left[e^{-\lambda t}|\BP_{\ve}(t)|\right]^l\right) \leq H_l.\]
    \end{enumeratea}
\end{lem}
\begin{proof}
The parameter $\lambda$ in~\eqref{eqn:malthus-edge} is the so-called Malthusian rate of growth of the edge branching process $\BP^\circ_{\ve}$, which under the ``$\hat\mvzeta\log^+(\hat\mvzeta)$'' condition for continuous time branching processes implies
\[
e^{-\lambda t}|\BP^{\circ}_{\ve}(t)| \stackrel{a.s., \bL^1}{\longrightarrow} W^{\circ}_{\ve},
\]
see~\cite{jagers-nerman-1,jagers-nerman-2,jagers-ctbp-book}, where the limiting random variable is finite and strictly positive. In particular,
\[
\sup_{t \ge 0}e^{-\lambda t}\E(|\BP^{\circ}_{\ve}(t)|)< \infty.
\]
For $i \ge 0$, let $W^{\circ}_i$ denote the almost sure limit of $e^{-\lambda t}|\BP^{\circ}_{\ve,i}(t)|$ as $t \to \infty$. Recall that the point process $\chi$ has rate $r_{\chi}(t) := \frac{\alpha}{2 + \alpha} r(t), \, t \ge 0,$ { where we recall $r(t) := \int_0^t e^s dF_{\eta}(s)$,} and hence,
\begin{equation*}
\E\left(\sum_{i=0}^{\infty}e^{-\gl \theta_i}\right) = \int_0^\infty e^{-\lambda t}r_{\chi}(t)dt = \alpha\int_0^\infty e^{-\lambda t} \mu_{\mvzeta_{\ve}}(dt) = \alpha < \infty
\end{equation*}
using the definition of $\lambda$.

Thus, writing $W_{\ve} := \sum_{i=0}^{\infty}e^{-\lambda \theta_i}W^{\circ}_i$, we conclude from the above observations that $e^{-\lambda t}|\BP_{\ve}(t)| \stackrel{a.s.}{\rightarrow} W_{\ve}$ as $t \to \infty$. Moreover, $\E(W_{\ve}) < \infty$ gives, in particular, the (almost sure) finiteness of $W_{\ve}$. Further,
\begin{align*}
\E\left(\Big|e^{-\gl t}|\BP_{\ve}(t)| - W_{\ve}\Big|\right) &\le \E\left(\sum_{i=0}^\infty e^{-\gl \theta_i} \Big| e^{-\gl (t- \theta_i)}|\BP^{\circ}_{\ve,i}(t - \theta_i)| - W^{\circ}_i\Big|\right)\\
&\qquad + \E\left(\sum_{i=\chi(t)}^\infty e^{-\gl \theta_i}W^{\circ}_i\right).
\end{align*}
The convergence to zero of the right-hand side above follows from the recorded observation { that $e^{-\lambda t}|\BP^{\circ}_{\ve}(t)| \stackrel{\bL^1}{\longrightarrow} W^{\circ}_{\ve}$} along with the dominated convergence theorem and the independence of $\chi$ and $\{\BP^{\circ}_{\ve,i} : i \in \mathbb{N}_0\}$.

To prove the second result, we first note by the equivalence of the existence of moments of $\hat\mvzeta_{\ve}(\gl)$ and the corresponding  moments of the normalized branching process, see~\cite[Theorem 4]{mori2019moments}, that there exists $H^\circ_l \in (0, \infty)$ such that
\[\sup_{t\geq 0}\E\left(\left[e^{-\gl t}|\BP^\circ_{\ve}(t)|\right]^l\right) \leq H^\circ_l.\]
{ By H\"older's inequality $\sum_{i=0}^n |a_ib_i| \le \left(\sum_{i=0}^n |a_i|^p\right)^{1/p}\left(\sum_{i=0}^n |b_i|^q\right)^{1/q}$, with $n=\chi(t)-1$, $p=\frac{l}{l-1}, q = l$, $a_i = e^{-\gl \frac{l-1}{l} \theta_i}$, $b_i = e^{-\gl \frac{1}{l} \theta_i}e^{-\gl (t-\theta_i)}|\BP^{\circ}_{\ve,i}(t - \theta_i)|$,}
\begin{align*}
\E\left(\left[e^{-\lambda t}|\BP^\circ_{\ve}(t)|\right]^l\right) &= \E\left(\left[\sum_{i=0}^{\chi(t)-1}e^{-\gl \theta_i}e^{-\gl (t-\theta_i)}|\BP^{\circ}_{\ve,i}(t - \theta_i)|\right]^l\right)\\
&\le \E\left(\left[\sum_{i=0}^{\chi(t)-1}e^{-\gl \theta_i}\right]^{l-1}\left[\sum_{i=0}^{\chi(t)-1}e^{-\gl \theta_i}\left(e^{-\lambda (t-\theta_i)}|\BP^{\circ}_{\ve,i}(t - \theta_i)|\right)^l\right]\right)\\
& \le H^\circ_l \,\E\left(\left[\sum_{i=0}^{\chi(t)-1}e^{-\gl \theta_i}\right]^{l}\right) \le H^\circ_l \,\E\left(\left[\int_0^\infty e^{-\gl t}\chi(dt)\right]^l\right),
\end{align*}
where we applied the independence of $\chi$ and $\{\BP^{\circ}_{\ve,i} : i \in \mathbb{N}_0\}$ in the second inequality. By rate considerations, the point process $\{\chi(t) : t \ge 0\}$ has the same law as $\{\mvzeta_{\ve}(\alpha t) : t \ge 0\}$. Hence,
\begin{align*}
    \E\left(\left[\int_0^\infty e^{-\lambda t}\chi(dt)\right]^l\right) &= \E\left(\left[\int_0^\infty \lambda e^{-\lambda t}\chi(t)dt\right]^l\right) = \E\left(\left[\int_0^\infty \lambda e^{-\lambda t}\mvzeta_{\ve}(\alpha t)dt\right]^l\right)\\
    &= \E\left(\left[\int_0^\infty e^{-\lambda \alpha^{-1} t}\mvzeta_{\ve}(dt)\right]^l\right) = \E([\hat\mvzeta_{\ve}(\gl \alpha^{-1})]^l) < \infty
\end{align*}
by our assumption. The result follows from the above two displays.
\end{proof}

\noindent{\bf Completing the proof of Theorem~\ref{thm:non-conden}:} Let us first prove part (a), namely the lower bound. By Lemma~\ref{lem:bpe-moment}(a), there exist $\delta, t_0, \eta >0$ such that

\begin{equation}
    \label{eqn:248}
     \pr(e^{-\lambda t}|\BP_{\ve}(t)| >\delta )\geq \eta \qquad \forall~t\geq t_0.
\end{equation}
Then for the degree distribution, for all $k\geq \delta e^{\lambda t_0}$, with $C =\eta \delta^{1/\lambda} $,
\[\pr(D_{\Ma}\geq k) = \int_0^\infty e^{-t}\pr(|\BP_{\ve}(t)|\geq k) dt \geq \int_{\frac{\log (k/\delta)}{\lambda}}^{\infty} e^{-t}\pr(|\BP_{\ve}(t)|\geq k) dt \geq \frac{C}{k^{1/\lambda}}.   \]
To prove (b), assume there exists $l> 1/\gl$ such that $\E([\mvzeta_{\ve}(\gl(1 \wedge \alpha^{-1}))]^l) < \infty$. Using Lemma~\ref{lem:bpe-moment}(b) and Markov's inequality gives for any $x>0$ and $t\geq  0$,
\[\pr(|\BP_{\ve}(t)| > x) \leq \frac{H_l}{[e^{-\lambda t} x]^l}.\]
Thus,
\begin{align*}
    \pr(D_{\Ma} \geq k) &=\int_0^\infty e^{-t}\pr(|\BP_{\ve}(t)|\geq k) dt \leq \int_0^{\frac{\log{k}}{\lambda}} \frac{H_l}{[e^{-\lambda t} x]^l} e^{-t} dt + \int_{\frac{\log{k}}{\lambda}}^\infty e^{-t} dt \\
    &\le \left[\frac{H_l}{\lambda l -1} + 1\right] \frac{1}{k^{1/\lambda}}.
\end{align*}
This completes the proof of the theorem. \qed\\

\noindent{\bf Completing the proof of Corollary~\ref{cor:degree-tail-macro-non-cond}:}
Part (a) of Corollary~\ref{cor:degree-tail-macro-non-cond} follows directly from Theorem~\ref{thm:non-conden}.~\eqref{eqn:gltheta} in Corollary~\ref{cor:degree-tail-macro-non-cond}(b) follows by a straightforward algebraic manipulation. We omit the details. To prove the remaining claims of part (b), for any $\gl' >0$, $l >1$,
\begin{equation}\label{lapbound}
\E\left([\hat\mvzeta_{\ve}(\gl')]^l\right) = \E\left(\left[\int_0^\infty \gl'e^{-\gl' t}\mvzeta_{\ve}(t)dt\right]^l\right) \le \int_0^\infty \gl'e^{-\gl' t}\E\left(\left[\mvzeta_{\ve}(t)\right]^l\right)dt
\end{equation}
where the last step follows from Jensen's inequality. { Using the explicit form of the rate associated with the process $\mvzeta_{\ve}$, we conclude that $\mvzeta_{\ve}(t)$ is a Poisson random variable with mean
$$
\frac{1}{2+ \alpha}\int_0^t\int_0^se^u\theta e^{-\theta u}du ds = \frac{\theta}{(2+ \alpha)(1-\theta)}\left(\frac{1}{1 - \theta}(e^{(1-\theta)t}-1) - t\right)
$$
if $\theta \neq 1$ and $t^2/(4 + 2 \alpha)$ if $\theta = 1$. By standard moment bounds for Poisson random variables, if $X \sim \Poi(\gamma)$, then for any $l >1$, there exists $c_l >0$ such that $\E(X^l) \le c_l \gamma^l$. Using these observations, for any $l>1$, there exists $C_l \in (0, \infty)$ (also depending on $\theta, \alpha$) such that 
$$
\E\left(\left[\mvzeta_{\ve}(t)\right]^l\right) \le C_l\left(t^l + t^{2l}\ind\set{\theta=1} + e^{l(1-\theta)t}\ind\set{\theta<1}\right).
$$}
Checking the assumptions of Theorem~\ref{thm:non-conden}(a) for all values of $\alpha \ge 0, \theta>0$ (thereby giving the lower bound on the degree distribution) and those of Theorem~\ref{thm:non-conden}(b) when $\alpha \ge 0, \theta \ge 1$ (giving degree distribution upper bound) is now routine using the above and the observation $\lambda + \theta>1$. To check the moment condition of Theorem~\ref{thm:non-conden}(b) when $\alpha=0$, $\theta \in (0,1)$, we use the Poisson moment bound above in~\eqref{lapbound} to obtain for any $l >1$ a constant $C'_l \in (0, \infty)$ such that
$$
\E\left([\hat\mvzeta_{\ve}(\gl)]^l\right) \le C'_l \int_0^{\infty}e^{(l(1-\theta)- \gl)t}dt.
$$
As we only need finiteness of the right-hand side for some $l>1/\gl$, it suffices to show $\gl^2>1-\theta$. Using the explicit formula of $\gl$ in~\eqref{eqn:gltheta}, note that
\begin{align*}
    \gl^2 &= \frac{\left((1-\theta) + \sqrt{1 + \theta^2}\right)^2}{4} = \frac{1+ \theta^2 - \theta + (1-\theta)\sqrt{1 + \theta^2}}{2}\\
    &\ge \frac{1+ \theta^2 - \theta + (1-\theta)}{2} = (1-\theta) + \frac{\theta^2}{2}
\end{align*}
which completes the proof.\qed

\subsection{Root degree behavior: Proof of Theorem~\ref{thm:macro-max-deg}} 
\label{sec:macro-max-deg-proof}
Here, we consider the model with $\alpha =0$, and the goal is to prove Theorem~\ref{thm:macro-max-deg}.  The proof of the root degree asymptotics relies on a careful quantification of the approximation of the network dynamics by that of the edge branching process described in Section~\ref{edgedesc}, that extends beyond the `fringe' (which the local limit captures) to the older vertices of the network evolution. First recall the ``edge copying equivalence'' of preferential attachment explained in Fig.~\ref{fig:three graphs}. We describe a construction which is a natural extension to the delay regime. 
\begin{constr}[Copying construction of the delayed model]
\label{constr-copy-delay}
    Consider the following random tree model $\set{\overline{\cT}(n): n\geq 1}$, started with a single edge $\overline{\cT}(1) = \{e_1\} =\{\set{v_0, v_1}\} $ and recursively constructed as follows:
    \begin{enumeratea}
        \item For $n\geq 2$ assume the tree $\overline{\cT}(n-1)$ consists of vertex set $\cV(\overline{\cT}(n)) = \set{v_1, v_1, \ldots, v_{n-1}}$ and edge set $\cE(\overline{\cT}(n)):= \set{e_1, e_2, \ldots, e_n}$ in the order of attachment. 
        \item With $\xi\sim F_\xi$ denoting the delay distribution as before, define the pmf $\vp^{\sss(n)} = \set{p_j^n: j\in [n]}$,{
        \begin{equation}
        \label{eqn:pjss}
            p_j^n:= \int_0^{1 - {j}/{n}}\frac{1}{\lfloor n(1-x)\rfloor}dF_{\xi}(x) + \ind\set{j=1}\int_{1 - {1}/{n}}^1 1\cdot  dF_{\xi}(x), \qquad j\in [n]. 
        \end{equation}}
        \item {\bf Copying procedure: } To transition to $\overline{\cT}(n+1)$, first select one of the edges $\cE(\overline{\cT}(n))$ using $\vp^{\sss(n)}$ (i.e.  $e_j$ selected with probability $p_j^n$).  Suppose $e_j = \set{u,v}$ has been selected. Then $\overline{\cT}(n+1)$ is obtained from $\overline{\cT}(n)$ by adding a new vertex $v_{n+1}$ connected to $\overline{\cT}(n)$ with a new edge which is  $\set{v_{n+1}, u}$ with probability $1/2$ or edge $\set{v_{n+1}, v}$ with probability $1/2$.  
    \end{enumeratea}
\end{constr}
The following is elementary to check.

\begin{lem}
\label{lem:equiv}
    Consider the delayed preferential attachment model with $\alpha =0$ and delay distribution $\xi \sim F_\xi$.  As models of growing random trees,  $\set{\overline{\cT}(n):n\geq 1} \stackrel{d}{=}\set{{\cT}(n):n\geq 1} $. 
\end{lem}
In order to understand the rate of growth of the root degree, essentially one needs to carefully understand the rate of copying events that involve the edges attached to the root $v_0$.

{ Recall the discussion in Section~\ref{intuition} to analyze the local limit in terms of an associated continuous time process. To formalize this intuition at the level of the entire network (not just locally around a uniformly chosen vertex)}, we will move Construction~\ref{constr-copy-delay} into continuous time so that copying events happen at a rate proportional to the number of edges in the system. To do this, we first describe a continuous time analogue of the construction in $\set{\overline{\cT}(n):n\geq 1}$  in terms of time-changed Poisson processes. Let $\{N_i(\cdot) : i \in \mathbb{N}\}$ be {\it iid} Poisson rate one processes on $\mathbb{R}_+$. Consider the following evolution equations for $t \ge 0$:
\begin{align}\label{couplingdyn}
 \ve_i(t) &= N_i\left(\int_0^t \hat r_i(s)ds\right), \nonumber\\
 \hat r_i(t) &= \ind\set{i \le Y(t)}\int_0^{1 - {i}/{Y(t)}}\frac{Y(t)}{\lfloor Y(t)(1-x)\rfloor}dF_{\xi}(x) + \ind\set{i=1}\int_{1 - {1}/{Y(t)}}^1Y(t)dF_{\xi}(x),\nonumber\\
 Y(t) &= \sum_{i=1}^{\infty}\ve_i(t),\nonumber\\
\end{align}
where $F_\xi$ is the cdf of the delay variable $\xi$. Since for any time $t$, using the fact that the expression in~\eqref{eqn:pjss} is a probability mass function, it is easy to check that $\sum_{j} \hat r_j(t) = Y(t)$. Thus it is easy to verify that $Y$ is a rate one Yule process with $Y(0)=1$. The above dynamics can be used to describe a growing network process by thinking of $\ve_i(\cdot)$ as the reproduction process of the $i$-th added edge in continuous time as in Section~\ref{edgedesc}. More precisely, consider the continuous time network process $\{\TT_c(t) : t \ge 0\}$, started with $\TT_c(0) = \TT(0)$, where the $i$-th edge, added to the system at time $\tau_i := \inf\{t \ge 0: Y(t) = i\}$ reproduces, giving birth to new edges (equivalently, vertices) at the arrival times of $\{\ve_i(t) : t \ge \tau_i\}$. At each birth time, the new individual is connected to the parent vertex or child vertex of the $i$-th edge with equal chance via an outgoing edge. Lemma~\ref{lem:equiv} now implies that, 
\begin{equation}\label{disctocon}
\set{\TT_c(\tau_i) : i \ge 1} \stackrel{d}{=} \set{\TT(i) : i \ge 1}.
\end{equation}
Although the above recipe provides a continuous time description of the entire discrete time process (not just the local weak limit), the rates $\hat r_i$ appearing above have a substantially more complicated form than those for the edge branching process in Definition~\ref{def:bpmave}. In particular, the \emph{reproduction rates across different vertices are correlated} through their mutual dependence on $Y(\cdot)$. However, using the approximations $\lfloor Y(t)(1-x)\rfloor \approx Y(t)(1-x)$, $\log Y(t) \approx t$ and $\tau_i \approx \log i$, we observe that
\begin{align*}
\hat r_i(t) &\approx \ind\set{\tau_i \le t}\int_0^{\log (Y(t)/i)} e^sdF_{\eta}(s) + \ind\set{i=1}\int_{\log Y(t)}^\infty Y(t)dF_{\eta}(s)\\
&\approx \int_0^{t - \tau_i} e^sdF_{\eta}(s) + \ind\set{i=1}\int_{t}^\infty Y(t)dF_{\eta}(s)\\
&= r(t-\tau_i) + \ind\set{i=1}\int_{t}^\infty Y(t)dF_{\eta}(s),
\end{align*}
where we recall $r(t) := \int_0^t e^s dF_{\eta}(s)$ and $\eta = -\log(1-\xi)$. 
The right-hand side above looks exactly like the reproduction rates of the edge branching process in Definition~\ref{def:bpmave} with some adjustment for the root.

 In view of the above, a crucial technical ingredient is thus quantifying
 \[
 |\hat r_i(t) - r(t-\tau_i) - \ind\set{i=1}\int_{t}^\infty Y(t)dF_{\eta}(s)|
 \]
 as a function of $t,i$. 

The following lemma provides bounds on this quantity. 

In the remainder of this section, we will denote by $\mathbb{P}_i$ and $\mathbb{E}_i$ the conditional probability and expectation given $\mathcal{F}_{\tau_i} :=\{\TT_c(s) : s \le \tau_i\}$. We will write $\prob, \E$ for $\prob_1, \E_1$. Also, $C, C', C'', C_1,C_2,..$ will denote generic positive constants, not depending on $i,\mathcal{F}_{\tau_i}$, whose values might change between lines.

\begin{lemma}\label{ratecomp}
Assume $\eta$ has a density $f_\eta$ such that $\sup_{u \ge 0} e^u f_\eta(u) < \infty$, and there exists $\gamma>1$ such that $\E(e^{\gamma \eta}) < \infty$. Then, for $i \ge 1$ and $t \ge \tau_i$,
writing
\begin{equation*}
\mathcal{E}_i(t) := \Big| \hat r_i(t) - r(t-\tau_i) - \ind\set{i=1}\int_{t}^\infty Y(t)dF_{\eta}(s)\Big|,
\end{equation*}
we have
\begin{align*}
&\E_i\left[\mathcal{E}_i(t)\right] \\
&\quad\le  \left(C_1 i^{-\frac{\gamma-1}{2\gamma}} e^{-C_2 (t-\tau_i)} + C_1 e^{-C_2 i} e^{\frac{\gamma-1}{2\gamma}(t-\tau_i)}\right) \wedge \left(C_1e^{-(\gamma-1)(t-\tau_i)/2} + C_1 e^{-i(t-\tau_i)/2}\right),
\end{align*}
where the constants $C_1, C_2$ depend only on $\gamma$.
\end{lemma}

Note that the first bound in the lemma has a better dependence on $i$ and a worse dependence on $t$. The second bound has the reverse nature. As we will see later, the first bound will play a crucial role in the upper bound for the root degree while the second bound is required for the root degree lower bound.

In order to prove the above lemma, we need some estimates about the Yule process recorded in the following lemma. Proof is given in Appendix~\ref{app}.

\begin{lemma}\label{Yuleest}
Let $Y$ be a rate one Yule process with $Y(0)=1$ and let $i \ge 1$. Then,
\begin{enumeratei}
\item For $t \ge 0$,
\[
Y(t) \sim \operatorname{Geom}(e^{-t}).
\]
Moreover,  $\{e^{-(t-\tau_i)}Y(t) : t \ge \tau_i\}$ is a $\prob_i$-martingale and
\[
e^{-(t-\tau_i)} Y(t) \stackrel{\prob_i-a.s., L^2}{\longrightarrow} W_i
\]
as $t \to \infty$, where $W_i \stackrel{d}{=} \sum_{l=1}^i E_l$ and $\{E_l\}$ are iid $\exp(1)$ random variables. Further, 
$$
\E_i\left|Y(t) - e^{t-\tau_i}W_i\right|^2 \le Ci e^{t- \tau_i}, \ t \ge \tau_i.
$$
\item 
$
\prob_i\left(W_i < i x\right) \le \left(e x e^{-x}\right)^i, \ x \in [0, 1).
$
\item $\E_i \left| \log \frac{W_i}{i} \right| \le \frac{C}{\sqrt{i}}$.
\item For any $\beta>0$, $t \ge \tau_i$,
\begin{align*}
\E_i\left[\left(\frac{Y(t)}{i}\right)^{-\beta} \wedge 1\right] &\le e^{-\beta(t-\tau_i)/2} + e^{-i(t-\tau_i)/2}.
\end{align*}
\end{enumeratei}
\end{lemma}

\begin{proof}[Proof of Lemma~\ref{ratecomp}]
{ Define the `error terms':
\begin{align*}
\mathcal{E}^{(1)}_i(t) &:= \int_0^{1 - \frac{i}{Y(t)}}\Big|\frac{Y(t)}{\lfloor Y(t)(1-x)\rfloor} - \frac{1}{1-x}\Big|dF_{\xi}(x),\\
 \mathcal{E}^{(2)}_i(t) &:= \Big|\int_0^{\log \left(Y(t)/i\right)} e^s dF_{\eta}(s) -  \int_0^{t-\tau_i} e^s dF_{\eta}(s)\Big|,\\
 \mathcal{E}^{(3)}(t) &:= \Big| \int_{\log Y(t)}^{\infty}Y(t)dF_\eta(s) - \int_{t}^\infty Y(t)dF_{\eta}(s)\Big|,  \ \ i \ge 1, t \ge \tau_i.
\end{align*}
Observe that, for $t \ge \tau_i$,
\begin{align}\label{r1}
   &\Big| \hat r_i(t) - r(t-\tau_i) - \ind\set{i=1}\int_{t}^\infty Y(t)dF_{\eta}(s)\Big|\notag\\
   & \le  \int_0^{1 - \frac{i}{Y(t)}}\Big|\frac{Y(t)}{\lfloor Y(t)(1-x)\rfloor} - \frac{1}{1-x}\Big|dF_{\xi}(x) + \Big|\int_0^{1 - \frac{i}{Y(t)}}\frac{1}{1-x}dF_{\xi}(x) - \int_0^{t-\tau_i} e^s dF_{\eta}(s)\Big|\notag\\
   &\qquad + \ind\set{i=1}\Big|\int_{1-\frac{1}{Y(t)}}^1Y(t)dF_{\xi}(x) - \int_{t}^\infty Y(t)dF_{\eta}(s)\Big|\notag\\
   &= \mathcal{E}^{(1)}_i(t) + \mathcal{E}^{(2)}_i(t) + \ind\set{i=1}\mathcal{E}^{(3)}(t),
\end{align}
where we used the transformation $\eta = -\log(1-\xi)$ in noticing $\int_0^{1 - \frac{i}{Y(t)}}\frac{1}{1-x}dF_{\xi}(x) = \int_0^{\log \left(Y(t)/i\right)} e^s dF_{\eta}(s)$ and $\int_{1-\frac{1}{Y(t)}}^1Y(t)dF_{\xi}(x) = \int_{\log Y(t)}^{\infty}Y(t)dF_\eta(s)$.}

Now we will estimate the three errors above. Note that
\begin{align}\label{r2}
&\int_0^{1 - \frac{i}{Y(t)}}\Big|\frac{Y(t)}{\lfloor Y(t)(1-x)\rfloor} - \frac{1}{1-x}\Big|dF_{\xi}(x) \le \int_0^{1 - \frac{i}{Y(t)}}\frac{1}{(1-x)\lfloor Y(t)(1-x)\rfloor}dF_{\xi}(x)\nonumber\\
&\quad = \int_0^{1 - \frac{ie^{(t-\tau_i)/2}}{Y(t)}}\frac{1}{(1-x)\lfloor Y(t)(1-x)\rfloor}dF_{\xi}(x) + \int_{1 - \frac{ie^{(t-\tau_i)/2}}{Y(t)}}^{1 - \frac{i}{Y(t)}}\frac{1}{(1-x)\lfloor Y(t)(1-x)\rfloor}dF_{\xi}(x)\nonumber\\
&\quad \le \frac{2}{ie^{(t-\tau_i)/2}}\int_0^{\infty}e^udF_\eta(u) + \frac{1}{i}\int_{1 - \frac{ie^{(t-\tau_i)/2}}{Y(t)}}^{1}\frac{1}{(1-x)}dF_{\xi}(x)\nonumber\\
&\quad \le  \frac{2}{i}e^{-(t-\tau_i)/2}\int_0^{\infty}e^udF_\eta(u) + \frac{1}{i}\left(\frac{ie^{(t-\tau_i)/2}}{Y(t)} \wedge 1\right)^{\gamma - 1}\int_{0}^\infty e^{\gamma s}dF_{\eta}(s),
\end{align}
{ where the second inequality follows upon noting that $\lfloor Y(t)(1-x)\rfloor \ge \lfloor i e^{(t-\tau_i)/2}\rfloor \ge \frac{i}{2} e^{(t-\tau_i)/2}$ when $x \le 1 - \frac{ie^{(t-\tau_i)/2}}{Y(t)}$ in the first integral, and $\lfloor Y(t)(1-x)\rfloor \ge i$ when $x \le 1 - \frac{i}{Y(t)}$ for the second integral.}

To estimate the second term in the above bound, first note that
\begin{align*}
\prob_i\left(\frac{Y(t)}{i} < e^{3(t-\tau_i)/4}\right) \le \prob_i\left(\left|Y(t) - e^{t-\tau_i} W_i\right| > i e^{3(t-\tau_i)/4}\right) + \prob_i\left(W_i < 2i e^{-(t-\tau_i)/4}\right).
\end{align*}
{  By Lemma~\ref{Yuleest}(i) and Chebyshev's inequality, we have}
\begin{align*}
\prob_i\left(\left|Y(t) - e^{t-\tau_i} W_i\right| > i e^{3(t-\tau_i)/4}\right) \le \frac{Ci e^{t- \tau_i}}{i^2 e^{3(t-\tau_i)/2}} = \frac{C}{i} e^{-(t-\tau_i)/2}.
\end{align*}
Moreover, by Lemma~\ref{Yuleest}(ii), there exists $t_0>0$ such that for all $t \ge \tau_i + t_0$,
\begin{align*}
\prob_i\left(W_i < 2i e^{-(t-\tau_i)/4}\right) \le \left(2e^{1-(t-\tau_i)/4}\right)^i \le e^{-i(t-\tau_i)/8}.
\end{align*}
Using these estimates, we obtain
$$
\prob_i\left(\frac{Y(t)}{i} < e^{3(t-\tau_i)/4}\right) \le \frac{C}{i} e^{-(t-\tau_i)/2} + C e^{-(t-\tau_i)/8}.
$$
Hence,
\begin{align*}
\E_i\left[\left(\frac{ie^{(t-\tau_i)/2}}{Y(t)} \wedge 1\right)^{\gamma - 1}\right] &\le e^{-(\gamma-1)(t - \tau_i)/4} + \prob_i\left(\frac{Y(t)}{i} < e^{3(t-\tau_i)/4}\right)\\
&\le  e^{-(\gamma-1)(t - \tau_i)/4} + \frac{C}{i} e^{-(t-\tau_i)/2} + C e^{-(t-\tau_i)/8}.
\end{align*}
Using this in~\eqref{r2}, we obtain
\begin{equation}\label{rem1}
\E_i\left[\mathcal{E}^{(1)}_i(t)\right] \le \frac{C}{i}e^{-(t-\tau_i)/2} + \frac{1}{i}e^{-(\gamma-1)(t - \tau_i)/4} + \frac{C}{i^2} e^{-(t-\tau_i)/2} + \frac{C}{i} e^{-(t-\tau_i)/8}.
\end{equation}
To estimate the second error term in~\eqref{r1}, note that
\begin{align}\label{r3}
\begin{split}
& \Big|\int_0^{\log \left(Y(t)/i\right)} e^s dF_{\eta}(s) -  \int_0^{t-\tau_i} e^s dF_{\eta}(s)\Big| \\
&\qquad  \le \Big|\int_0^{\log \left(Y(t)/i\right)} e^s dF_{\eta}(s) -  \int_0^{t-\tau_i + \log\frac{W_i}{i} } e^s dF_{\eta}(s)\Big|\\
&\qquad \qquad \qquad  + \Big|\int_0^{t-\tau_i + \log\frac{W_i}{i} } e^s dF_{\eta}(s) -  \int_0^{t-\tau_i} e^s dF_{\eta}(s)\Big|.
 \end{split}
\end{align}
By an application of H\"older's inequality,
\begin{align}\label{rem20}
&\E_i\Big|\int_0^{\log \left(Y(t)/i\right)} e^s dF_{\eta}(s) -  \int_0^{t-\tau_i+ \log\frac{W_i}{i}} e^s dF_{\eta}(s)\Big|\nonumber\\
&\le \left(\E_i\Big|\log \left(\frac{Y(t)}{i}\right) - \left(t-\tau_i - \log\frac{W_i}{i}\right)^+  \Big|\right)^{\frac{\gamma-1}{\gamma}}\left(\E_i\int_{\log \left(Y(t)/i\right) \wedge \left(t-\tau_i + \log\frac{W_i}{i}\right)}^{\infty}e^{\gamma s}f_\eta^\gamma(s)ds\right)^{\frac{1}{\gamma}}\nonumber\\
&\le C \left(\E_i\Big|\log \left(\frac{Y(t)}{i}\right) - \left(t-\tau_i - \log\frac{W_i}{i} \right)^+ \Big|\right)^{\frac{\gamma-1}{\gamma}}\left(\E_i\int_{\log \left(Y(t)/i\right) \wedge \left(t-\tau_i + \log\frac{W_i}{i}\right)}^{\infty}e^{s}f_\eta(s)ds\right)^{\frac{1}{\gamma}}\nonumber\\
&\le C_1 \left(\E_i\Big|\log \left(\frac{Y(t)}{i}\right) - \left(t-\tau_i - \log\frac{W_i}{i}\right)^+  \Big|\right)^{\frac{\gamma-1}{\gamma}}\nonumber\\
&\qquad \qquad \times \left(\E_i\left[C_2 \wedge \left(\left(\frac{Y(t)}{i}\right) \wedge \left(e^{t-\tau_i} \frac{W_i}{i}\right)\right)^{-(\gamma-1)}\right]\right)^{\frac{1}{\gamma}},
\end{align}
where we have used $\sup _{u \ge 0}e^uf_{\eta}(u) < \infty$ in the second inequality and $\E(e^{\gamma \eta}) < \infty$ in the third inequality. 
As $|\log x - \log y| \le (x \wedge y)^{-1} |x-y|$ for all $x, y \ge 1$,
\begin{align}\label{r4}
&\E_i\Big|\log \left(\frac{Y(t)}{i}\right) - \left(t-\tau_i - \log\frac{W_i}{i}\right)^+  \Big|\nonumber\\
& \qquad \le \E_i\left[\left(\left(\frac{Y(t)}{i} \wedge e^{(t - \tau_i)}\frac{W_i}{i}\right)^{-1} \wedge 1\right) \, \left|\frac{Y(t)}{i} -  e^{(t - \tau_i)}\frac{W_i}{i}\right|\right]\nonumber\\
&\qquad \le  \left(\E_i\left[\left(\frac{Y(t)}{i} \wedge e^{(t - \tau_i)}\frac{W_i}{i}\right)^{-2} \wedge 1\right]\right)^{1/2} \left(\E_i\left|\frac{Y(t)}{i} -  e^{(t - \tau_i)}\frac{W_i}{i}\right|^2\right)^{1/2}.
\end{align}
Now, using Lemma~\ref{Yuleest}(i),
$$
\E_i\left|\frac{Y(t)}{i} -  e^{(t - \tau_i)}\frac{W_i}{i}\right|^2 \le \frac{C}{i}e^{t-\tau_i}.
$$
Moreover, to estimate the first term in the bound~\eqref{r4}, note that by Lemma~\ref{Yuleest}(i) and (ii),
\begin{align*}
\prob_i\left(Y(t) < \frac{1}{2}e^{(t - \tau_i)}W_i \right) &\le \prob_i\left(Y(t) - e^{(t - \tau_i)}W_i  \le - \frac{1}{4}i e^{(t-\tau_i)}\right) + \prob_i\left(W_i < \frac{i}{2}\right)\\
& \le \frac{C}{i}e^{-(t-\tau_i)} + e^{-C'i}.
\end{align*}
Using this fact, we obtain
\begin{align*}
&\E_i\left[\left(\frac{Y(t)}{i} \wedge e^{(t - \tau_i)}\frac{W_i}{i}\right)^{-2} \wedge 1\right] \\
&\le \prob_i\left(Y(t) < \frac{1}{2}e^{(t - \tau_i)}W_i \right)  + \E_i\left[\left(e^{(t - \tau_i)}\frac{W_i}{2i}\right)^{-2} \wedge 1\right]\\
& \le \frac{C}{i}e^{-(t-\tau_i)} + e^{-C'i} + \prob_i\left(W_i < \frac{i}{2}\right) + 16 e^{-2(t-\tau_i)}\\
& \le \frac{C}{i}e^{-(t-\tau_i)} + 2 e^{-C' i} + 16 e^{-2(t-\tau_i)}.
\end{align*}
Using the above estimates in~\eqref{r4}, we get
\begin{equation}\label{r5}
\E_i\Big|\log \left(\frac{Y(t)}{i}\right) - \left(t-\tau_i - \log\frac{W_i}{i}\right)^+  \Big| \le \frac{C}{i} + C e^{-C' i}e^{(t-\tau_i)/2} + \frac{C}{\sqrt{i}}e^{-(t - \tau_i)/2}.
\end{equation}
The second term in~\eqref{rem20} can be estimated as in the first term in~\eqref{r4} to obtain
\begin{align}\label{r6}
\E_i\left[C_2 \wedge \left(\left(\frac{Y(t)}{i}\right) \wedge \left(e^{t-\tau_i} \frac{W_i}{i}\right)\right)^{-(\gamma-1)}\right] \le \frac{C}{i}e^{-(t-\tau_i)} + 2 e^{-C' i} + 16 e^{-(\gamma-1)(t-\tau_i)}.
\end{align}
Using~\eqref{r5} and~\eqref{r6} in~\eqref{rem20}, we obtain
\begin{align}\label{rem21}
\E_i\Big|\int_0^{\log \left(Y(t)/i\right)} e^s dF_{\eta}(s) -  \int_0^{t-\tau_i+ \log\frac{W_i}{i}} e^s dF_{\eta}(s)\Big|
\le \frac{C}{i^{\frac{\gamma-1}{2\gamma}}}e^{-C'(t-\tau_i)} + Ce^{-C'i}e^{\left(\frac{\gamma-1}{2\gamma}\right)(t- \tau_i)}.
\end{align}
Similarly, using Lemma~\ref{Yuleest}(ii) and (iii),
\begin{align}\label{rem22}
&\E_i\Big|\int_0^{t-\tau_i + \log\frac{W_i}{i} } e^s dF_{\eta}(s) -  \int_0^{t-\tau_i} e^s dF_{\eta}(s)\Big|\nonumber\\
&\qquad \le C_1\left(\E_i\left|\log \frac{W_i}{i}\right|\right)^{\frac{\gamma-1}{\gamma}}\left(\E_i\left[C_2 \wedge \left(e^{(t-\tau_i)} \wedge e^{(t-\tau_i)} \frac{W_i}{i}\right)^{-(\gamma - 1)}\right]\right)^{\frac{1}{\gamma}}\nonumber\\
&\qquad \le \frac{C}{i^{\frac{\gamma-1}{2\gamma}}}\left(C e^{-C' i} + C e^{-\left(\frac{\gamma-1}{\gamma}\right)(t- \tau_i)}\right).
\end{align}
Hence, using~\eqref{rem21} and~\eqref{rem22} in~\eqref{r3}, we obtain
\begin{equation}\label{rem2}
\E_i\left[ \mathcal{E}^{(2)}_i(t)\right] \le \frac{C}{i^{\frac{\gamma-1}{2\gamma}}}e^{-C'(t-\tau_i)} + Ce^{-C'i}e^{\left(\frac{\gamma-1}{2\gamma}\right)(t- \tau_i)}.
\end{equation}
Finally,
\begin{align}\label{rem3}
\E\left[\mathcal{E}^{(3)}(t)\right] &\le \E\left[Y(t)\int_{\log Y(t) \wedge t}^\infty dF_\eta(s)\right] \le C\E\left[Y(t)(Y(t)^{-\gamma} \vee e^{-\gamma t})\right]\nonumber\\
&\le C\E\left[Y(t)^{-\gamma+ 1} + e^{-\gamma t} Y(t)\right] \le C \prob(Y(t) \le e^{t/2}) + e^{-(\gamma - 1)t/2} + e^{-(\gamma-1)t}\nonumber\\
& \le e^{-t/2} + e^{-(\gamma - 1)t/2} + e^{-(\gamma-1)t},
\end{align}
where we have used $\E(e^{\gamma \eta}) < \infty$ and the fact that $Y(t) \sim \operatorname{Geom}(e^{-t})$.
The first bound on $\E_i[\mathcal{E}_i(t)]$ in the lemma now follows upon using the estimates~\eqref{rem1},~\eqref{rem2} and~\eqref{rem3} in~\eqref{r1}.

To obtain the second bound, we will change the bound on $\mathcal{E}^{(2)}_i(t)$ by simply writing
\begin{equation*}
    \mathcal{E}^{(2)}_i(t) \le C\left[\left(\frac{Y(t)}{i} \wedge e^{t- \tau_i}\right)^{-(\gamma-1)} \wedge 1\right] \le C\left[\left(\frac{Y(t)}{i}\right)^{-(\gamma-1)} \wedge 1\right] + Ce^{-(\gamma-1)(t-\tau_i)}.
\end{equation*}
Hence, by Lemma~\ref{Yuleest}(iv), we obtain
\begin{equation*}
\E_i\left[\mathcal{E}^{(2)}_i(t)\right] \le 2Ce^{-(\gamma-1)(t-\tau_i)/2} + C e^{-i(t-\tau_i)/2},
\end{equation*}
which gives the second bound.
\end{proof}

\noindent\textbf{Completing the proof of upper bound in Theorem~\ref{thm:macro-max-deg}:} For $t \ge 0$, let $\rho(t)$ denote the degree of the root in $\TT_c(t)$ at time $t$ and let $$S_{\rho}(t) := \{i \le Y(t) : \text{$i$-th added vertex attaches to root of } \TT_c(t)\}.$$ Then, writing $m(t) := \E[\rho(t)]$, observe that
\begin{align*}
m'(t) = \frac{1}{2}\E\left[\sum_{i \in S_{\rho}(t)}\hat r_i(t)\right] \le \frac{1}{2}\E\left[\sum_{i \in S_{\rho}(t)}r(t-\tau_i) + \int_{t}^\infty Y(t)dF_{\eta}(s)\right] + \frac{1}{2}\E\left[\sum_{i \in S_{\rho}(t)}\mathcal{E}_i(t)\right].
\end{align*}
By Lemma~\ref{ratecomp},
\begin{align*}
&\int_0^t\E\left[\sum_{i \in S_{\rho}(s)}\mathcal{E}_i(s)\right] ds \le \E\left[\sum_{i \in S_{\rho}(t)}\int_{\tau_i}^tC_1\left(i^{-\frac{\gamma-1}{2\gamma}}e^{-C_2 (s-\tau_i)}  + e^{-C_2 i}e^{\frac{\gamma-1}{2\gamma}(s-\tau_i)}\right)ds\right]\\
&\le C\E\left[\rho(t)^{\frac{1}{2} + \frac{1}{2\gamma}}\right] + C e^{\frac{\gamma-1}{2\gamma}t}
\le C\E\left[Y(t)^{\frac{1}{2} + \frac{1}{2\gamma}}\right] + C e^{\frac{\gamma-1}{2\gamma}t}\\
&\le C\left(\E(Y(t))\right)^{\frac{1}{2} + \frac{1}{2\gamma}} + C e^{\frac{\gamma-1}{2\gamma}t} = Ce^{\frac{t}{2} + \frac{t}{2\gamma}} + C e^{\frac{\gamma-1}{2\gamma}t}.
\end{align*}
Moreover, writing $\bar F_\eta(t):= \int_{t}^\infty dF_{\eta}(s)$,
\begin{align*}
\E\left[\sum_{i \in S_{\rho}(t)}r(t-\tau_i) + \int_{t}^\infty Y(t)dF_{\eta}(s)\right] = \E\left[\int_{[0,t]}r(t-u)\rho(du)\right] + e^{t}\bar F_\eta(t)
\end{align*}
and, using integration by parts,
\begin{align*}
\int_0^t\E\left[\int_{[0,s]}r(s-u)\rho(du)\right] ds = \int_0^t\int_0^sm(s-u) r'(u)duds = \int_0^tm(t-s)r(s)ds.
\end{align*}
Combining the above observations, we obtain
\begin{align}\label{ren1}
m(t) &\le 1 + \frac{1}{2}\int_0^tm(t-s)r(s)ds + \frac{1}{2}\int_0^te^{s}\bar F_\eta(s)ds + \frac{1}{2}\int_0^t\left(Ce^{\frac{s}{2} + \frac{s}{2\gamma}} + C e^{\frac{\gamma-1}{2\gamma}s}\right)ds\nonumber\\
&\le 1 + \frac{1}{2}\int_0^tm(t-s)r(s)ds + \frac{1}{2}\int_0^te^{s}\bar F_\eta(s)ds + C' e^{\frac{t}{2} + \frac{t}{2\gamma}} + C' e^{\frac{\gamma-1}{2\gamma}t}\nonumber\\
&= g(t) +  \frac{1}{2}\int_0^tm(t-s)r(s)ds,
\end{align}
where $g(t) := 1 + \frac{1}{2}\int_0^te^{s}\bar F_\eta(s)ds + C' e^{\frac{t}{2} + \frac{t}{2\gamma}} + C' e^{\frac{\gamma-1}{2\gamma}t}$. Now, define the measure $\nu(ds) := \frac{1}{2}e^{-\lambda s} r(s) ds$ on $\mathbb{R}_+$. By the definition of $\lambda$, $\nu$ is a probability measure. Hence, from~\eqref{ren1}, we obtain the following `renewal theoretic' upper bound
\begin{equation*}
e^{-\lambda t} m(t) \le \int_0^t e^{-\lambda (t-s)} g(t-s) R(ds)
\end{equation*}
where $R(ds) := \delta_{0} + \sum_{l=1}^{\infty}\nu^{*l}(ds)$ is the renewal measure associated with $\nu$. Once again, using the definition of $\lambda$, it can be checked that
$$
\int_0^{\infty}e^{-\lambda s}\int_0^se^{u}\bar F_\eta(u)duds = \frac{2\lambda - 1}{\lambda(1- \lambda)} \in (0,\infty).
$$
Moreover, recalling $\lambda \in (1/2, 1)$ we can choose and fix $\gamma>1$ large enough such that $\frac{1}{2} + \frac{1}{2\gamma} < \lambda$. This is allowed as $\E(e^{\gamma \eta})< \infty$ by the assumed hypothesis of the Theorem. For this choice of $\gamma$,
$$
\int_0^{\infty}e^{-\lambda s}\left(e^{\frac{s}{2} + \frac{s}{2\gamma}} + e^{\frac{\gamma-1}{2\gamma}s}\right)ds < \infty.
$$
Further, by the assumption on finiteness of exponential moments, $\int_0^{\infty}s e^{-\lambda s} r(s) ds < \infty$.
Hence, by the key renewal theorem~\cite{jagers-ctbp-book}, we obtain
\begin{equation}\label{mbdct}
m^* := \limsup_{t \rightarrow \infty} e^{-\lambda t} m(t) \le \frac{2\int_0^{\infty}e^{-\lambda s}g(s) ds}{\int_0^{\infty}s e^{-\lambda s} r(s) ds} < \infty.
\end{equation}
Finally, we will translate the above bound for the root degree in the original discrete time network $\{\TT(n): n \ge 1\}$. Note that, by~\eqref{disctocon}, $\{M(\rho, n) : n \ge 1\} \stackrel{d}{=} \{\rho(\tau_n) : n \ge 1\}$. Therefore, for $A,B>0$,
\begin{align*}
\prob(M(\rho,n) > A n^{\lambda}) &\le \prob(\tau_n - \log n \ge B) + \prob\left(e^{-\lambda (\log n + B)} \rho(\log n + B) > A e^{-\lambda B}\right)\\
& =\prob(Y(\log n + B) \le n) + \prob\left(e^{-\lambda (\log n + B)} \rho(\log n + B) > A e^{-\lambda B}\right)\\
&\le ne^{-(\log n + B)} + \frac{e^{\lambda B}}{A}\E\left[e^{-\lambda (\log n + B)} \rho(\log n + B)\right].
\end{align*}
Hence, from~\eqref{mbdct},
$$
\limsup_{n \to \infty} \prob(M(\rho,n) > A n^{\lambda}) \le e^{-B} + m^* \frac{e^{\lambda B}}{A}.
$$
The upper bound in the theorem now follows upon choosing $B = \log A/(\lambda + 1)$. \qed

\subsection{Lower bound in Theorem~\ref{thm:macro-max-deg}}\label{lbsec} 
For the lower bound, we construct a coupling between $\TT_c(\cdot)$ and the edge branching process described in Section~\ref{edgedesc}. We denote the root of $\TT_c$ by $\rho$. Recall the Ulam-Harris set $\mathcal{U} := \cup_{i=0}^{\infty} \mathbb{N}^i$, where $\mathbb{N}_0 := \{\rho\}$. To describe the coupling, we will view the vertices in any growing tree process $\mathcalb{T} := \{\mathcalb{T}(t) : t \ge 0\}$ (with edges directed from child to parent) as $\mathcal{U} \times \mathbb{R}_+$-valued objects where the first coordinate encodes genealogical information and the second coordinate encodes birth time. For vertex $v \in \mathcalb{T}$, denote by $e(v)$ the {\bf unique} outgoing edge from $v$. For $v_1, v_2 \in \mathcalb{T}$, we will write $v_1 \rightarrow v_2$ to denote that $v_1$ is a child of $v_2$, and $v_1\uparrow e(v_2)$ to denote $v_1$ descends from $e(v_2)$ (via a possible chain of ancestral edge reproductions). 
As before, $\E_v$ and $\prob_v$ will respectively denote the conditional expectation and probability given the history of the network observed up to the birth time of vertex $v$.

By~\eqref{couplingdyn}, there is an independent rate one Poisson process $N_v$ associated with each $v \in \TT_c$ encoding the times at which the edge $e(v)$ gives birth. 
For $v \in \TT_c \setminus \{\rho\}$, writing $\tau_v$ for the birth time of $v$ and $i(v)$ for the chronological index of $v$, define the event
$$
\mathcal{A}_v := \left\lbrace N_v\left(\int_0^t\hat r_{i(v)}(s)ds\right) = N_v\left(\int_0^tr(s-\tau_v)ds\right) \ \forall t \ge \tau_v\right\rbrace.
$$
Recall that the $l$-th child of the edge $e(v)$ in $\TT_c$ chooses to attach to $v$ or its parent vertex with equal chance. We denote the associated Bernoulli random variable as $B_{v,l}$.

Fix $M>0$. Intuitively we will think of $M$ large but fixed so that various functionals involved in the coupling have ``stabilized' in $\cT_c$ by the time it reaches time $M$ and one can potentially run two different coupled processes, one having dynamics modulated by the probabilistic rules in $\BP_{\Ma}$ while the other run via the dynamics of $\cT_c$, such that miscouplings can be handled.\\ 

\noindent\textbf{The coupling: }Let us now be more precise. The coupling involves the construction of a growing tree process $\TT_{M,c}(\cdot)$ which attempts to retain the maximal amount of genealogical and temporal information from $\TT_c$ while keeping the law of the reproduction point process of each edge $e(v)$ born after time $M$ the same as that in $\BP_{\Ma,\ve}$. In the following, we say a vertex $v \in \TT_{M,c}$ is \emph{copied} from $\TT_c$ if there exist vertices in $\TT_c$ with the same values in $\mathcal{U} \times \mathbb{R}_+$ as $v$ and its ancestors in $\TT_{M,c}$. In this case, we will also use $v$ to denote the corresponding vertex in $\TT_c$. Let $v_M$ denote the first vertex born into $\TT_{c}$ after time $M$ that attaches to the root $\rho$.
The construction is now described as follows:
\begin{enumeratei}
\item $\TT_{M,c}(s) = \TT_c(s)$ for all $s \in [0,\tau_{v_M})$, that is, all vertices born into $\TT_c$ on this interval, along with their birth times and edges spanning them, are copied from $\TT_c$ onto $\TT_{M,c}$. Moreover, copy all the vertices that are produced by reproduction of these edges.
\item For each vertex $v \in \TT_{M,c}$ with birth time $\tau^M_v \ge \tau_{v_M}$, the edge $e(v)$ reproduces at times given by the epochs of $\{N_v\left(\int_0^tr(s-\tau^M_v)ds\right) : t \ge \tau^M_v\}$ and the newly born edges (equivalently, vertices) connect to $v$ or its parent vertex with equal chance. Moreover, if $v$ is copied, the $l$-th newly born edge (equivalently, vertex) in the above point process connects to $v$ or its parent according to the value of the same Bernoulli variable $B_{v,l}$ governing the corresponding attachment in $\TT_c$.  
\end{enumeratei}
In other words, all children of vertices born before time $M$ in $\TT_c$ are copied onto $\TT_{M,c}$ and those born after time $M$ reproduce in a coupled way attempting to maximize the number of copied children on any compact time interval.

Observe that if $v \in \TT_{M,c}$ is copied and $\mathcal{A}_v$ holds, then all the vertices originating from the reproduction of $e(v)$ are copied as well.\\

\noindent\textbf{Quantifying number of miscoupled vertices: }
Now, we will quantify the number of vertices in $\TT_{M,c}$ (born by a given time) which descend from $e(v_M)$, and which attach to the root $\rho$ in $\TT_{M,c}$ and have not been copied. We will call such vertices `miscoupled'. Define the following for $t \ge 0$:
\begin{align*}
Z_{M,v}(t) &:= \#\{u \in \TT_{M,c}: \, u \uparrow e(v), \, u \rightarrow \rho \text{ and } \tau_u \le t\}, \ v \in \TT_{M,c}, \ v \rightarrow \rho,\\
C_M(t) &:= \{u \in \TT_{M,c} \textbf{ copied}: \, u \uparrow e(v_M), \, u \rightarrow \rho \text{ and } \tau_u \le t\},\\
S_M(t) &:= \{u \in C_M(t) \cup \{v_M\} : \mathcal{A}_u^c \text{ holds}\}.
\end{align*}

In words, $Z_{M,v}(t)$ denotes the number of vertices in $\TT_{M,c}$ born up to time $t$, descending from the edge $e(v)$, which attach to the root. $C_M(t)$ denotes the number of \textbf{copied} vertices in $\TT_{M,c}$ born up to time $t$, descending from the edge $e(v_M)$, which attach to the root.
$S_M(t)$ is the set of `bad' vertices in $C_M(t) \cup\{v_M\}$ in the sense that at least one of their children is not copied.

Observe that every vertex born in $[0,t]$ that is not copied descends from a vertex in $S_M(t)$. Hence, the number of miscoupled vertices in $\TT_{M,c}$ with birth times in $[0,t]$ is upper bounded by
$$
\Mis(t) := \sum_{v \in S_M(t)}Z_{M,v}(t).
$$
Our proof of the root degree lower bound relies on the fact
\begin{equation}\label{discrepancy}
\rho(t) \ge |C_M(t)| \ge Z_{M,v_M}(t) - \Mis(t), \ t \ge \tau_{v_M}.
\end{equation}
We will quantify $\Mis(t)$ and show that, with high probability, it is much smaller than $Z_{M,v_M}(t)$ for large $M$. This is achieved by the following lemma.
\begin{lemma}\label{misclem}
Assume $\eta$ has a density $f_\eta$ such that $\sup_{u \ge 0} e^u f_\eta(u) < \infty$, and there exists $\gamma>1$ such that $\E(e^{\gamma \eta}) < \infty$. There exists $\phi \in (\lambda, 1)$, $C,\delta, M_0>0$ such that for all $M \ge M_0$ and $t \ge M$,
$$
\prob\left[\Mis(t) \ge e^{-\phi M}e^{\lambda t}\right] \le C e^{-\delta M}.
$$
\end{lemma}
The proof of the above lemma involves viewing $Z_{M,v}(t)$, for $v \in C_M(t) \cup \{v_M\}$, as the size of the branching process $\BP^\circ_{\ve}(\cdot)$ introduced in Section~\ref{edgedesc} run up till time $t - \tau_v$. An individual $u$ in this branching process can be marked as either `good' or `bad' depending on whether or not $\mathcal{A}_u$ holds. Quantitative estimates on $\prob_u(\mathcal{A}_u^c)$, derived from Lemma~\ref{ratecomp}, are then used to bound the number of bad individuals via renewal theory.

We will need the following estimate for the Yule process $Y(\cdot)$ whose proof is given in Appendix~\ref{app}.
\begin{lemma}\label{Yuleupdown}
    For any $\delta \in (0,1/2)$, there exist positive constants $C_1, C_2,M_0$ depending on $\delta$ such that for all $M \ge M_0$,
    \begin{align*}
        \prob\left(e^{-s}Y(s) \le e^{-\delta s} \text{ for some } s \ge M\right) \le C_1e^{-C_2M}.
    \end{align*}
\end{lemma}

\begin{proof}[Proof of Lemma~\ref{misclem}]
For $\epsilon \in (0,1)$, define the subset of vertices
\begin{align*}
\mathcal{S}_{M,\epsilon} := \{v \in \TT_{M,c}  \text{ copied and } 
\tau_v \le (1+\epsilon)\log i(v)\}
\end{align*}
and the event
$$
\mathcal{I}_{M,\epsilon} := \{v \in \mathcal{S}_{M,\epsilon} \text{ whenever } v \in \TT_{M,c} \text{ is copied and } \tau_v \ge M\}.
$$
For any $\phi \in (\lambda,1)$,
\begin{multline}\label{misbd}
\prob\left[\Mis(t) \ge e^{-\phi M}e^{\lambda t}, \, \mathcal{I}_{M,\epsilon}\right] \le e^{\phi M -\lambda t} \E\left[\sum_{v \in C_M(t) \cup\{v_M\}} \ind\set{\mathcal{A}_v^c \cap \{v \in \mathcal{S}_{M,\epsilon}\}} Z_{M,v}(t)\right]\\
 \le e^{\phi M -\lambda t}\E\left[\sum_{v \in C_M(t)\cup\{v_M\}} \ind\set{v \in \mathcal{S}_{M,\epsilon}} \left(\prob_v(\mathcal{A}_v^c)\right)^{1/2} \left(\E_{v}\left(Z_{M,v}(t)^2\right)\right)^{1/2}\right].
\end{multline}
By Markov's inequality and Lemma~\ref{ratecomp} (this time requiring both bounds), for any $\delta>0$,
\begin{align*}
    \prob_v(\mathcal{A}_v^c) &\le \mathbb{E}_v\left[\int_{\tau_v}^t\left|\hat r_{i(v)}(s) - r(s - \tau_v)\right|ds\right]\\
    &\le \int_{\tau_v}^{(1+ \delta)\tau_v}C_1\left(i(v)^{-\frac{\gamma-1}{2\gamma}} e^{-C_2 (s-\tau_v)} + e^{-C_2 i(v)} e^{\frac{\gamma-1}{2\gamma}(s-\tau_v)}\right)ds\\
    &\quad + \int_{(1+ \delta)\tau_v}^t C_1\left(e^{-(\gamma-1)(s-\tau_v)/2} + e^{-(s-\tau_v)/2}\right)ds\\
    &\le C_1'\left(i(v)^{-\frac{\gamma-1}{2\gamma}} + e^{-C_2 i(v)} e^{\frac{(\gamma-1)\delta}{2\gamma}\tau_v} + e^{-C_3\delta \tau_v}\right),
\end{align*}
where the constants do not depend on $\delta$. Hence, on the event $\{v \in \mathcal{S}_{M,\epsilon}\}$,
\begin{equation}\label{aprob}
    \prob_v(\mathcal{A}_v^c) \le C_1'\left(e^{-\frac{\gamma - 1}{2\gamma(1 + \epsilon)}\tau_v} + e^{-C_2e^{\tau_v/(1+\epsilon)}}e^{\frac{(\gamma-1)\delta}{2\gamma}\tau_v} + e^{-C_3\delta \tau_v}\right)\le C e^{-C'\tau_v}
\end{equation}
on choosing and fixing $\delta = \delta(\epsilon)$ sufficiently small (the constants $C,C'$ depend on $\epsilon$).

Moreover, by the assumed hypothesis $\E(e^{\eta}) < \infty$, $\E\left(\left[\zeta_{\ve}(t)\right]^2\right) \le C t^2$ for all $t \ge 0$. From this and Jensen's inequality, we conclude that $\E\left(\left[\hat \zeta_{\ve}(\lambda)\right]^2\right) < \infty$ which, by Lemma~\ref{lem:bpe-moment}(b), gives
\begin{equation}\label{zless}
\E_{v}\left(Z_{M,v}(t)^2\right) \le C e^{2\lambda (t - \tau_v)}.
\end{equation}
Using~\eqref{aprob} and~\eqref{zless} in~\eqref{misbd}, we get
\begin{align*}
\prob\left[\Mis(t) \ge e^{-\phi M}e^{\lambda t}, \, \mathcal{I}_{M,\epsilon}\right] &\le C e^{\phi M -\lambda t} \E\left[\sum_{v \in C_M(t)\cup\{v_M\}} \ind\set{v \in \mathcal{S}_{M,\epsilon}} e^{-C'\tau_v}e^{\lambda (t - \tau_v)}\right]\nonumber\\
&\le C e^{\phi M}\E\left[\sum_{v \in C_M(t)\cup\{v_M\}} e^{-(\lambda + C')\tau_v}\right].
\end{align*}
Now, as vertices in $C_M(t)$ descend from $e(v_M)$, the above expectation can be bounded in terms of the birth times $\{\tau^\circ_v : v \in \BP^\circ_{\ve}(t)\}$, where $\BP^\circ_{\ve}(\cdot)$ is the edge branching process determining root degree that was introduced in Section~\ref{edgedesc}. Hence, we obtain
\begin{align}\label{mis1}
\prob\left[\Mis(t) \ge e^{-\phi M}e^{\lambda t}, \, \mathcal{I}_{M,\epsilon}\right] \le e^{-(\lambda + C' - \phi)M}\E\left[\sum_{v \in \BP^\circ_{\ve}(t)} e^{-(\lambda + C')\tau^\circ_v}\right].
\end{align}
Writing $m^\circ(t) := \E\left[\sum_{v \in \BP^\circ_{\ve}(t)} e^{-(\lambda + C')\tau^\circ_v}\right]$, note that $m^\circ$ satisfies the following renewal-type equation
\begin{equation*}
m^\circ(t) = 1 + \int_0^te^{-(\lambda+ C')s} m^\circ(t-s)\mu_{\ve}(ds).
\end{equation*}
As $\int_0^\infty e^{-\lambda s}\mu_{\ve}(ds) = 1$ and $\mu_{\ve}(0,\infty)>0$, we conclude
$
\hat{\gamma} := \int_0^\infty e^{-(\lambda + C')s}\mu_{\ve}(ds) <1
$
and hence 
$$
\sup_{t \ge 0}m^\circ(t) = \frac{1}{1-\hat\gamma} < \infty.
$$
Using this in~\eqref{mis1} gives
\begin{equation}\label{mis2}
\prob\left[\Mis(t) \ge e^{-\phi M}e^{\lambda t}, \, \mathcal{I}_{M,\epsilon}\right] \le \frac{C}{1-\hat\gamma} e^{-(\lambda + C' - \phi)M}.
\end{equation}
Now, we bound $\prob[\mathcal{I}_{M,\epsilon}^c]$. Note that, if for some $v \in \TT_{M,c}$ copied with $\tau_v \ge M$, we have $\tau_v \ge (1 + \epsilon)\log i(v)$, then
$$
e^{-\tau_v}Y(\tau_v) = e^{-\tau_v} i(v) \le e^{-\epsilon \tau_v/(1+ \epsilon)}.
$$

By Lemma~\ref{Yuleupdown}, choosing and fixing $\epsilon \in (0,1)$, we thus obtain for $  \ge M_1$,
\begin{equation}\label{mis3}
\prob[\mathcal{I}_{M,\epsilon}^c] \le C_1e^{-C_2M}.
\end{equation}
The lemma follows from~\eqref{mis2} and~\eqref{mis3} upon choosing any $\phi \in (\lambda, \lambda + C')$.
\end{proof}
\textbf{Completing the proof of lower bound in Theorem~\ref{thm:macro-max-deg}:}
All the constants $C,C',\dots$ appearing in the proof are independent of $M$. By Lemma~\ref{lem:bpe-moment}(b),
\begin{equation*}
e^{-\lambda(t-\tau_{v_M})}Z_{M,v_M}(t) \stackrel{a.s.}{\rightarrow} W_{M,\ve}, \text{ as } t \to \infty
\end{equation*}
where $W_{M,\ve} \stackrel{d}{=} W_{\ve}$ is a strictly positive finite random variable. 

Choose and fix $\theta>0$ such that $\lambda(1+ \theta) < \phi$, where $\phi \in (\lambda,1)$ is as in Lemma~\ref{misclem}. As new vertices are added in $\TT_c$ according to a Yule process $Y(\cdot)$ and a new vertex arriving at time $t$ attaches to the root $\rho$ with probability at least $1/(2Y(t-))$, it follows from this that the law of $\tau_{v_M} - M$ is stochastically dominated by $\exp(1/2)$. Hence,
\begin{align*}
    \prob\left(\tau_{v_M} \ge (1+ \theta)M\right) \le \prob(\exp(1/2) \ge \theta M) \le e^{-\theta M/2}.
\end{align*}

From these observations and Fatou's lemma, we conclude that for any $x>0$,
\begin{align*}
    \limsup_{t \to \infty}\prob\left(e^{-\lambda t}Z_{M,v_M}(t) \le e^{-\lambda(1 + \theta)M}x\right) &\le e^{-\theta M/2} + \limsup_{t \to \infty}\prob\left(e^{-\lambda (t-\tau_{v_M})}Z_{M,v_M}(t) \le x\right)\nonumber\\
    &\le e^{-\theta M/2} + \prob\left(W_{\ve} \le x\right).
\end{align*}
Let $\sigma'_1$ denote the first arrival time in the Poisson point process $\mvzeta_{\ve}$ governing the edge branching process $\BP^\circ_{\ve}$. Then it is straightforward to check that $\E\left[e^{\lambda a \sigma'_1}\right] < \infty$ for any $a \in (0, \E[e^\eta]/\lambda)$. By using the same argument as in the proof of~\cite[Lemma 3.5]{banerjee2023degree}, we conclude that for any such $a$, there exists $C_a>0$ such that
$$
\prob\left(W_{\ve} \le x\right) \le C_a x^{a}, \, x>0.
$$
Choosing and fixing such an $a$, and setting $x=2e^{-(\phi - \lambda(1+ \theta))M}$, we thus have
\begin{equation}\label{lbfin1}
\limsup_{t \to \infty}\prob\left(e^{-\lambda t}Z_{M,v_M}(t) \le 2 e^{-\phi M}\right) \le C e^{-C'M}.
\end{equation}
Now, leveraging the key bound in~\eqref{discrepancy} and using~\eqref{lbfin1},
\begin{align}\label{lbfin2}
    &\limsup_{t \to \infty}\prob\left(e^{-\lambda t} \rho(t) < e^{-\phi M}\right)\nonumber\\
    &\le \limsup_{t \to \infty}\prob\left( e^{-\lambda t} \left(Z_{M,v_M}(t) - \Mis(t)\right) <  e^{-\phi M}\right)\nonumber\\
    &\le \limsup_{t \to \infty}\prob\left(e^{-\lambda t} Z_{M,v_M}(t) < 2 e^{-\phi M}\right) + \limsup_{t \to \infty}\prob\left( e^{-\lambda t} \Mis(t) \ge e^{-\phi M}\right)\nonumber\\
    &\le C e^{-C'M} + C e^{-\delta M} \le Ce^{-C''M},
\end{align}
for all $M \ge M_0$, where we have used Lemma~\ref{misclem} in the last step.

Finally, we translate the above result to a discrete-time lower bound on the root degree. Recall that $\{M(\rho,n) : n \ge 1\} \stackrel{d}{=}\{\rho(\tau_n) : n \ge 1\}$. Hence, using~\eqref{lbfin2}, we obtain for all $M \ge M_0$,
\begin{align*}
  &\limsup_{n \to \infty}\prob\left(n^{-\lambda}M(\rho,n) \le e^{-\phi(1 + \lambda) M}\right)\\
  &\qquad\le \limsup_{n \to \infty}\prob(\tau_n \le \log n - \phi M) + \limsup_{n \to \infty}\prob\left(n^{-\lambda}\rho(\log n - \phi M) \le e^{-\phi(1 + \lambda) M}\right)\\
  &\qquad= \limsup_{n \to \infty}\prob(Y(\log n - \phi M) \ge n) + \limsup_{n \to \infty}\prob\left(e^{-\lambda (\log n - \phi M)}\rho(\log n - \phi M) \le e^{-\phi M}\right)\\
  &\qquad\le e^{-\phi M} + Ce^{-C''M},
\end{align*}
where we have used the fact that $Y(t) \sim \operatorname{Geom}(e^{-t}), \, t \ge 0$, in the last inequality. The lower bound in Theorem~\ref{thm:macro-max-deg} follows from the above.

\section{Conclusion}
\label{sec:conc}

The main goal of this paper is to formulate models of network evolution that incorporate delay as well as preferential attachment and initiate the study of asymptotics of specific functionals. This class of models suggests a plethora of open directions, which we now describe. 
\begin{enumeratea}
    \item {\bf Mesoscopic regime: } This paper dealt with the macroscopic regime. The regime $\beta <1$ was studied in~\cite{BBDS04_meso} for not just linear but more general attachment functions.   Under regularity conditions on the attachment function $f$ and minor moment conditions on the delay distribution $\mu$,  {\bf irrespective} of the parameter $\gb < 1$ and delay distribution $\mu$, it is shown that the local weak limit of the entire graph {\bf are the same} as in the setting without delay.
    
    \item {\bf Non-{\it i.i.d.}~delays:} The paper states results for {\it i.i.d.}~delays, however if one looks at the proofs, many of the results should go through assuming a sequence of independent delays $\set{\xi_n:n\geq 1}$ that satisfy for example convergence of the empirical distribution to $\mu$ as $n\to\infty$. We leave such extensions (and limits thereof) to other interested researchers.

    \item {\bf Non-tree regime:} To keep this paper to a manageable length, we have dealt with the setting where the network stream is a collection of trees. All of the main results for both the meso and macro regimes should be extendable to the more general network setting where vertices enter with more than one edge that they recursively use to connect via probabilistic choices with information limited by delay. The no delay setting has witnessed significant interest over the last few years~\cite{garavaglia2020local,garavaglia2022universality,banerjee2023local}. 
    
    \item {\bf General attachment functions in the Macroscopic regime:} In the macroscopic regime, this paper only dealt with the affine linear attachment function, and this model itself led to non-trivial asymptotics, with even the root degree requiring detailed technical analysis. Understanding the macroscopic regime for general attachment functions and the types of condensation phenomenon in this more general setting is a worthy next endeavor. 
    
    \item {\bf General delay distributions in the Macroscopic regime:} In this regime, we largely considered two subclasses of distributions: 
    \begin{inparaenuma}
    \item {\bf Delay with mass at one:} In this case, each new vertex has a positive probability of attaching to the root, and thus, it is not surprising that the root degree grows like $\Theta_P(n)$ and there is condensation; perhaps the most surprising finding in this setting is that (in the linear attachment function setting), this is the only way to achieve condensation. 
    \item {\bf Transformed delay with finite exponential moments:} The other regime (Theorem~\ref{thm:non-conden} and its Corollaries) is the setting where the (transformed delay, namely $\eta$) has exponential moments. 
    \end{inparaenuma}
    All other regimes, especially settings where $\pr(\eta < \infty) =1$ but $\E(e^{\theta \eta}) = \infty$ for all $\theta >0$ are completely unexplored and suggest a fascinating plethora of directions. Extending the results on root degree asymptotics in the regimes considered in this paper to the maximal degree would be another interesting direction to pursue. 
\end{enumeratea}

\appendix
\section{Proofs of auxiliary lemmas}\label{app}

\begin{proof}[Proof of Lemma~\ref{lem:bigggg-lemma}]
We will prove the bounds for $\sup_{\mvx \in \cC}\left|\vP_{i,n}(\mvx)-\vP_{i+1,n}(\mvx)\right|$ for each $i=0,1,2$ separately.
\begin{enumeratei}
    \item {\bf{When $i=0$}:} Since $\left|\exp(-x) - (1-x)\right| \leq x^2$ for all $x\in \bR_+$ and using the observation that {$p_j^n(i,\mvx)\leq \frac{K+1 +\alpha}{2+\alpha}\frac{1}{n}$} for all $0 \le j \le K$, $i\in \bN$, we have from~\eqref{eqn:prob-Ex} and~\eqref{eqn:prob-ex-1},
    \begin{align*}
    &\left|\vP_{0,n}(\mvx) - \vP_{1,n}(\mvx) \right| \\
    &\leq \sum_{i\in \cI_{\mvx}}\left|\exp\left(-\sum\limits_{j=0}^K p_j^n(i,\mvx)\right) - \left(1-\sum\limits_{j=0}^K p_j^n(i,\mvx)\right)\right|\leq \sum_{i\in \cI_{\mvx}} \left(\sum\limits_{j=0}^K p_j^n(i,\mvx)\right)^2\\
    &\leq \sum_{i=n+1}^{N_{\mvx}} \left(\frac{(K+1)(K+1 + \alpha)}{(2+\alpha)n}\right)^2 = \frac{N_{\mvx}-n}{n^2}\left(\frac{(K+1)(K+ 1 + \alpha)}{(2+\alpha)}\right)^2,\end{align*}
    which proves the case $i=0$ upon noting that $N_{\mvx} \leq n+ \sum\limits_{j=0}^K \sum\limits_{k=1}^K x_{j,k}$.
\item {\bf{When $i=1$}:} For all $\mvx \in \cC$, since $T^n_{j,k}(\mvx) \geq n$, we have for all $i\in \bN$ 
\begin{align*}
    &\left|p_j^n(i,\mvx)-q_j^n(i,\mvx)\right|\\
    &\leq \frac{1}{2+\alpha}\left[\sum\limits_{k=0}^K\E\left(\frac{\ind\set{i(1-\xi) \geq T_{j,k}^n(\mvx)}}{[i(1-\xi)]^2}\right) + \alpha \E\left(\frac{\ind\set{i(1-\xi) \geq \rho_j^n(\mvx)}}{[i(1-\xi)]^2}\right)\right]\\
    &\leq \frac{K+ 1 +\alpha}{2+\alpha}\frac{1}{n^2}.
\end{align*}
Therefore, we have
\begin{align*}
&\left|\vP_{1,n}(\mvx) - \vP_{2,n}(\mvx) \right|\\
 &\leq \sum\limits_{i\in \cI_{\mvx}}\sum\limits_{j=0}^K\left|p_j^n(i,\mvx)-q_j^n(i,\mvx)\right|+\sum\limits_{j,k\leq K}\left|p_{j}^n\left(T_{j,k}^n(\mvx),\mvx\right)-q_{j}^n\left(T_{j,k}^n(\mvx),\mvx\right)\right|\\
&\leq \frac{(K+1)(K+1+\alpha)}{2+\alpha} \frac{N_{\mvx}-n}{n^2} + \frac{K(K+1)(K+1+\alpha)}{2+\alpha}\frac{1}{n^2}, 
    \end{align*}
    which proves the case $i=1$.

\item {\bf{When $i=2$}:} Note for that for any $N\in \bN$, we have \begin{align}\label{eqn:log-approximation}
    \sup\limits_{N \leq z_1\leq z_2 }\left|\sum_{i}\frac{\ind\set{z_1\leq i\leq z_2}}{i} - \log\left(\frac{z_2}{z_1}\right)
    \right| \leq \frac{5}{N}.
\end{align} 
{Moreover, by a simple change of variables formula, for any $0 \le k \le K$,
$$
H\left(\log\left(\frac{T_{j,K}^n(\mvx)}{T_{j,k}^n(\mvx)}\right)\right) = \E\left(\ind\set{(1-\xi)^{-1} \le \frac{T_{j,K}^n(\mvx)}{T_{j,k}^n(\mvx)}}(1-\xi)^{-1}\log\left(
\frac{T^n_{j,K}(\mvx)(1-\xi)}{T^n_{j,k}(\mvx)}\right)\right).
$$}
Hence, for any $0 \le j \le K,0 \le k\leq K$, using~\eqref{eqn:log-approximation} we have \begin{align*}
    &\left|\sum_{i=n}^{T^n_{j,K}(\mvx)} \E\left(\frac{1}{2+\alpha} \frac{\ind\set{i(1-\xi)\geq T_{j,k}^n(\mvx)}}{i(1-\xi)}\right) -H\left(\log\left(\frac{T_{j,K}^n(\mvx)}{T_{j,k}^n(\mvx)}\right)\right) \right|& \\
    &\leq \frac{1}{2+\alpha}\E\left(\frac{\ind\set{ T^n_{j,K}(\mvx)(1-\xi) \geq T^n_{j,k}(\mvx)}}{1-\xi}  \right. \\ 
    &{\hspace{1cm} \left. \cdot \left|\sum_{i=n}^{T^n_{j,K}(\mvx)}\frac{\ind\set{ T^n_{j,k}(\mvx)(1 - \xi)^{-1} \leq i \leq T^n_{j,K}(\mvx) }}{i} - \log\left(
\frac{T^n_{j,K}(\mvx)(1-\xi)}{T^n_{j,k}(\mvx)}\right)\right|\right)}\\
 &\leq  \frac{1}{2+\alpha}\frac{T^n_{j,K}(\mvx)}{T^n_{j,k}(\mvx)} \frac{5}{n} \leq  \frac{1}{2+\alpha}\frac{5N_{\mvx}}{n^2}.
\end{align*}
Thus,
    \begin{align*}
        \left|\vP_{2,n}(\mvx)-\vP_{3,n}(\mvx)\right|&\leq \sum_{j=0}^K\left|\sum_{i=n}^{T^n_{j,K}(\mvx)}q_j^n(i,\mvx) - r_j^n(\mvx)\right| {\leq \frac{C'K^2}{n}}
    \end{align*}
for some constant $C'>0$ depending only on $\alpha$, proving case $i=2$. 
\end{enumeratei}
This proves the lemma.
\end{proof}

{ 
\begin{proof}[Proof of Lemma~\ref{lem: invalid-set}]
\noindent \textbf{(1)} Observe that $\hat{\rho}^n_j(\mvy) \leq \hat{T}^n_{j,k}(\mvy) \leq \hat{\rho}^n_j(\mvy) \exp(\sum_{i=1}^k y_{j,i}).$ Moreover, we have $$\hat{T}^n_{j,k}(\mvy) - \hat{T}^n_{j,k-1}(\mvy) \geq \hat{\rho}^n_j(\mvy) \exp(\sum_{i=1}^k y_{j,i}) - \hat{\rho}^n_j(\mvy) \exp(\sum_{i=1}^{k-1} y_{j,i}) -1$$ for $1 \le k \le K$ (with convention $\exp(\sum_{i=1}^{0} y_{j,i}) =1$), which gives $\hat{T}^n_{j,k}(\mvy) \ge \hat{\rho}^n_j(\mvy) \exp(\sum_{i=1}^k y_{j,i}) - k$. Therefore $\hat{\tau}^n_{j,k}(\mvy) \leq L$ implies
    $$
    e^{y_{j,k}} - \frac{k}{n} \le e^{y_{j,k}} - \frac{k}{\hat{\rho}^n_j(\mvy) \exp(\sum_{i=1}^{k-1} y_{j,i})} \le e^L \ \Rightarrow \ y_{j,k} \le \log\left(e^L + K\right) \ \forall \ j,k \le K.
    $$
    Hence the support of the integrand on the right-hand side of the above display is contained in
    \[
    [0, \log\left(e^L + K\right)]^{K^2+K}.
    \]
Therefore, if  $\hat{\tau}_{1:K}^n(\mvy) \in \cS$, using ~\eqref{eqn:bound-for-dct}, and hence $\vP_{3,n}(G_n(\mvy)) \fJ_n(\mvy)$, is uniformly bounded (over $n$). This proves (1) of the lemma.

\vspace{0.4cm}
    
\noindent \textbf{(2)} Note that $G_n(\mvy)$ is not \emph{valid} if and only if in the $\log$-scale the arrival times of two vertices in the fringe of the vertex, $n$ are equal on the event $E_{\mvx}$. Observe that every arrival time in $\log$-scale can be expressed as a sum of some elements in $\hat{\tau}^n_{1:K}(\mvy) = (\hat{\tau}^n_{j,k}:0\leq j\leq K, 1\leq k\leq K)$. Let $S_1(\hat{\tau}^n_{1:K}(\mvy))$ and $S_2(\hat{\tau}^n_{1:K}(\mvy))$ denote two partial sums of subsets of coordinates in $\hat{\tau}^n_{1:K}(\mvy)$. We first show $$\lim\limits_{n\to \infty} \lambda(\set{\mvy:S_1(\hat{\tau}^n_{1:K}(\mvy)) = S_2(\hat{\tau}^n_{1:K}(\mvy))}\cap \cL) \to 0.$$ Observe that if $S_1(\hat{\tau}^n_{1:K}(\mvy)) = S_2(\hat{\tau}^n_{1:K}(\mvy))$, then \begin{align*}
        |S_1(\mvy) -  S_2(\mvy)| &=  |S_1(\hat{\tau}^n_{1:K}(\mvy)) - S_1(\mvy) + S_2(\mvy) - S_2(\hat{\tau}^n_{1:K}(\mvy))|\\
        &\leq C'' \sup_{j,k}|\hat{\tau}^n_{j,k}(\mvy) - y_{j,k}|
    \end{align*} for some constant $C''$ (depending on $K$), as $S_1$ and $S_2$ are linear functions of $\hat{\tau}^n_{j,k}$ and hence Lipschitz globally. Also, using the bounds on $\hat{T}^n_{j,k}(\mvy)$ obtained just before equation~\eqref{eqn:bound-for-dct}, we have $$\log\left(1-\frac{Ke^{-y_{j,k}}}{n}\right)\leq \hat{\tau}^n_{j,k}(\mvy) - y_{j,k} \leq \log\left(1+\frac{K}{n-K}\right). $$ Therefore, there exists a non-negative sequence $\set{c_n}_{n\geq 1}$ such that $c_n \to 0$ such that for all $\mvy \in \cL$, we have $$\sup_{j,k}|\hat{\tau}^n_{j,k}(\mvy) - y_{j,k}| \leq c_n.$$ Hence, \begin{align*}
        \lim\limits_{n\to \infty}\lambda(\set{\mvy:S_1(\hat{\tau}^n_{1:K}(\mvy)) = S_2(\hat{\tau}^n_{1:K}(\mvy))} \cap \cL ) &\leq \lim\limits_{n\to \infty}\lambda(\set{\mvy:|S_1(\mvy)-S_2(\mvy)| \leq C''c_n} \cap \cL)\\
        &= \lambda(\set{\mvy: S_1(\mvy) = S_2(\mvy)} \cap \cL) = 0. 
    \end{align*} Now, a union bound over all possible pairs of partial sums yields the result.
\end{proof}}

\begin{proof}[Proof of Lemma~\ref{Yuleest}]
    (i) follows from standard results on the Yule process, see~\cite[Section 2.5]{norris-mc-book}. To prove (ii), note that by using moment-generating functions and Markov's inequality, for any $\theta>0$,
    $$
    \prob_i\left(W_i < i x\right) \le e^{\theta i x}\frac{1}{(1+\theta)^i}, \ x \in [0,1).
    $$
    The claim follows by optimizing over $\theta$. Now, we prove (iii). Observe that
    $$
    \E_i \left| \log \frac{W_i}{i} \right| = \frac{1}{\Gamma(i)}\int_0^\infty \left| \log \frac{z}{i} \right| z^{i-1}e^{-z} dz = \frac{i^i}{\Gamma(i)}\int_{-\infty}^\infty \left| y \right| e^{i(y-e^y)} dy.
    $$
    Using the fact that $e^y \ge 1 + y + \frac{y^2}{2}$ and $e^{-y} \ge 1-y+\frac{y^2}{2} - \frac{y^3}{6}$ for all $y \ge 0$, we estimate the integrals as follows:
    \begin{align*}
        \frac{i^i}{\Gamma(i)}\int_0^{\infty} y e^{i(y-e^y)} dy \le  \frac{i^i}{\Gamma(i)}\int_0^{\infty} y e^{i(-1 -\frac{y^2}{2})} dy \le \frac{Ci^ie^{-i}}{i\Gamma(i)} \le \frac{C'}{\sqrt{i}}
    \end{align*}
    and
    \begin{align*}
        \frac{i^i}{\Gamma(i)}\int_{-\infty}^0 \left| y \right| e^{i(y-e^y)} dy &= \frac{i^i}{\Gamma(i)}\int_0^{\infty} y e^{i(-y-e^{-y})} dy\\
        &\le \frac{i^i}{\Gamma(i)}\int_0^1 y e^{i(-1-\frac{y^2}{2} + \frac{y^3}{6})} dy + \frac{i^i}{\Gamma(i)}\int_1^{\infty} y e^{-iy} dy\\
        &\le \frac{i^i}{\Gamma(i)}\int_0^1 y e^{i(-1-\frac{y^2}{3})} dy + \frac{i^i}{\Gamma(i)}\int_1^{\infty} y e^{-iy} dy
        \le \frac{Ci^ie^{-i}}{i\Gamma(i)} \le \frac{C'}{\sqrt{i}},
    \end{align*}
    where we have used Stirling's approximation to obtain the final bounds. The claim follows.
    
    To prove (iv), note that under $\prob_i$, for $t  > \tau_i$, $Y(t)$ is distributed as the sum of $i$ $\operatorname{Geom}(p)$ random variables, where $p=e^{-(t-\tau_i)}$. Using moment-generating functions and Markov's inequality, for any $\theta>0$,
    $$
    \prob_i\left(Y(t) \le ix\right) \le e^{\theta i x}\left(\frac{pe^{-\theta}}{1-e^{-\theta} + pe^{-\theta}}\right)^i, \ x \in [1,1/p).
    $$
    The bound above is minimized at $\theta = -\log\left((1-x^{-1})(1-p)^{-1}\right)$. Hence,
    $$
     \prob_i\left(Y(t) \le ix\right) \le p^ix^{ix}\left(\frac{1-p}{x-1}\right)^{i(x-1)}\le (px)^i, \ x \in [1,1/p).
    $$
    Using this bound for $x = 1/\sqrt{p}$, we obtain
    $$
    \prob_i\left(Y(t) \le ie^{-(t-\tau_i)/2}\right) \le e^{-i(t-\tau_i)/2}, \ t > \tau_i,
    $$
which implies the result.
\end{proof}

\begin{proof}[Proof of Lemma~\ref{Yuleupdown}]
    Recall from Lemma~\ref{Yuleest}(i) that $\{\mathcal{M}(t) := e^{-t}Y(t) : t \ge 0\}$ is a martingale. It's quadratic variation $\langle \mathcal{M}\rangle$ is given by
$$
\frac{d}{dt}\langle \mathcal{M}\rangle (t) = -2e^{-2t}Y(t)^2 + e^{-2t}\left((Y(t) + 1)^2 - Y(t)^2\right)Y(t) = e^{-2t} Y(t).
$$
Hence, by Doob's $L^2$-maximal inequality,
$$
\E\left[\sup_{s \ge t} \left| e^{-s}Y(s) - W\right|^2\right] \le C\E\left[\int_{t}^\infty e^{-2s}Y(s)ds\right] = C\int_{t}^\infty e^{-s}ds = Ce^{-t}, \ t \ge 0.
$$
Using this, we obtain for any $\delta \in (0,1/2)$,
\begin{align*}
&\prob\left(\left| e^{-s}Y(s) - W\right| \ge e^{-\delta s} \text{ for some } s \ge M\right)\\
&\qquad\le \sum_{l = \lfloor M \rfloor}^\infty \prob\left(\left| e^{-s}Y(s) - W\right| \ge e^{-\delta s} \text{ for some } s \in [l,l+1]\right)\\
&\qquad\le \sum_{l = \lfloor M \rfloor}^\infty \prob\left(\sup_{s \ge l}\left| e^{-s}Y(s) - W\right| \ge e^{-\delta (l+1)}\right)
 \le C\sum_{l = \lfloor M \rfloor}^\infty e^{2\delta (l+1)}e^{-l} \le C' e^{-(1- 2\delta)M},
\end{align*}
where $C'$ depends on $\delta$. Thus,
\begin{align*}
&\prob\left(e^{-s}Y(s) \le e^{-\delta s} \text{ for some } s \ge M\right)\\
&\qquad\le \prob\left(\left| e^{-s}Y(s) - W\right| \ge e^{-\delta s} \text{ for some } s \ge M\right) + \prob(W \le 2e^{-\delta M})\\
&\qquad \le C' e^{-(1- 2\delta)M} + C'' e^{-\delta M}
\end{align*}
for all sufficiently large $M$, where we used Lemma~\ref{Yuleest}(ii) in the last step. 
The lemma follows.
\end{proof}
\section*{Acknowledgments}
We thank Charles Cooper for writing the Python program to simulate the model and Prabhanka Deka for enlightening discussions and comments related to the paper. The simulations in Fig.~\ref{fig:sim} were plotted using the excellent Graph tools~\cite{peixoto_graph-tool_2014}. { We thank two referees for closely going through the paper which led to a significant improvement in the exposition and organization of the paper. }

\section*{Funding}
{ Banerjee was partially supported by the NSF CAREER award DMS-2141621. Bhamidi and Sakanaveeti were partially supported by NSF DMS-2113662, DMS-2413928, and DMS-2434559. Banerjee and Bhamidi were partially funded by NSF RTG grant DMS-2134107. Part of
this material is based upon work supported by the National Science Foundation under Grant
No. DMS-1928930, while Banerjee and  Bhamidi were in residence at the Simons Laufer
Mathematical Sciences Institute in Berkeley, California, during the Spring 2025 semester.}

\bibliographystyle{plainnat}
\bibliography{macro_refs}

\end{document}